\documentclass[11pt]{article}
\usepackage[utf8]{inputenc}
\usepackage{amsmath,amsfonts,amssymb,amsthm} 
\usepackage{bm} 
\usepackage{bbm}
\usepackage{enumerate} 
\usepackage{indentfirst}
\usepackage{color} 
\usepackage{float}
\usepackage{graphics}
\usepackage{epsfig}
\usepackage{amsbsy}
\usepackage{mathrsfs}
\usepackage{mathtools}
\usepackage{comment}
\usepackage{tikz}
\usepackage{empheq}
\usepackage{hyperref}
\hypersetup{colorlinks=true, 
	linkcolor=blue 
}
\usepackage{authblk}
\AtBeginDocument{}

\allowdisplaybreaks
\newtheorem{theorem}{Theorem}[section]
\newtheorem{lemma}[theorem]{Lemma}
\newtheorem{proposition}[theorem]{Proposition}

\newtheorem{remark}[theorem]{Remark}
\newtheorem{definition}[theorem]{Definition}

\def\R{{\mathbb{R}}}

\def\CC{{\rm \kern.24em \vrule width.02em height1.4ex depth-.05ex
		\kern-.26emC}}
	
\def\TagOnRight

\def\AA{{it I} \hskip-3pt{\tt A}}

\def\QQ{\rlap {\raise 0.4ex \hbox{$\scriptscriptstyle |$}} {\hskip -0.1em Q}}

\catcode`\@=11

\def\theequation{\@arabic{\c@section}.\@arabic{\c@equation}}
\renewcommand{\div}{{\mathrm{div}}} 

\newcommand{\DT}{\mathbb{D}} 

\newcommand{\HC}[2]{\mathcal{C}^{{#1},{#2}}} 
\newcommand{\vn}{\bm{n}} 
\newcommand{\vt}{\bm{\tau}} 
\newcommand{\vu}{\bm{u}} 

\renewcommand{\Re}{\operatorname{Re}}

\begin{document}
\title{$L^p$-Strong solution to fluid-rigid body interaction system with Navier slip boundary condition}
\date{}

\author[1]{Hind Al Baba \thanks{hind.albaba@ul.edu.lb}}
\affil[1]{Laboratoire de Math\'{e}matiques et Applications, Lebanese University, Beirut, Lebanon.}
\author[2,3]{Amrita Ghosh \thanks{guddi.ghosh9@gmail.com}}
\affil[2]{Institute of Mathematics of the Academy of Sciences of the Czech Republic,\v Zitn\' a 25, 115 67 Praha,
Czech Republic.}
\affil[3] {Institute for Applied Mathematics, University of Bonn, Endenicher Allee 60, 53115 Bonn, Germany.}
\author[4]{Boris Muha \thanks{borism@math.hr}}
\affil[4]{Department of Mathematics, Faculty of Science, University of Zagreb, Croatia.}
\author[2]{\v{S}{\' a}rka Ne{\v c}asov{\' a} \thanks{matus@math.cas.cz}}


\maketitle







\begin{abstract}
We study a fluid-structure interaction problem describing movement of a rigid body inside a bounded domain filled by a viscous fluid. The fluid is modelled by the generalized incompressible Naiver-Stokes equations which include cases of Newtonian and non-Newtonian fluids. The fluid and the rigid body are coupled via the Navier slip boundary conditions and balance of forces at the fluid-rigid body interface. Our analysis also includes the case of the nonlinear slip condition.
The main results assert the existence of strong
solutions, in an $L^p-L^q$ setting, globally in time, for small data in the Newtonian case, while existence of strong solutions in $L^p$-spaces, locally in time, is obtained for non-Newtonian case. The proof for the Newtonian fluid essentially uses the maximal regularity property of the associated linear system which is obtained by proving the $\mathcal{R}$-sectoriality of the corresponding operator. The existence and regularity result for the general non-Newtonian fluid-solid system then relies upon the previous case.
Moreover, we also prove the exponential stability of the system in the Newtonian case.
\end{abstract}


\maketitle

\section{Introduction}
\setcounter{equation}{0}

We consider a fluid-rigid body interaction problem in $\R^3$ and our main interest is to develop an $L^p$-theory for strong solutions of the coupled system, for both Newtonian and non-Newtonian fluids with the moving rigid body.

Let us begin with a description of the model. We study a system of equations modelling the interaction between a fluid flow satisfying the incompressible generalized Navier-Stokes equations and a rigid body satisfying  the conservation of linear and angular momentum. The rigid body moves inside the fluid and at time $t\ge 0$, occupies a bounded domain $\Omega_S(t)$, while the fluid fills a bounded domain $\Omega_F(t)$. The common boundary of $\Omega_F(t)$ and $\Omega_S(t)$ is denoted by $\partial \Omega_S(t)$. Note that
$
\Omega_S(0)\cup \Omega_F(0) \cup \partial \Omega_S(0) = \Omega_S(t)\cup \Omega_F(t) \cup \partial \Omega_S(t)=: \Omega \subsetneq \R^3; \quad t\ge 0
$
and
$
\partial \Omega_F(t) = \partial \Omega \cup \partial \Omega_S(t).
$
\begin{figure}
\begin{center}
	\begin{tikzpicture}
	\draw[thick,red,rotate around={45:(9,-.6)}] plot [smooth cycle] coordinates {(8,0) (9,1) (11,1) (9,0) (10,-1)};
	\draw [thick] plot [smooth cycle] coordinates {(7,0) (8,2) (9,2.5) (9.75, 1.75) (11,2) (12,1) (12,-1) (11,-2) (8,-2) (7.5,-1.5) (7,-1.75) (6.75,-1.25)};
	\draw [thick,red] plot [smooth cycle] coordinates {(-0.5,-.7) (.5,.3) (2.5,.3) (.5,-.7) (1.5,-1.7)};
	\draw [thick] plot [smooth cycle] coordinates {(-1,0) (0,2) (1,2.5) (1.75, 1.75) (3,2) (4,1) (4,-1) (3,-2) (0,-2) (-0.5,-1.5) (-1,-1.75) (-1.25,-1.25)};
	\node at (3,-1) {$\Omega_F(0)$};
	\node at (0.8,0.65) {$\partial\Omega_S(0)$};
	\node at (3.2,2.2) {$\partial\Omega$};
	\node at (.1,-0.6) {$\Omega_S(0)$};
	\node at (11,-0.3) {$\Omega_F(t)$};
	\node at (8.25,0.2) {$\Omega_S(t)$};
	\node at (11.2,2.2) {$\partial\Omega$};
	\node at (8.9,-1) {$\partial\Omega_S(t)$};
	\draw[thick,->] (5,0) -- (6,0);
	\end{tikzpicture}
\end{center}
\caption{Domain}
\label{fig:M1}
\end{figure}

For the sake of simplicity, we assume the fluid has a constant density equals $1$. By choosing a frame of coordinates whose origin initially coincides with the centre of mass of the rigid body, the domain $\Omega_S(t)$ at any instant $t$ can be given by
$$
\Omega_S(t) = \{\bm{h}(t) + Q(t)\bm{y}: \bm{y}\in \Omega_S(0)\}
$$
where $\bm{h}(t)$ is the centre of mass of the rigid body at time $t$ and $Q(t)$ is a rotation matrix associated to the angular velocity $\bm{\omega}(t)$ of the rigid body. The matrix $Q(t)$ is the solution of the initial value problem
\begin{equation}
\label{Q}
\begin{aligned}
\dot{Q}(t)Q^T(t) \bm{y} & = \bm{\omega}(t)\times \bm{y} \quad \forall \bm{y}\in\R^3 \\
Q(0) & = I_3 .
\end{aligned}
\end{equation}
Here $A^T$ denotes the transpose matrix of $A$ and $I_3$ is the $3\times 3$ identity matrix.
The system of equations modelling the motion of the fluid and the rigid body is given by:

\begin{equation}
\label{GFSI}
\left\{
\begin{aligned}
\partial_t \vu + (\vu \cdot \nabla_x) \vu &= \div \ \mathbb{T} (\vu,\pi) \ &&\text{ in } \ \Omega_F(t)\times (0,T),\\
\div \ \vu &= 0 \ &&\text{ in } \ \Omega_F(t)\times (0,T),\\
\vu &=\bm{0} \ &&\text{ on } \ \partial \Omega \times (0,T),\\
\vu\cdot \vn &=\vu_{S}\cdot \vn \ &&\text{ on } \ \partial \Omega_S(t) \times (0,T),\\
\left[\mathbb{T} (\vu,\pi)\vn\right]_{\vt}+\alpha\vu_{\vt}&=\alpha \vu_{S\vt} \ &&\text{ on } \ \partial \Omega_S(t) \times (0,T),\\
m \bm{l}^{'}(t) &= - \int\displaylimits_{\partial \Omega_S(t)}{\mathbb{T} (\vu,\pi)\vn}, \ && \ t\in(0,T),\\
(J \bm{\omega})'(t) &= - \int\displaylimits_{\partial \Omega_S(t)}{(\bm{x}-\bm{h}(t))\times \mathbb{T} (\vu,\pi)\vn},  && \ t\in(0,T),\\
\vu(0) &= \vu_0 \ &&\text{ in } \ \Omega_F(0),\\
\bm{l}( 0) = \bm{l}_0, &\quad \bm{\omega}(0) = \bm{\omega}_0
\end{aligned}
\right.
\end{equation}
where $\vu $ and $\pi$ denote the velocity field and pressure of the fluid respectively,
$\mathbb{T} (\vu,\pi) := \mu(|\DT\vu|^2)\DT \vu - \pi I_3$
{ is the Cauchy stress tensor} with the viscosity function $\mu\in C^{1,1}(\R^+;\R)$ satisfying the following assumptions
\begin{equation}
\label{viscosity}
\mu(s)>0 \quad \text{ and } \quad \mu(s)+ 2s\mu'(s)>0 \quad \text{ for all } \ s\ge 0
\end{equation}
and $\DT\vu := \frac{1}{2} \left( \nabla \vu + \nabla \vu^T\right)$ \textit{i.e.}, $(\DT\vu)_{ij}: = \DT_{ij}\vu := \frac{1}{2}\left(\partial_i u_j + \partial_j u_i \right)  $ denotes the symmetrical stress tensor with $|\DT\vu|^2 = \sum_{i,j=1}^3 (\DT_{ij}\vu)^2$ is the Hilbert-Schmidt norm. The friction coefficient $\alpha(\bm{x}) \ge 0 $ is a given function and $\vn(\bm{x},t)$ denotes the unit outward normal vector with respect to the domain $\Omega_F(t)$. The subscript $(\cdot)_{\vt}$ denotes the tangential component of a vector \textit{i.e.} $\bm{v}_{\vt} = \bm{v} - (\bm{v}\cdot \vn) \vn$. The constant $m>0$ is the mass of the rigid body and $J(t)$ is its inertia tensor, given by
$$
J(t)a\cdot b= \int\displaylimits_{ \Omega_S(0)}{\rho_S (a\times (\bm{x}-\bm{h}(t)))\cdot (b\times (\bm{x}-\bm{h}(t)))} \qquad \forall\, a,b\in\R^3
$$
where $\rho_S>0$ is the density of the body. Lastly, $\bm{l}(t):= \bm{h}'(t)$ denotes the translational velocity such that $\vu_S(\bm{x},t) := \bm{l}(t) + (\bm{x}- \bm{h}(t))\times \bm{\omega}(t)$ is the velocity of the rigid body.

The assumptions on the stress tensor $T$ allow a wide flexibility of stress law coming from various experimentally verified physical models. In particular, it includes the power-law type fluids, namely,
\begin{equation*}
\mu(|\DT\vu|^2) = \mu_0\left( 1+ |\DT\vu|^2\right)^{\frac{d-2}{2}}, \quad \mu(|\DT\vu|^2) = \mu_0|\DT\vu|^{d-2} \quad \text{ for }
\ \mu_0\in (0,\infty), \ d\in (1,\infty).
\end{equation*}

\begin{remark}
(1) The case $d=2$ corresponds to the classical Newtonian fluids \textit{i.e.} the case of constant viscosity.
 In that case, we denote the stress tensor by $\sigma(\vu,\pi)$ (just to distinguish) which is simply given by $\sigma(\bm{\vu},\pi) :=  \mu_0 \DT\vu - \pi I$.\\
 \\
(2) Of particular importance within this class are, the {\it shear-thinning fluids} \textit{i.e.}, the case $d\in (1,2)$ which include many important materials of interest (e.g. can be applied for modelling of  blood) and also {\it shear-thickening fluids}, \textit{i.e.}, the case $d\in (2,\infty)$. For further discussions on the related non-Newtonian fluids, we refer e.g. \cite {BB, BBN,Necas}.
\end{remark}

Concerning the study of the fluid-rigid body interaction system involving non-Newtonian fluid, there are not many works available. The authors in \cite{GGH} provide the $L^p$-theory for local, strong solutions in the case of the Dirichlet coupling condition (no-slip) at the fluid-solid interface and for $\Omega = \mathbb{R}^3$. Furthermore in \cite{FHN}, the authors have considered a similar system with only power-law type fluid and with the Dirichlet boundary condition ($\alpha=\infty$) and establish the existence of a global in time, weak solution.
{ For the existence of a weak solution of the Bingham type fluid with motion of a rigid body, see \cite{OT}. }
{
The global existence of weak solution in the case of the motion of several rigid bodies in an incompressible non-Newtonian fluid with growth conditions, in Orlicz spaces, has been shown in \cite{Aneta}, for Dirichlet boundary condition.}

Another novelty of our paper is that we consider here a non-standard boundary condition at the fluid-solid interface. We investigate here the so-called slip boundary conditions.\footnote{The slip boundary conditions were introduced by Navier \cite{Navier} (the linear version), later proposed independently by Maxwell \cite{max}.} These conditions state that, apart from being the physical domain impermeable, it allows the fluid to slip over the solid boundary, rather than sticking to it. More precisely, the tangential component of the fluid velocity at the interface need not be equal to the tangential component of the velocity of the solid boundary, and this difference is proportional to the stress exerted by the fluid on the boundary, and the proportionality constant is called the \textit{friction coefficient} (the inverse of the \textit{slip length}). Observe that the friction coefficient $\alpha =\infty$ gives the no-slip Dirichlet boundary condition, formally, while $\alpha =0$ corresponds to the full slip conditions. Although the no-slip boundary condition is the one widely studied and accepted in the context of fluid dynamics, there are many problems at the macro scale where the no-slip condition is not applicable, examples include the moving contact line problem \cite{dussan} and the corner flows \cite{KB}. These and many other paradoxes may possibly appear because of the no-slip boundary conditions. Moreover, in the context of fluid-rigid body interaction, the no-slip condition gives rise to the well-known \textit{no-contact paradox} which says that collision between the rigid body and the boundary will not occur in finite time \cite{Hillairet07, Starovoitov04}. Therefore it is necessary to study more deeply the slip boundary condition experimentally, as well as mathematically.
There are very few works done on the fluid-solid interaction system where slip boundary conditions are treated. Existence of a weak solution for the Newtonian fluid was proved in \cite{ChemetovSarka,GVHill}, while the existence and uniqueness of local-in-time strong solutions were studied in \cite{HindSarkaBoris,Wang}. In \cite{Bravin} the author proved uniqueness of weak solution in the $2D$ case. Weak-strong uniqueness in $3D$ was studied in \cite{ChemetpvSarkaBoris,muha2019uniqueness}. Finally, we mention the existence result for weak solution in the case when the elastic structure is part of the fluid boundary \cite{SunnyBoris} which is a $3D-2D$ interaction problem.
{ In the present work, we focus on the case with the slip condition on the boundary of the rigid body $\partial \Omega _S$ and the Dirichlet boundary condition on the outer boundary $\partial \Omega$. Such type of problem is called as {\it mixed case} and firstly was investigated in \cite{VHW}. Further several results for the {\it mixed case} can be found in e.g. \cite{ChemetovSarka,HindSarkaBoris,ChemetpvSarkaBoris}}.
All the above-mentioned results are in $L^2$-setting, for the Newtonian fluids and the friction coefficient $\alpha$ is assumed either $0$ or a constant.

Further, based on molecular dynamics simulations, Thompson and Trojan proposed a slip boundary condition which is highly non-linear \cite{TT1}, even though the fluid is still considered to be Newtonian (see also \cite{MH}). This universal condition may determine the degree of slip at a fluid-solid interface as the interfacial parameters and the shear rate are varied. But this non-linear boundary condition seems out of reach with our present mathematical technique. However, we may treat the non-linearity discussed by Lewandowski et al. \cite{Lewandowski} in the context of turbulence model
\begin{equation}
\label{nonlin_slip}
[\sigma(\vu,\pi)\vn]_{\vt} + \alpha |\vu| \vu_{\vt} =\bm{0},
\end{equation}
which is not much different from its linear counterpart, at least concerning the qualitative analysis.

Our main goal in this work is to develop an $L^p$-theory for strong solutions of the fluid-rigid body interaction system with slip boundary condition at the interface, for both Newtonian and non-Newtonian fluid.
The main novelty of this work is to provide a unified result considering the linear as well as some non-linear slip condition at the fluid-solid interface, where the non-constant slip coefficient depends on the space variable.

We start with studying the Newtonian case. Since the domain $\Omega_F(t)\times (0,t)$ depends on the motion of the rigid body, this is a moving boundary problem where the domain is also an unknown a priori. Hence it is natural to transform the system (\ref{GFSI}) to a fixed domain and solve the problem there.
Among several possibilities for this transformation, the usual one is a global, linear transformation which says the whole space is rigidly rotated and shifted back to its original position at each time $t>0$ (cf. \cite{galdi}). This refers the equations of motion of the fluid-rigid body system in a frame attached to the rigid body, with origin in the center of mass of the latter and coinciding with an inertial frame at time $t=0$. One conceptual difference is that in \cite{galdi}, the fluid occupies an exterior domain where it is reasonable to perform such a transformation. But in our case where the fluid and solid contains a bounded domain,
it is not suitable to choose such transformation. Also technically, this transformation generates an extra drift term of the form $[(\bm{\omega}\times \bm{y})\cdot \nabla]\vu$ which has unbounded coefficients. This produces a fundamental difficulty as the transformed problem is no longer parabolic. To overcome this difficulty, Tucsnak, Cumsille and Takahashi (cf. \cite{takeo, TT, cumsille}) used another non-linear, local change of variables which only act in a bounded neighbourhood of the body \textit{i.e.} coincides with $Q(t)\bm{y} + \bm{h}(t)$ in a neighbourhood of the rigid body and is equal to the identity far from the rigid body. This transformation preserves the solenoidal condition of the fluid velocity and do not change the regularity of the solutions, although the rigid body equations change and become non-linear.

We follow this second approach with a different point of view. We use the rotation matrix $Q$ instead of the Jacobian matrix $J_Y$ in the change of variables (cf. \ref{13}). In that sense it lies somewhat in between the above two methods. Although this transformation does not preserve the divergence free condition as compared to the second approach involving $J_Y$ used in \cite{takeo}, it makes the computation and corresponding estimates on the non-linear terms, appeared from the change of variables, much easier, since all the derivatives of $Q$ with respect to $x$ can be avoided. After the change of variables, our strategy is based on the maximal regularity property of the linearised system. We extend the maximal regularity result for the Stokes problem to the coupled fluid-rigid body system. At this step, we write the full system in terms of $\mathbb{P}\vu, \bm{l}, \bm{\omega}$ only, where $\mathbb{P}$ is the Helmholtz projection. This helps us to achieve further the exponential stability of the system in the Newtonian case. We finally rewrite the full non-linear transformed problem as a fixed point problem and deduce several estimates on the coordinate transform and on the extra terms appearing from the transformed system which make the fixed point mapping contractive, provided the given data is small. This gives the existence of a unique strong solution in $L^p$ spaces of the fluid-rigid body system.

{ Also note that the approach in \cite{GGH} is different from ours for the Newtonian case. The crucial step is both approaches the maximal regularity result for the linearised problem. In \cite{GGH}, first the solution operator for the fluid problem is constructed (for the velocity and the pressure). In order to construct the solution operator, boundary condition are homogenized via Bogovski\u{i} operator. Then the problem is reformulated in terms of the rigid body variables, solved locally in time and iterated to obtain the global in time solution. Inspired by \cite{MT}, we directly study the resolvent problem for the linearised coupled fluid-structure problem. After eliminating the pressure, we use perturbation argument to show that the corresponding fluid-structure operator is $\mathcal{R}$-bounded. This approach is reminiscent to the monolithic approach in numerical analysis of the FSI problems, where the fluid and the structure problem are solved simultaneously. In both of these approaches, $\mathcal{R}$-boundedness of the Stokes operator is used. An advantage of this approach is, it allows us to directly obtain maximal $L^p-L^q$ regularity for the linear coupled problem for $t>0$. As a consequence, one further obtains global existence and uniqueness result for the full non-linear system under smallness assumption of given data together with the exponential decay of the solution in the suitable norms.

}

Next we discuss the non-Newtonian case. Using the same transformation as before, we reduce the system on fixed domain and then linearize the corresponding operator by fixing the coefficients. To prove the maximal regularity of the linearized system, the same approach as in the Newtonian case can not be followed. Because of the complicated structure of the generalized operator, writing the full system in terms of only $\mathbb{P}\vu, \bm{l}, \bm{\omega}$ seems not possible. Thus we follow here the approach used in \cite{GGH} which relies on classical maximal regularity estimates for the fluid problem and suitable pressure estimates near the boundary. In the last subsection, we show that the same analysis can be done as well for the non-linear slip condition (\ref{nonlin_slip}).

Note that we can also consider the Navier slip condition $\vu\cdot \vn=\bm{0}, [\mathbb{T}(\vu, \pi)\vn]_{\vt} + \alpha \vu_{\vt}=0$ on the outer boundary $\partial\Omega$ instead of $(\ref{GFSI})_3$ and study the system. This makes no big difference in the analysis. The change of variables we used here (cf. \ref{13}), do not preserve the Navier condition on $\partial\Omega$, although it conserves the slip condition at the fluid-solid interface. Thus the case with slip condition at the outer fluid boundary can be treated with the usual change of variables (cf. \cite{GGH}).

\begin{remark}
Let us mention that recently the $L^p - L^q$ regularity of the non-Newtonian compressible fluid was investigated in the whole space see \cite{KMN} and maybe it  can give us a way to get the $L^p - L^q$ regularity of the non-Newtonian incompressible fluid.
\end{remark}

\section{Main results}
\setcounter{equation}{0}

We assume that the rigid body at the initial position does not touch the wall of the fluid domain, i.e.
\begin{equation*}
\text{dist}\left( \Omega_S(0), \partial \Omega\right) \ge \beta >0.
\end{equation*}
For reference purpose, we rewrite the generalized system (\ref{GFSI}) in case of Newtonian fluid:

\begin{equation}
\label{1}
\left\{
\begin{aligned}
\partial_t \vu + (\vu \cdot \nabla_x) \vu &= \div \ \sigma (\vu,\pi) \ &&\text{ in } \ \Omega_F(t)\times (0,T),\\
\div \ \vu &= 0 \ &&\text{ in } \ \Omega_F(t)\times (0,T),\\
\vu&=\bm{0} \ &&\text{ on } \ \partial \Omega \times (0,T),\\
\vu\cdot \vn &=\vu_{S}\cdot \vn \ &&\text{ on } \ \partial \Omega_S(t) \times (0,T),\\ 2\left[(\DT\vu)\vn\right]_{\vt}+\alpha\vu_{\vt}&=\alpha \vu_{S\vt} \ &&\text{ on } \ \partial \Omega_S(t) \times (0,T),\\
m \bm{l}^{'}(t) &= - \int\displaylimits_{\partial \Omega_S(t)}{\sigma (\vu,\pi)\vn}, \ && \ t\in(0,T),\\
(J \bm{\omega})'(t) &= - \int\displaylimits_{\partial \Omega_S(t)}{(\bm{x}-\bm{h}(t))\times \sigma (\vu,\pi)\vn}, \ && \ t\in(0,T),\\
\vu(0) &= \vu_0 \ &&\text{ in } \ \Omega_F(0),\\
\bm{l}( 0) = \bm{l}_0, &\quad \bm{\omega}(0) = \bm{\omega}_0,
\end{aligned}
\right.
\end{equation}
where $\sigma(\bm{\vu},\pi) :=  \mu \DT\vu - \pi I$ with $\mu>0$ being a constant is the stress tensor.

We use the following function spaces. For a domain $D\in \R^n$, the Sobolev spaces are denoted by $W^{m,q}(D)$ for $m>0$ an integer and $q\in [1,\infty]$. The fractional Sobolev space $W^{s,q}(D)$ for non-integer $s$ is defined as, for $q\in [1,\infty), s\in (0,1)$,
\begin{equation*}
W^{s,q}(D) := \left\lbrace v\in L^p(D): \frac{|v(x) - v(y)|}{|x-y|^{s+ n/q}}\in L^p(D\times D) \right\rbrace,
\end{equation*}
with the standard norm. For $s=m+\sigma, m\ge0$ an integer and $\sigma\in (0,1)$, it is then defined as (cf. \cite{valdi}),
\begin{equation*}
W^{s,q}(D) :=\left\lbrace v\in W^{m,q}(D): D^\alpha v\in W^{\sigma,q}(D) \ \text{ for all } |\alpha| = m\right\rbrace .
\end{equation*}

Let $ q\in(1,\infty)$, $m_1, m_2\in \mathbb{N}\cup \{0\}$ with $m_1\ne m_2$. Then for $p \in [1,\infty], \theta \in (0,1)$ and $s= m_1(1-\theta) +m_2 \theta$,
we denote the Besov spaces by $B^s_{q,p}(D)$ which can be defined (equivalently) by real interpolation of Sobolev spaces (cf. \cite[Section 1.6.4, page 39]{triebel})
\begin{equation}
\label{besov}
B^s_{q,p}(D): = \left( W^{m_1,q}(D), W^{m_2,q}(D) \right)_{\theta,p}.
\end{equation}
The Besov spaces in the case $q=p$ coincide with the classical Sobolev spaces $W^{s,q}(D)$.

We also introduce the notations: the subscript $\sigma$ denotes the divergence free condition in the domain and the subscript $\vt$ over a space denotes the zero normal component on the boundary. For example, we write:
$$
\bm{L}^q_{\sigma,\vt}(D):=\{\bm{v}\in\bm{L}^q(D): \div \ \bm{v}=0 \text{ in } D, \bm{v}\cdot \vn=0 \text{ on } \partial D\}.
$$
Let the Stokes operator with Navier boundary conditions on $\bm{L}^q_{\sigma,\vt}(\Omega_F(0))$ be defined as,
\begin{equation}
\label{SO}
\begin{cases}
\begin{aligned}
&\mathcal{D}(A_q) := \{\vu\in \bm{W}^{2,q}_{\sigma,\vt}(\Omega_F(0)): \vu = \bm{0} \text{ on } \partial\Omega, 2[(\DT\vu)\vn]_{\vt}+\alpha \vu_{\vt}=\bm{0} \text{ on } \partial \Omega_S(0)\},\\
&A_q \vu= \mathbb{P}\Delta \vu \quad \text{ for all } \vu\in \mathcal{D}(A_q)
\end{aligned}
\end{cases}
\end{equation}
where $\alpha$ is such that
\begin{equation}
\label{alpha}
\alpha\in\
\begin{cases}
W^{1-\frac{1}{\frac{3}{2}+\varepsilon}, \frac{3}{2}+\varepsilon}(\partial\Omega_F(0)) &\text{ if } \ 1<p\le \frac{3}{2}\\
W^{1-\frac{1}{p}, p}(\partial\Omega_F(0)) &\text{ if } \ p>\frac{3}{2}
\end{cases}
\end{equation}
with $\varepsilon>0$ arbitrarily small and $\mathbb{P}$ is the Helmholtz projection
$$
\mathbb{P}:\mathbf{L}^q(\Omega_F(0))\to \mathbf{L}^q_{\sigma,\vt}(\Omega_F(0))
$$
\textit{i.e.} for $\bm{\varphi} \in \mathbf{L}^q(\Omega_F(0))$, $\mathbb{P}\bm{\varphi} = \bm{\varphi} - \nabla p$ for some $p\in W^{1,q}(\Omega_F(0))$ which satisfies
\begin{equation*}
\begin{cases}
\div (\nabla p - \bm{\varphi}) = 0 & \text{ in } \Omega_F(0)\\
(\nabla p - \bm{\varphi})\cdot \vn = 0 & \text{ on } \partial\Omega_F(0).
\end{cases}
\end{equation*}
Also we say that $\vu_0\in B^{2-2/p}_{q,p}(\Omega_F(0))$ satisfies the compatibility condition if
$$
\vu_0 - \bm{v}_0\in \left( \bm{L}^q_{\sigma,\vt}(\Omega_F(0)), \mathcal{D}(A_q)\right) _{1-\frac{1}{p},p} \quad \text{ for some } \bm{v}_0\in C^2(\Omega_F(0)) \text{ satisfying }
$$
\begin{equation}
\label{CN}
\begin{aligned}
& \qquad \qquad \qquad \qquad \div \ \bm{v}_0 = 0 \ \text{ in } \Omega_F(0),\\
& \bm{v}_0\cdot \vn = \left( \bm{l}_0 + (\bm{\omega}_0\times y)\right) \cdot \vn \ \text{ on } \ \partial \Omega_S(0), \quad \bm{v}_0 = \bm{0} \ \text{ on } \ \partial \Omega,\\
\text{and} \text{ if } p>3, \qquad
	&2[(\DT\bm{v}_0)\vn]_{\vt} + \alpha \bm{v}_{0\vt} = \alpha (\bm{l}_0 + \bm{\omega}_0\times \bm{y})_{\vt} \ \text{ on } \partial\Omega_S(0).
\end{aligned}
\end{equation}
Lastly we introduce the following notation for the time-dependent spaces which will be used later,
$$
W^{2,1}_{q,p}(Q^\infty_F) := L^p(0,\infty;\bm{W}^{2,q}(\Omega_F(0))) \cap W^{1,p}(0,\infty;\bm{L}^q(\Omega_F(0)))
$$
with the graph norm
$$
\|\bm{v}\|_{W^{2,1}_{q,p}(Q^\infty_F)} := \|\bm{v}\|_{L^p(0,\infty;\bm{W}^{2,q}(\Omega_F(0)))} + \|\bm{v}\|_{W^{1,p}(0,\infty;\bm{L}^q(\Omega_F(0)))}.
$$

We can now state our main results on the existence of a unique, global in time, strong solution for the Newtonian system (\ref{1}), under the smallness assumption on initial data and a unique, local in time, strong solution for the generalized Newtonian system (\ref{GFSI}) for any initial data.

\begin{theorem}
	\label{T_N}
	Let $\Omega_F(0)$ be a bounded domain of class $\HC{2}{1}$, $p,q\in (1,\infty)$ satisfy the condition $\frac{1}{p}+ \frac{3}{2q}\le \frac{3}{2}$ and $\alpha\ge 0$ be as in (\ref{alpha}). Let $\eta\in (0,\eta_0)$ where $\eta_0$ is some constant (specified in Theorem \ref{T1}). Then there exist two constants $\delta_0>0$ and $C>0$, depending only on $p,q,\eta$ and $\Omega_F(0)$ such that for all $\delta\in (0,\delta_0)$ and for all $(\vu_0,\bm{l}_0,\bm{\omega}_0)\in B^{2(1-1/p)}_{q,p}(\Omega_F(0))\times \R^3 \times \R^3$ satisfying the compatibility conditions (\ref{CN})
	and
	$$
	\|\vu_0\|_{B^{2(1-1/p)}_{q,p}(\Omega_F(0))}+\|\bm{l}_0\|_{\R^3}+\|\bm{\omega}_0\|_{\R^3}\le \delta,
	$$
	the system (\ref{1}) possesses a unique global strong solution $(\vu,\pi,\bm{l},\bm{\omega})$ in the class of functions satisfying
	\begin{equation}
	\label{30}
	\begin{aligned}
	&\quad \|e^{\eta(\cdot)}\vu\|_{L^p(0,\infty;\bm{W}^{2,q}(\Omega_F(\cdot)))} + \|e^{\eta(\cdot)}\vu\|_{W^{1,p}(0,\infty;\bm{L}^q(\Omega_F(\cdot)))} + \|e^{\eta(\cdot)}\vu\|_{L^\infty(0,\infty;B^{2(1-1/p)}_{q,p}(\Omega_F(\cdot)))} \\
	&+ \|e^{\eta(\cdot)}\pi\|_{L^p(0,\infty;W^{1,q}(\Omega_F(\cdot)))} + \|e^{\eta(\cdot)}\bm{l}\|_{W^{1,p}(0,\infty;\R^3)}+ \|e^{\eta(\cdot)}\bm{\omega}\|_{W^{1,p}(0,\infty;\R^3)} \le C\delta.
	\end{aligned}
	\end{equation}
	Moreover, $\mathrm{dist}(\Omega_S(t),\partial\Omega)\ge \beta/2$ for all $t\in [0,\infty)$.
	
	In particular, we have,
	$$
	\|\vu(\cdot, t)\|_{B_{q,p}^{2(1-1/p)}(\Omega_F(t))} + \|\bm{l}(t)\|_{\R^3} + \|\bm{\omega}(t)\|_{\R^3} \le C \delta e^{-\eta t} .
	$$
\end{theorem}

\begin{remark}
	If we consider the slip condition on $\partial\Omega\times (0,T)$ in $(\ref{1})_3$ instead of the Dirichlet condition, then we obtain the same result as above, provided the compatibility condition for the initial data needs to be replaced with
	$$
	\vu_0 - \bm{v}_0\in \left( \bm{L}^q_{\sigma,\vt}(\Omega_F(0)), \mathcal{D}(A^{NS}_q)\right) _{1-\frac{1}{p},p} \quad \text{ for some } \bm{v}_0\in C^2(\Omega_F(0)) \text{ satisfying }
	$$
	\begin{equation*}
	\label{}
	\begin{aligned}
	& \qquad \qquad \qquad \qquad \div \ \bm{v}_0 = 0 \ \text{ in } \Omega_F(0),\\
	& \bm{v}_0\cdot \vn = \left( \bm{l}_0 + (\bm{\omega}_0\times y)\right) \cdot \vn \ \text{ on } \ \partial \Omega_S(0), \quad \bm{v}_0\cdot \vn = 0 \ \text{ on } \ \partial \Omega,\\
\text{and if } p>3, \quad &
\begin{cases}
2[(\DT\bm{v}_0)\vn]_{\vt} + \alpha \bm{v}_{0\vt} &= \bm{0} \ \text{ on } \partial\Omega,\\
2[(\DT\bm{v}_0)\vn]_{\vt} + \alpha \bm{v}_{0\vt} &= \alpha (\bm{l}_0 + \bm{\omega}_0\times \bm{y})_{\vt} \ \text{ on } \partial\Omega_S(0).
\end{cases}
	\end{aligned}
	\end{equation*}
	Here $A^{NS}_q: \bm{L}^q_{\sigma,\vt}(\Omega_F(0)) \to \bm{L}^q_{\sigma,\vt}(\Omega_F(0))$ denotes the following Stokes operator with Navier slip boundary condition on the  complete fluid boundary:
	\begin{equation*}
	\begin{cases}
	\begin{aligned}
	&\mathcal{D}(A^D_q) := \{\vu\in \bm{W}^{2,q}_{\sigma,\vt}(\Omega_F(0)); 2[(\DT\vu)\vn]_{\vt}+\alpha \vu_{\vt}=\bm{0} \text{ on } \partial \Omega_F(0)\},\\
	&A_q \vu= \mathbb{P}\Delta \vu \quad \text{ for all } \vu\in \mathcal{D}(A^{NS}_q).
	\end{aligned}
	\end{cases}
	\end{equation*}
\end{remark}

\begin{theorem}
	\label{T_G}
	Let $p>5$, $\Omega_F(0)$ be a bounded domain of class $\HC{2}{1}$ and $\alpha\ge 0$ satisfies (\ref{alpha}). Then 
	for all $(\vu_0, \bm{l}_0, \bm{\omega}_0)\in W^{2-2/p,p}(\Omega_F(0))\times \R^3\times \R^3$ satisfying the compatibility conditions (\ref{CN}),
	the problem (\ref{GFSI}) admits a local in time unique strong solution
	\begin{equation*}
	\begin{aligned}
	& \vu \in L^p(0,T;\bm{W}^{2,p}(\Omega_F(\cdot))) \cap W^{1,p}(0,T;\bm{L}^p(\Omega_F(\cdot))),\\
	& \pi \in L^p(0,T;W^{1,p}(\Omega_F(\cdot))), \bm{l} \in W^{1,p}(0,T;\R^3), \bm{\omega} \in W^{1,p}(0,T;\R^3).
	\end{aligned}
	\end{equation*}
\end{theorem}

\begin{remark}
Since the generalized stress tensor $\mathbb{T}$ includes the Newtonian stress tensor $\sigma$ as the special case, Theorem \ref{T_G} generalizes Theorem \ref{T_N}. On the other hand, less restrictive assumptions on $p,q$ are needed in Theorem \ref{T_N}, compared to Theorem \ref{T_G}.
\end{remark}

Our last result concerns the nonlinear slip condition (\ref{nonlin_slip}).

\begin{theorem}
\label{T_nonlin_slip}
Let $p>5$, $\Omega_F(0)$ be a bounded domain of class $\HC{2}{1}$ and $\alpha\ge 0$ satisfies (\ref{alpha}). Then 
for all $(\vu_0, \bm{l}_0, \bm{\omega}_0)\in W^{2-2/p,p}(\Omega_F(0))\times \R^3\times \R^3$ satisfying the compatibility conditions (\ref{CN}),
the problem (\ref{GFSI}) with the boundary conditions $(\ref{GFSI})_5$ replaced by the nonlinear slip condition (\ref{nonlin_slip}) admits a local in time unique strong solution
	\begin{equation*}
	\begin{aligned}
	& \vu \in L^p(0,T;\bm{W}^{2,p}(\Omega_F(\cdot))) \cap W^{1,p}(0,T;\bm{L}^p(\Omega_F(\cdot))),\\
	& \pi \in L^p(0,T;W^{1,p}(\Omega_F(\cdot))), \ \bm{l} \in W^{1,p}(0,T;\R^3), \ \bm{\omega} \in W^{1,p}(0,T;\R^3).
	\end{aligned}
	\end{equation*}
\end{theorem}

\section{Preliminaries and notations}
\setcounter{equation}{0}

In this section, we introduce the notation used throughout this paper, in particular some results concerning maximal regularity and $\mathcal{R}$-boundedness in Banach spaces.

The concept of the maximal regularity is classical and the main achievement for abstract theory goes back to  the  work of Ladyzhenskaya, Uraltzeva, Solonnikov \cite{LUS}, Da-Prato, Grisvard \cite{DG}, Amann \cite{Amann1} and Pr{\" u}ss \cite{P}.  The description of maximal $L^p$-regularity  by $\cal R$-boundedness made the possibility to prove the maximal regularity for a large class of parabolic boundary valued problems. The notion of  $\cal R$-sectoriality was defined  by Clement and Pr\"uss \cite{CP}. The fundamental result by Weis \cite{weis} about equivalency   of description for the maximal regularity in terms of vector-valued Fourier multipliers and ${\cal R}$-sectoriality was a breakthrough for the applications.

Due to the work of Denk, Hieber, Pr\"uss  \cite{DHP}, the $\cal R$-sectoriality is nowadays one of the basic tool in this direction, together with the work by Kunstmann, Weis \cite{KW}.
The maximal $L^p$-regularity result for incompressible fluids can be found in work by Shibata, Shimizu \cite{SHS}, for the  non-Newtonian incompressible situation, it was established in \cite{BP} by Bothe, Pr\"uss.  

For Banach spaces $X$ and $Y$, we denote the space of all bounded, linear operators from $X$ to $Y$ by $\mathcal{L}(X,Y)$. The resolvent set of a linear operator $A$ is denoted by $\rho(A)$ and the domain of an operator $A$ is denoted by $\mathcal{D}(A)$. Whenever we consider $\mathcal{D}(A)$ as a Banach space, it is assumed to be equipped with the graph norm of $A$.

We consider the following problem:
\begin{equation}
\label{cauchy}
\left\{
\begin{aligned}
u'(t) + Au(t) &= f(t), \ t\ge 0\\
u(0) &= u_0 .
\end{aligned}
\right.
\end{equation}

\begin{definition}
	Let $p\in (1,\infty)$. We say that $A$ has the maximal $L^p$-regularity property on the interval $I = [0,T)$ with  $T \le \infty$ if there exists a constant $C>0$ such that for all $f\in L^p(I;X)$, there is a unique $u\in L^p(I;\mathcal{D}(A))$ with $u'\in L^p(I;X)$ satisfying (\ref{cauchy}) with $u_0= 0$ for almost every $t\in I$ and
	$$
	\|u'\|_{L^p(I;X)} + \|Au\|_{L^p(I;X)} \le C \|f\|_{{L^p(I;X)}} .
	$$
\end{definition}
Note that in the above definition (cf. \cite{kunstmann_weis}), we do not have that $u\in L^p(I;X)$ which is slightly weaker than the usual one (e.g. in \cite{dore}, \cite{dore_venni}).

We also need to define the notion of UMD-space (unconditional difference martingale property). Actually we give here a property of UMD-spaces which is equivalent to the original definition (for more on this subject, see \cite{burkholder}, \cite{bourgain}). 

\begin{definition}
A complex Banach space is said to be of class UMD if the Hilbert transform $\mathcal{H}$ is bounded in $L^p(\R;X)$ for all (or equivalently, for one) $p\in (1,\infty)$.
\end{definition}

\begin{remark}
    \begin{itemize}

\item Any Hilbert space is in the class UMD.
\item If $X$ is a Banach space in the UMD-class, then $L^p(\Omega;X)$ for $\Omega\subset\R^n$ and $p\in (1,\infty)$ is also in the UMD-class.
\item Let us mention here that maximal regularity for non-UMD spaces (for example, non-reflexive Banach spaces $L^1$ and $L^{\infty}$) requires different argument and we may refer to works e.g. Ogawa, Shimizu \cite{OG}.
\end{itemize}
\end{remark}

Now we want to state an equivalent property to maximal regularity in terms of $\mathcal{R}$-boundedness of the resolvent of the operator. For further details on $\mathcal{R}$-boundedness, cf.\cite{weis}.

\begin{definition}
	A set $\mathcal{T}\subset\mathcal{L}(X,Y)$ is called $\mathcal{R}$-bounded if there is a constant $C>0$ such that for all $n\in\mathbb{N}$, $T_1, ... , T_n\in\mathcal{T}$ and $x_1, ..., x_n\in X$,
	$$
	\int\displaylimits_0^1{\Big\Vert\sum_{j=1}^n r_j(s)T_jx_j\Big\Vert_Y\mathrm{d}s}\le C \int\displaylimits_0^1{\big\Vert\sum_{j=1}^n r_j(s)x_j \big\Vert_X\mathrm{d}s}
	$$
	where $\{r_j\}_{j=1,...,n}$ is a sequence of independent $\{-1,1\}$-valued random variables on $[0,1]$.
	
	The smallest such $C$ is called $\mathcal{R}$-bound of $\mathcal{T}$, we denote it by $\mathcal{R}(\mathcal{T})$.
\end{definition}
We also collect some useful properties of $\mathcal{R}$-boundedness which will be used later. For proof, see \cite[Remark 4.1.3, Proposition 4.1.6]{pruss}.

\begin{proposition}
	\label{P2}
	\item
	\begin{enumerate}
		\item If $\mathcal{T}\subset \mathcal{L}(X,Y)$ is $\mathcal{R}$-bounded, then it is uniformly bounded with
		$$
		\sup\{|T|:T\in \mathcal{T}\} \le \mathcal{R}(\mathcal{T}).
		$$
		\item If $X$ and $Y$ are Hilbert spaces, a set $\mathcal{T}\subset \mathcal{L}(X,Y)$ is $\mathcal{R}$-bounded if and only if it is bounded.
		\item Let $X, Y$ be Banach spaces and $\mathcal{T}, \mathcal{S}\subset \mathcal{L}(X,Y)$ be $\mathcal{R}$-bounded. Then $\mathcal{T}+\mathcal{S}$ is $\mathcal{R}$-bounded as well and
		$$
		\mathcal{R}\left( \mathcal{T}+\mathcal{S}\right) \le \mathcal{R}(\mathcal{T}) + \mathcal{R}(\mathcal{S}).
		$$
		\item Let $X, Y, Z$ be Banach spaces and $\mathcal{T}\subset\mathcal{L}(X,Y)$ and $\mathcal{S}\subset\mathcal{L}(Y,Z)$ be $\mathcal{R}$-bounded. Then $\mathcal{S}\mathcal{T}$ is also $\mathcal{R}$-bounded and
		$$
		\mathcal{R}(\mathcal{S}\mathcal{T}) \le \mathcal{R}(\mathcal{S})\mathcal{R}(\mathcal{T}).
		$$
	\end{enumerate}
\end{proposition}
To this end, let us introduce another notion for the sake of being in line with the references. Let us denote the sector in the complex plane
$$
\Sigma_\theta:=\{\lambda\in\mathbb{C}\backslash \{0\}: |\arg \lambda|< \theta\}, \quad \theta\in(0,\pi).
$$

\begin{definition}${\cite[page \ 417]{weis1}}.$
	Let $X$ be a complex Banach space and $A:\mathcal{D}(A)\subseteq X \to X$ be a densely defined, closed, linear operator. $A$ is said to be sectorial if $(0,\infty)\subset \rho(A)$, has dense range and there exists some $\theta>0$ such that
	$$
	|\lambda(\lambda I - A)^{-1}|\le C, \quad \lambda\in \Sigma_{\theta}.
	$$
for some constant $C<\infty$. Moreover, $A$ is called $\mathcal{R}$-sectorial if $\{\lambda(\lambda I - A)^{-1}: \lambda\in \Sigma_{\theta} \}$ is $\mathcal{R}$-bounded.
	
	The $\mathcal{R}$-angle of $A$ is defined by
	$$
	\theta_r(A):= \inf\{\theta\in (0,\pi): \mathcal{R}(\{\lambda(\lambda I - A)^{-1}: \lambda\in \Sigma_{\pi-\theta} \})<\infty\}.
	$$
\end{definition}

The next characterization which is due to Weis \cite[Theorem 4.2]{weis}, is the key tool to prove the existence of a strong solution of (\ref{1}).

\begin{theorem}
	\label{T5}
	Let $X$ be an UMD-space and $A$ be a generator of a bounded analytic semigroup. Then 
$A$ has maximal $L^p$-regularity if and only if $A$ is $\mathcal{R}$-sectorial of angle $\theta_r(A)>\pi/2$.
\end{theorem}

\begin{remark}
\label{re1}
One needs either the operator $A$ is invertible, i.e. $0\in \rho(A)$ (or, equivalently $A$ generates a negative exponential type semigroup) or, $T< \infty$ in order to have $u\in L^p(I;X)$. Otherwise $u\in L^p_{\text{loc}}(I;X)$ only.
\end{remark}

\begin{remark}
    \begin{itemize}

\item Recall that $A$ generates a bounded analytic semigroup in $X$ if and only if $\{\lambda(\lambda I-A)^{-1} : \lambda\in \Sigma_{\frac{\pi}{2} + \theta}\}$ is bounded for some $\theta>0$, i.e. $A$ is sectorial (cf. \cite[Theorem 4.6, Section II, page 101]{engel}).

\item If $X$ is an UMD-space and the operator $A$ has bounded imaginary powers, then $A$ has maximal $L^p$-regularity, by the well-known Dore-Venni result \cite[Theorem 3.2]{dore_venni}.
\end{itemize}
\end{remark}

The above characterization in Theorem \ref{T5} provides a convenient tool to check maximal regularity for concrete operators, as we will show in the next section. 
We will also need some perturbation result which is due to Weis and Kunstmann (although we did not find 
its exact reference; A slight variation has been proved in \cite[Corollary 2]{weis1}).

\begin{theorem}${\cite[Proposition 1.18]{MT17}}.$
\label{T6}
Let $A$ be $\mathcal{R}$-sectorial in a Banach space $X$ and
$
\Sigma_{\varepsilon_0, \gamma_0} \subset \rho(A) 
$ for some $\varepsilon_0 \in (0, \pi/2), \gamma_0 \ge 0$ and
$$
\mathcal{R}\left( \left\lbrace \lambda (\lambda I - A)^{-1}: \lambda \in \Sigma_{\varepsilon_0, \gamma_0} \right\rbrace \right) \le a < \infty,
$$
where the sector is defined by,
$$
\Sigma_{\theta, \gamma} = \{\lambda \in \mathbb{C} \backslash \{0\}: |\arg\lambda | \le \pi - \theta, \; |\lambda| > \gamma\}.
$$
Let $B$ be a linear operator satisfying $\mathcal{D}(B) \supset \mathcal{D}(A)$ and for $x\in \mathcal{D}(A)$,
$$
\|Bx\|\le \delta_1\|Ax\|+ \delta_2\|x\|, \quad \text{ with } \delta_1, \delta_2  \ge 0.
$$
If $\delta_1 < 1/a$, then there exists $\gamma_1\ge \gamma_0$ such that
$$
\mathcal{R}\left( \left\lbrace A (\lambda I -\left( A+B\right) )^{-1}: \lambda \in \Sigma_{\varepsilon_0, \gamma_1} \right\rbrace \right) < \infty.
$$
\end{theorem}
We conclude this section by stating the above theorem in terms of the Cauchy problem (\ref{cauchy}) (see \cite[Therem 2.4]{dore}).

\begin{proposition}
\label{P_cauchy}
Suppose $X$ be a Banach space of class UMD, $p\in(1,\infty)$ and let $A$ be a $\mathcal{R}$-sectorial operator with $\theta_r(A)>\pi/2$. Moreover, the semigroup generated by $A$ has negative exponential type. Then (\ref{cauchy}) has a unique solution $u\in W^{1,p}(0,\infty;X)\cap L^p(0,\infty;\mathcal{D}(A))$ if and only if $f\in L^p(0,\infty;X)$ and $u_0\in \left(X ,\mathcal{D}(A)\right) _{1-\frac{1}{p},p}$.
\end{proposition}

\section{Linear problem}
\label{S5}
\setcounter{equation}{0}

After changing the full non-linear system (\ref{1}) on a fixed domain (see Appendix), we first consider the corresponding linearized problem. Assuming all the right hand side terms in (\ref{fixed_FSI}) are given fixed data, it reduces to the following system (for notational convenience, in this section, we omit the tilda on the variables):

\begin{equation}
\label{2}
\begin{cases}
\begin{aligned}
\partial_t \vu - \div \ \sigma (\vu,\pi) &= \bm{f} \ &&\text{ in } \ \Omega_F(0)\times (0,T),\\
\div \ \vu &= \div \ \bm{h} \ &&\text{ in } \ \Omega_F(0)\times (0,T),\\
\vu &=\bm{0} \ &&\text{ on } \ \partial \Omega \times (0,T),\\
\vu\cdot \vn&=(\bm{l} + \bm{\omega} \times \bm{y})\cdot \vn \ &&\text{ on } \ \partial \Omega_S(0) \times (0,T),\\ 2\left[(\DT\vu)\vn\right]_{\vt}+\alpha\vu_{\vt}&=\alpha(\bm{l} + \bm{\omega} \times \bm{y})_{\vt} \ &&\text{ on } \ \partial \Omega_S(0) \times (0,T),\\
m \bm{l}' &= - \int\displaylimits_{\partial \Omega_S(0)}{\sigma (\vu,\pi)\vn} + \bm{g}_1, \ && \ t\in(0,T),\\
J(0) \bm{\omega}' &= - \int\displaylimits_{\partial \Omega_S(0)}{\bm{y}\times \sigma (\vu,\pi)\vn}+\bm{g}_2, \ && \ t\in (0,T),\\
\vu(0) &= \vu_0 \ &&\text{ in } \ \Omega_F(0),\\
\bm{l}( 0) = \bm{l}_0, &\quad \bm{\omega}(0) = \bm{\omega}_0 .
\end{aligned}
\end{cases}
\end{equation}

We want to re-formulate the system (\ref{2}) in the form:
$$
z'(t)=Az(t)+ f(t), \quad z(0)=z_0
$$
or equivalently, to approach via semigroup theory, we consider the corresponding resolvent problem. First we treat the system with divergence-free condition and then return to the full inhomogeneous divergence condition.

As in the classical approach, we need to eliminate the pressure from both the fluid and the structure equations. The standard way to eliminate pressure from the fluid equations is to invoke the Helmholtz projection (cf. \cite{kato}).
But we also decompose the velocity field into $\mathbb{P}\vu$ and $(I_3-\mathbb{P})\vu$ which is crucial since the pressure which is eliminated from the fluid equations using the projector $\mathbb{P}$, also appears in the structure equations.

\subsection{Resolvent problem}
Given $\lambda\in\mathbb{C}, \bm{f}\in \mathbf{L}^q(\Omega_F(0))$ and $(\mathbf{g}_1,\mathbf{g}_2)\in \mathbb{C}^3\times \mathbb{C}^3$, consider the system
\begin{equation}
\label{R}
\begin{cases}
\begin{aligned}
\lambda\vu - \Delta \vu+\nabla \pi &= \bm{f} \ &&\text{ in } \ \Omega_F(0),\\
\div \ \vu &= 0 \ &&\text{ in } \ \Omega_F(0),\\
\vu &=\bm{0} \ &&\text{ on } \ \partial \Omega, \\
\vu\cdot \vn&=(\bm{l} + \bm{\omega} \times \bm{y})\cdot \vn \ &&\text{ on } \ \partial \Omega_S(0),\\ 2\left[(\DT\vu)\vn\right]_{\vt}+\alpha\vu_{\vt}&=\alpha(\bm{l} + \bm{\omega} \times \bm{y})_{\vt} \ &&\text{ on } \ \partial \Omega_S(0),\\
\lambda m \bm{l} &= - \int\displaylimits_{\partial \Omega_S(0)}{\sigma (\vu,\pi)\vn} + \bm{g}_1,\\
\lambda J(0) \bm{\omega} &= - \int\displaylimits_{\partial \Omega_S(0)}{\bm{y}\times \sigma (\vu,\pi)\vn}+\bm{g}_2.
\end{aligned}
\end{cases}
\end{equation}

The following existence result governing the steady fluid equations is required, in order to reformulate the fluid part in the above system.

\begin{proposition}
Let $q\in (1,\infty)$ and $\alpha \ge 0$ be as in (\ref{alpha}).	Given $(\bm{l}, \bm{\omega})\in \mathbb{C}^3\times \mathbb{C}^3$, there exists a unique solution $(\bm{v},\psi)\in \bm{W}^{2,q}(\Omega_F(0))\times W^{1,q}(\Omega_F(0))$ of the following Stokes problem
	\begin{equation}
	\label{3}
	\begin{cases}
	\begin{aligned}
	- \Delta \bm{v}+\nabla \psi &= \bm{0} \ &&\text{ in } \ \Omega_F(0),\\
	\div \ \bm{v} &= 0 \ &&\text{ in } \ \Omega_F(0),\\
	\bm{v}&=\bm{0} \ &&\text{ on } \ \partial \Omega,\\
	\bm{v}\cdot \vn&=(\bm{l} + \bm{\omega} \times \bm{y})\cdot \vn \ &&\text{ on } \ \partial \Omega_S(0),\\ 2\left[(\DT\bm{v})\vn\right]_{\vt}+\alpha\bm{v}_{\vt}&=\alpha(\bm{l} + \bm{\omega} \times \bm{y})_{\vt} \ &&\text{ on } \ \partial \Omega_S(0) .
	\end{aligned}
	\end{cases}
	\end{equation}
\end{proposition}

For proof, see \cite[Theorem 2.1]{AACG} where only slip condition has been considered on the whole boundary. The same technique can be adapted for the present case for no-slip condition on some part of the boundary.

Let us use the following notation, $S(\bm{l},\bm{\omega}):= \bm{v}$, $S_{pr}(\bm{l},\bm{\omega}):= \psi$ where for given $(\bm{l}, \bm{\omega})$, $(\bm{v}, \psi)$ solves (\ref{3}).
Also denote the Neumann operator
\begin{align*}
N:W^{1-1/q,q}(\partial\Omega_F(0)) &\to W^{2,q}(\Omega_F(0))\\
h &\mapsto \bm{\varphi}
\end{align*}
where $\bm{\varphi}$ solves $\Delta \bm{\varphi} = 0$ in $\Omega_F(0)$, $\frac{\partial \bm{\varphi}}{\partial \vn} = h$ on $\partial \Omega_F(0)$.
Set $N_S(h) := N(\mathbbm{1}_{\partial\Omega_S(0)}h)$ for any $h\in W^{1-1/q,q}(\partial\Omega_S(0))$.

By extrapolation, we extend the Stokes operator $A_q$ defined in (\ref{SO}) to an unbounded operator $\tilde{A_q}$ with domain $\mathcal{D}(\tilde{A}_q) := \bm{L}^q_{\sigma,\vt}(\Omega_F(0))$ on $\mathcal{D}((A_q)^*)' =\mathcal{D}(A_q)'$, so that $ (\tilde{A}_q, \mathcal{D}(A_q)')$ be the infinitesimal generator of a strongly continuous semigroup on $ \mathcal{D}(A_q)'$, satisfying
$$
\tilde{A}_q \bm{\varphi} = A_q \bm{\varphi} \quad \forall \ \bm{\varphi} \in \mathcal{D}(A_q) .
$$
Here $A^*$ denotes the adjoint operator of $A$ and $X'$ denotes the dual space of $X$.

Next, we write an equivalent formulation of the fluid part of the resolvent problem (\ref{R}), decomposing the fluid velocity into $\mathbb{P}\vu$ and $(I_3-\mathbb{P})\vu$. 
This decoupling enables us to write the pressure in terms of $\mathbb{P}\vu, \bm{l}, \bm{\omega}$ which is useful to eliminate the pressure from the structure equation.

\begin{proposition}
	\label{P3}
	Let $q\in(1,\infty)$, $\alpha\ge 0$ be as in (\ref{alpha}) and $(\bm{f}, \bm{l}, \mathbf{\omega})\in \bm{L}^q_{\sigma,\vt}(\Omega_F(0))\times \mathbb{C}^3\times \mathbb{C}^3$. Then $(\vu,\pi)\in \bm{W}^{2,q}(\Omega_F(0))\times W^{1,q}(\Omega_F(0))$ satisfies the system
	\begin{equation}
	\label{4}
	\begin{cases}
	\begin{aligned}
	\lambda \vu - \Delta \vu+\nabla \pi = \bm{f},\quad
	\div \ \vu &= 0 \ &&\text{ in } \ \Omega_F(0),\\
	\vu&=\bm{0} \ &&\text{ on } \ \partial \Omega,\\
	\vu\cdot \vn=(\bm{l} + \bm{\omega} \times \bm{x})\cdot \vn, \quad 2\left[(\DT\vu)\vn\right]_{\vt}+\alpha\vu_{\vt}&=\alpha(\bm{l} + \bm{\omega} \times \bm{x})_{\vt} \ &&\text{ on } \ \partial \Omega_S(0)
	\end{aligned}
	\end{cases}
	\end{equation}
	iff
	\begin{equation}
	\label{5}
	\begin{cases}
	\begin{aligned}
	&\lambda \mathbb{P}\vu - \tilde{A}_q\mathbb{P}\vu + \tilde{A}_q \mathbb{P}S(\bm{l},\bm{\omega}) = \mathbb{P}\bm{f}\\
	&(I_3-\mathbb{P})\vu = (I_3-\mathbb{P})S(\bm{l},\bm{\omega})\\
	&\pi = N(\Delta \mathbb{P}\vu\cdot \vn) - \lambda N_S((\bm{l}+ \bm{\omega}\times \bm{x})\cdot \vn).
	\end{aligned}
	\end{cases}
	\end{equation}
\end{proposition}

\begin{proof}
	Let $(\vu,\pi)$ satisfies (\ref{4}) and Denote $(\tilde{\vu},\tilde{\pi}):= (\vu - S(\bm{l},\bm{\omega}), \pi - S_{pr}(\bm{l},\bm{\omega}))$. Then $(\tilde{\vu},\tilde{\pi})$ satisfies
	\begin{equation*}
	\begin{cases}
	\begin{aligned}
	\lambda \tilde{\vu}- \Delta \tilde{\vu}+\nabla \tilde{\pi} = \bm{f} - \lambda S(\bm{l},\bm{\omega}), \quad \div \ \tilde{\vu} &= 0 \ &&\text{ in } \ \Omega_F(0),\\
	\tilde{\vu}&= \bm{0} \ &&\text{ on } \ \partial \Omega,\\
	\tilde{\vu}\cdot\vn=0, \quad 2\left[(\DT\tilde{\vu})\vn\right]_{\vt}+\alpha\tilde{\vu}_{\vt}&=\bm{0} \ &&\text{ on } \ \partial \Omega_S(0).
	\end{aligned}
	\end{cases}
	\end{equation*}
	This shows $\tilde{\vu} \in \mathcal{D}(A_q)$ and $\mathbb{P}\tilde{\vu} = \tilde{\vu}$. Therefore, applying the projection $\mathbb{P}$ on the first equation of the above system, we get
	\begin{equation}
	\label{6}
	\lambda \mathbb{P}(\tilde{\vu} + S(\bm{l},\bm{\omega})) - A_q \tilde{\vu}+ \mathbb{P}\nabla \tilde{\pi} = \mathbb{P}\bm{f}.
	\end{equation}
	But, since
	\begin{align*}
	- A_q \tilde{\vu}+ \mathbb{P}\nabla \tilde{\pi} = \mathbb{P}(-\Delta \tilde{\vu} + \nabla \tilde{\pi})& = \mathbb{P}\mathbb{P}(-\Delta \tilde{\vu})\\
	& = \mathbb{P}(-\Delta \tilde{\vu}) = -A_q\tilde{\vu} = -A_q\mathbb{P}\tilde{\vu} = -A_q\mathbb{P}(\vu - S(\bm{l},\bm{\omega})),
	\end{align*}
we obtain from (\ref{6}), $\lambda \mathbb{P}\vu - A_q\mathbb{P}\vu + A_q\mathbb{P}S(\bm{l},\bm{\omega}) = \mathbb{P}\bm{f}$ in $\Omega_F(0)$.
	
	Also, as $(I_3-\mathbb{P})\tilde{\vu} = \bm{0}$, we deduce $(I_3-\mathbb{P})\vu = (I_3-\mathbb{P})S(\bm{l},\bm{\omega})$.
	
	Furthermore, from (\ref{4}), taking divergence in the first equation yields, $\Delta \pi =0 \text{ in } \Omega_F(0)$. And since $\Delta (I_3-\mathbb{P})\vu = \bm{0} \text{ in } \Omega_F(0)$ (follows from the properties of Helmholtz projection),
	\begin{align*}
	\frac{\partial \pi}{\partial \vn}\Big\arrowvert_{\partial \Omega_F(0)} = \Delta \mathbb{P}\vu \cdot \vn - \lambda \vu \cdot \vn =
	\begin{cases}
	\Delta \mathbb{P}\vu \cdot \vn &\text{ on } \partial \Omega\\
	\Delta \mathbb{P}\vu \cdot \vn - \lambda (\bm{l}+ \bm{\omega}\times \bm{x})\cdot \vn &\text{ on } \partial \Omega_S(0).
	\end{cases}
	\end{align*}
	Therefore, the expression of $\pi$ in (\ref{5}) follows from the definition of the operators $N$ and $N_S$.
	
	Conversely, let $\vu\in \bm{W}^{2,q}(\Omega_F(0))$ satisfies the system (\ref{5}). Because of the relation $(I_3-\mathbb{P})\vu = (I_3-\mathbb{P})S(\bm{l},\bm{\omega})$, defining $\tilde{\vu}:= \vu - S(\bm{l},\bm{\omega})$ we get, $\tilde{\vu}\in \bm{L}^q_{\sigma,\vt}(\Omega_F(0))$ and $\mathbb{P}\tilde{\vu} = \tilde{\vu}$. Thus the first equation of (\ref{5}) can be written as
	$$
	\tilde{A}_q \tilde{\vu} = \mathbb{P}(\lambda \vu - \bm{f})=: \bm{h}.
	$$
	But since, $\bm{h}\in \bm{L}^q_{\sigma,\vt}(\Omega)$ and $\tilde{A}_q$ is the generator of a strongly continuous semigroup (in fact, the maximal monotone property of the operator $A_q$ and $\tilde{A}_q$ is sufficient), then $\tilde{\vu} \in \mathcal{D}(A_q)$ and hence, the boundary conditions in (\ref{4}) is satisfied by $\vu$.
	
	[The proof of the above statement is very simple and holds for general unbounded operators. For completeness, we mention it here:
	Let $A$ be a maximal dissipative operator in a Banach space $X$ with dense domain $D(A)$ and $A_{-1}$ be its extension by extrapolation in $X_{-1}$ with domain $X$. If $x\in X$ is such that $A_{-1}x \in X$, then $x\in D(A)$ and $A_{-1}x = Ax$.
	
	Proof: Define, $f = x+ A_{-1}x \in X$. Since $A$ is maximum dissipative, there exists $y\in D(A)$ such that $ y+A y = f$. Hence, $y + A_{-1}y = f$. But as $A_{-1}$ is dissipative, it follows $x=y\in D(A)$.]
	
	Now we write once again the first equation of (\ref{5}) in terms of $\tilde{\vu}$ as,
	$$
	\lambda \tilde{\vu} - A_q\tilde{\vu} = \mathbb{P}(\bm{f} - \lambda S(\bm{l},\bm{\omega})).
	$$
	Therefore, from the characterization of $(I-\mathbb{P})$, there exists $\tilde{\pi}\in W^{1,q}(\Omega_F(0))$ such that
	$$
	\lambda \tilde{\vu} - \Delta \tilde{\vu}+ \nabla \tilde{\pi} = \bm{f} - \lambda S(\bm{l},\bm{\omega}).
	$$
	Then $(\vu, \pi)$ with $\pi = \tilde{\pi} + S_{pr}(\bm{l},\bm{\omega})$ satisfies (\ref{4}).
	\hfill
\end{proof}

Now using the expression of the pressure $(\ref{5})_3$ obtained above, we can re-write the two equations in (\ref{R}) satisfied by $\bm{l}$ and $\bm{\omega}$.
\begin{equation}
\label{7}
\begin{aligned}
\lambda m \bm{l} & = -2\int\displaylimits_{\partial \Omega_S(0)}{(\DT\vu)\vn} + \int\displaylimits_{\partial \Omega_S(0)}{\pi \vn} + \bm{g}_1\\
& = -2 \left[ \int\displaylimits_{\partial \Omega_S(0)}{(\DT(\mathbb{P}\vu))\vn} + \int\displaylimits_{\partial \Omega_S(0)}{\DT((I_3-\mathbb{P})S(\bm{l},\bm{\omega}))\vn}\right] \\
& + \int\displaylimits_{\partial \Omega_S(0)}{N(\Delta\mathbb{P}\vu\cdot \vn)\vn} - \lambda \int\displaylimits_{\partial \Omega_S(0)}{N_S((\bm{l}+ \bm{\omega}\times \bm{y})\cdot \vn)\vn} + \bm{g}_1
\end{aligned}
\end{equation}
and
\begin{equation}
\label{8}
\begin{aligned}
\lambda J\bm{\omega}&  = -2\int\displaylimits_{\partial \Omega_S(0)}{\bm{y}\times (\DT\vu)\vn} + \int\displaylimits_{\partial \Omega_S(0)}{\bm{y}\times \pi\vn} + \bm{g}_2\\
& = -2 \left[ \int\displaylimits_{\partial \Omega_S(0)}{\bm{y}\times (\DT(\mathbb{P}\vu))\vn} + \int\displaylimits_{\partial \Omega_S(0)}{\bm{y}\times \DT((I_3-\mathbb{P})S(\bm{l},\bm{\omega}))\vn}\right] \\
& + \int\displaylimits_{\partial \Omega_S(0)}{\bm{y}\times N(\Delta \mathbb{P}\vu \cdot \vn)\vn} - \lambda \int\displaylimits_{\partial \Omega_S(0)}{\bm{y}\times N_s((\bm{l}+\bm{\omega}\times \bm{y})\cdot \vn)\vn} + \bm{g}_2.
\end{aligned}
\end{equation}
So, (\ref{7}) and (\ref{8}) can be written combindedly in the following form:
\begin{equation*}
\lambda K
\begin{pmatrix}
\bm{l}\\
\bm{\omega}
\end{pmatrix}
= C_1 \mathbb{P}\vu + C_2
\begin{pmatrix}
\bm{l}\\
\bm{\omega}
\end{pmatrix}
+
\begin{pmatrix}
\bm{g}_1\\
\bm{g}_2
\end{pmatrix}
\end{equation*}
where
\begin{equation*}
K = \mathbb{I} +M
\end{equation*}
with
\begin{equation*}
\mathbb{I}= \begin{pmatrix}
mI_3 & 0\\
0 & J
\end{pmatrix}_{6\times 6}
\end{equation*}
be the constant momentum matrix,
\begin{equation*}
M \begin{pmatrix}
\bm{l}\\
\bm{\omega}
\end{pmatrix}
= \begin{pmatrix}
\int\displaylimits_{\partial \Omega_S(0)}{N_S((\bm{l}+\bm{\omega}\times \bm{y})\cdot \vn)\vn}\\
\int\displaylimits_{\partial \Omega_S(0)}{\bm{y}\times N_S((\bm{l}+\bm{\omega}\times \bm{y})\cdot \vn)\vn}
\end{pmatrix}_{6\times 1}
\end{equation*}
be the added mass matrix,
\begin{equation*}
C_1\mathbb{P}\vu =
\begin{pmatrix}
-2 \int\displaylimits_{\partial \Omega_S(0)}{(\DT(\mathbb{P}\vu))\vn}+ \int\displaylimits_{\partial \Omega_S(0)}{N(\Delta\mathbb{P}\vu\cdot \vn)\vn}\\
-2 \int\displaylimits_{\partial \Omega_S(0)}{\bm{y}\times (\DT(\mathbb{P}\vu))\vn} + \int\displaylimits_{\partial \Omega_S(0)}{\bm{y}\times N(\Delta \mathbb{P}\vu \cdot \vn)\vn}
\end{pmatrix}_{6\times 1}
\end{equation*}
and
\begin{equation*}
C_2 \begin{pmatrix}
\bm{l}\\
\bm{\omega}
\end{pmatrix} =
\begin{pmatrix}
\int\displaylimits_{\partial \Omega_S(0)}{\DT((I_3-\mathbb{P})S(\bm{l},\bm{\omega}))\vn} \\
\int\displaylimits_{\partial \Omega_S(0)}{\bm{y}\times \DT((I_3-\mathbb{P})S(\bm{l},\bm{\omega}))\vn}
\end{pmatrix}_{6 \times 1}.
\end{equation*}

\begin{lemma}
	The matrix  $K$ defined above is an invertible matrix.
\end{lemma}

	The main point to prove the above lemma is that $M$ is a positive semi-definite, symmetric matrix. Then $K$ being the sum of an invertible matrix and a semi-definite matrix, is itself invertible. For details, see \cite[Lemma 4.3]{GGH}.

Let us now define the fluid-structure operator $\mathcal{A}_{FS}: \mathcal{D}(\mathcal{A}_{FS})\subset X \to X$ with
$$
X:= \bm{L}^q_{\sigma,\vt}(\Omega_F(0))\times \mathbb{C}^3 \times \mathbb{C}^3,
$$
and
\begin{equation*}
\begin{cases}
\begin{aligned}
&\mathcal{D}(\mathcal{A}_{FS}):=\{(\mathbb{P}\vu, \bm{l},\bm{\omega})\in X:\mathbb{P}\vu - \mathbb{P}S(\bm{l},\bm{\omega})\in\mathcal{D}(A_q) \},\\
&\qquad \qquad \mathcal{A}_{FS} =
\begin{pmatrix}
A_q & -A_q\mathbb{P}S\\
K^{-1}C_1 & K^{-1}C_2
\end{pmatrix}_{9\times 9}.
\end{aligned}
\end{cases}
\end{equation*}

Combining the above results, we obtain below an equivalent formulation of the resolvent problem (\ref{R}).

\begin{proposition}
	\label{P0}
	Let $q\in (1,\infty)$, $\alpha\ge 0$ be as in (\ref{alpha}) and $(\bm{f},\bm{g}_1, \bm{g}_2)\in X$. Then $(\vu, \pi, \bm{l}, \bm{\omega})\in W^{2,q}(\Omega_F(0))\times W^{1,q}(\Omega_F(0))\times \mathbb{C}^3 \times \mathbb{C}^3$ satisfies the resolvent problem (\ref{R}) iff
	\begin{equation}
	\label{9}
	\begin{aligned}
	&(\lambda I-\mathcal{A}_{FS})
	\begin{pmatrix}
	\mathbb{P}\vu\\
	\bm{l}\\
	\bm{\omega}
	\end{pmatrix}
	=
	\begin{pmatrix}
	\mathbb{P}\bm{f}\\
	\tilde{\bm{g}_1}\\
	\tilde{\bm{g}_2}
	\end{pmatrix}
	\\
	& (I_3-\mathbb{P})\vu = (I_3-\mathbb{P})S(\bm{l},\bm{\omega})\\
	& \pi = N(\Delta \mathbb{P}\vu\cdot \vn) - \lambda N_S((\bm{l}+ \bm{\omega}\times \bm{x})\cdot \vn)
	\end{aligned}
	\end{equation}
	where $(\tilde{\bm{g}_1}, \tilde{\bm{g}_2})^T = K^{-1}(\bm{g}_1, \bm{g}_2)^T$.
\end{proposition}

The following lemma gives an equivalent norm on the domain of the operator $\mathcal{A}_{FS}$.

\begin{lemma}
	\label{L1}
	The map
	$$
	(\mathbb{P}\vu, \bm{l}, \bm{\omega}) \mapsto \|\mathbb{P}\vu\|_{\bm{W}^{2,q}(\Omega_F(0))} + \|\bm{l}\|_{\mathbb{C}^3}+ \|\bm{\omega}\|_{\mathbb{C}^3}
	$$
	is a norm on $\mathcal{D}(\mathcal{A}_{FS})$, equivalent to the graph norm.
\end{lemma}

The proof is similar to the one in \cite[Proposition 3.3]{raymond1} for Dirichlet condition, since the Stokes operator with Navier condition $A_q$ is an isomorphism from $\mathcal{D}(A_q)$ to $ \bm{L}^q_{\sigma,\vt}(\Omega_F(0))$.

Next we show the $\mathcal{R}$-boundedness of the resolvent operator $\mathcal{A}_{FS}$ which is a cornerstone in order to prove the maximal $L^p-L^q$-regularity of the associated linear problem (\ref{2}) with $\bm{h} = \bm{0}$. 

\begin{theorem}
	\label{T2}
	Let $q\in (1,\infty)$ and $\alpha\ge 0$ be as in (\ref{alpha}). There exist $\gamma_1>0$ and $\theta>0$ such that $\gamma_1 + \Sigma_{\pi/2+\theta}\subset \rho(\mathcal{A}_{FS})$ and
$\mathcal{A}_{FS}$ is $\mathcal{R}$-sectorial.
\end{theorem}

\begin{proof}
	The proof follows the similar argument as in \cite[Theorem 3.11]{MT}.
	Therefore we mention only the differences.
	
	Since the matrices $K, C_1, C_2$ are exactly same as in \cite{MT}, we only need to take care of the operators $A_q$ and $S$.
Writing $\mathcal{A}_{FS} = \tilde{\mathcal{A}}_{FS} + B_{FS}$ where
	\begin{align*}
	\tilde{\mathcal{A}}_{FS} =
	\begin{pmatrix}
	A_q & -A_q\mathbb{P}S\\
	0 & 0
	\end{pmatrix}, \quad
	B_{FS} =
	\begin{pmatrix}
	0 & 0\\
	K^{-1}C_1 & K^{-1}C_2
	\end{pmatrix},
	\end{align*}
we first prove that $\tilde{\mathcal{A}}_{FS}$ with $\mathcal{D}(\tilde{\mathcal{A}}_{FS}) = \mathcal{D}(\mathcal{A}_{FS})$ is $\mathcal{R}$-sectorial on $X$. Using the identity $-(\lambda I_3 -A_q)^{-1}A_q\mathbb{P}S =-\lambda(\lambda I_3 - A_q)^{-1}\mathbb{P}S+\mathbb{P}S$, we can write,
	\begin{align*}
	\lambda (\lambda I_9 - \tilde{\mathcal{A}}_{FS})^{-1} =
	\begin{pmatrix}
	\lambda (\lambda I_3 - A_q)^{-1} & -\lambda(\lambda I_3 - A_q)^{-1}\mathbb{P}S+\mathbb{P}S\\
	0 & I
	\end{pmatrix}.
	\end{align*}
	Since the Stokes operator $A_q$ is $\mathcal{R}$-sectorial in $\bm{L}^q(\Omega_F(0))$ (see \cite{AEG} where it is proved for the Stokes operator with slip condition on the whole boundary; The same can be adapted straight forward for this mixed case as well) and using the properties 3. and 4. of Proposition \ref{P2}, we get
	$$
	\mathcal{R}\{-\lambda(\lambda I_3 - A_q)^{-1}\mathbb{P}S+\mathbb{P}S\}\le \mathcal{R}\{-\lambda(\lambda I_3 - A_q)^{-1}\} \mathcal{R}(\mathbb{P}S)+ \mathcal{R}(\mathbb{P}S)<\infty.
	$$
	Observe that the $\mathcal{R}$-boundedness of $\mathbb{P}S$ follows from the definition easily due to the continuity of $\mathbb{P}S$.
	Therefore the desired result follows.
	
Next we show that $B_{FS}$ is a small perturbation of $\tilde{A}_{FS}$ which yields the $\mathcal{R}$-sectoriality of $B_{FS}$. This concludes the proof.
	
	To do so, first let us show that $B_{FS}\in \mathcal{L}(\mathcal{D}(\mathcal{A}_{FS}),\mathbb{C}^9) $. By Lemma \ref{L1}, for any $(\mathbb{P}\vu, \bm{l}, \bm{\omega}) \in \mathcal{D}(\mathcal{A}_{FS})$, we have $(\mathbb{P}\vu, \bm{l}, \bm{\omega}) \in\bm{W}^{2,q}(\Omega_F(0))\times \mathbb{C}^3 \times \mathbb{C}^3 $. Therefore, by trace theorem, $(\mathbb{D}(\mathbb{P}\vu))\vn\in \bm{W}^{1-1/q,q}(\partial \Omega_S(0))$ and $\int\displaylimits_{\partial \Omega_S(0)}{(\DT(\mathbb{P}\vu))\vn\, \mathrm{d}s}\in\mathbb{C}^3$. On the other hand, $\Delta \mathbb{P}\vu\in\bm{L}^q(\Omega_F(0))$ and $\div \ \Delta \mathbb{P}\vu =0$ which implies $\Delta \mathbb{P}\vu\cdot \vn \in \bm{W}^{-1/q,q}(\partial\Omega_S(0))$. Also the following condition is satisfied (due to the divergence free condition)
	$$
	\left\langle \Delta \mathbb{P}\vu\cdot \vn,1\right\rangle _{\bm{W}^{-1/q,q} \times \bm{W}^{1/q,q'}}=0 .
	$$
	Thus $N(\Delta \mathbb{P}\vu\cdot \vn)\in \bm{W}^{1,q}(\Omega_F(0))$ and $\int\displaylimits_{\partial \Omega_S(0)}{N(\Delta \mathbb{P}\vu\cdot \vn)\vn \ \mathrm{d}s}\in\mathbb{C}^3$.
	Other terms of $C_1$ can be checked in the same way.
	Similarly, for the matrix $C_2$, notice that $S(\bm{l},\bm{\omega})\in \bm{W}^{2,q}(\Omega_F(0))$ and $\mathbb{P}S(\bm{l},\bm{\omega})\in \bm{W}^{2,q}(\Omega_F(0))$ since $\mathbb{P}\vu\in \bm{W}^{2,q}(\Omega_F(0))$ and hence, $(I_3-\mathbb{P})S(\bm{l},\bm{\omega})\in \bm{W}^{2,q}(\Omega_F(0))$. Therefore, $\DT ((I_3-\mathbb{P})S(\bm{l},\bm{\omega}))\in \bm{W}^{1-1/q,q}(\partial\Omega_S(0))$ and $\int\displaylimits_{\partial \Omega_S(0)}{\DT ((I_3-\mathbb{P})S(\bm{l},\bm{\omega}))\vn \ \mathrm{d}s}\in\mathbb{C}^3$. The other term of $C_2$ can be treated in the same manner. Thus, we deduce that
	$$
	B_{FS} :\mathcal{D}(B_{FS}) = \mathcal{D}(\mathcal{A}_{FS})\to \mathbb{C}^9
	$$
	is a bounded linear operator, due to the continuity of the trace operator and the elliptic regularity results. This concludes that $B_{FS}$ is a finite rank operator and hence compact. Therefore  we can say $(B_{FS},\mathcal{D}(B_{FS}))$ is $\tilde{\mathcal{A}}_{FS}$-compact by the definition \cite[Chapter III, Definition 2.15]{engel}. Then, by \cite[Chapter III, Lemma 2.16]{engel}, we get that $B_{FS}$ is $\tilde{\mathcal{A}}_{FS}$-bounded with $\tilde{\mathcal{A}}_{FS}$-bound being $0$, that is, for all $\delta>0$, there exists $C(\delta)>0$ such that
	$$
	\|B_{FS}\,z\|\le \delta\|\tilde{\mathcal{A}}_{FS}\, z\| + C(\delta)\|z\| \qquad \forall z\in \mathcal{D}(\tilde{\mathcal{A}}_{FS}).
	$$
	Finally, applying the perturbation Theorem \ref{T6} with $\gamma_0 =0$, we obtain the existence of $\gamma_1>0$ such that $\tilde{\mathcal{A}}_{FS} + B_{FS}$ is $\mathcal{R}$-sectorial on the shifted sector $\gamma_1 + \Sigma_{\pi/2+\theta}$.
\hfill
\end{proof}

Note that from the above result, it does not follow that the entire right half plane is contained in the resolvent i.e. $0\in \rho(\mathcal{A}_FS)$ and hence we do not have yet the global in time existence result (cf. Remark \ref{re1}).

\subsection{Exponential stability}

Next we show that the operator $\mathcal{A}_{FS}$ generates an exponentially stable semigroup.

\begin{theorem}
	\label{T1}
	Let $p,q\in (1,\infty)$ and $\alpha\ge 0$ be as in (\ref{alpha}). The operator $\mathcal{A}_{FS}$ generates an exponentially stable semigroup $(e^{t\mathcal{A}_{FS}})_{t\ge 0}$ on $X$. In other words, there exist constants $\eta_0>0$ and $C>0$ such that
	\begin{equation*}
	\|e^{t\mathcal{A}_{FS}}(\vu_0,\bm{l}_0,\bm{\omega}_0)^T\|_{X}\le C e^{-\eta_0 t}\|(\vu_0,\bm{l}_0,\bm{\omega}_0)^T\|_{X}.
	\end{equation*}
\end{theorem}

\begin{proof}
	First note that the entire right half plane $\{\lambda\in\mathbb{C}: \Re \lambda \ge 0\}$ is contained in the resolvent set of $\mathcal{A}_{FS}$. Indeed, from Theorem \ref{T2}, we have that
	$$
	\{\lambda\in\mathbb{C}: \Re \lambda \ge 0\}	\setminus\{0\} \in \rho(\mathcal{A}_{FS})
	$$
	and $0\in \rho(\mathcal{A}_{FS})$ is shown in the next theorem.
	Therefore, since the resolvent set is an open set, we get in fact, for some $\eta>0$,
	$$
	\{\lambda \in \mathbb{C}:\Re \lambda \ge -\eta\}\subset \rho(\mathcal{A}_{FS}).
	$$
	Now as $\mathcal{A}_{FS}$ generates an analytic semigroup, Proposition 2.9 in \cite[Part II, Chapter, pp 120]{ben} tells that the corresponding semigroup is of negative type. Then Corollary 2.2 (i) in \cite[pp 93]{ben} provides the exponential stability.
	\hfill
\end{proof}

\begin{theorem}
	\label{T3}
	Let $p,q\in (1,\infty)$, $\alpha\ge 0$ be as in (\ref{alpha}) and $\lambda \in \mathbb{C}$ with $\Re \lambda \ge 0$. Then, for any $(\bm{f}, \bm{g}_1,\bm{g}_2)\in X$, the resolvent system (\ref{R}) admits a unique solution satisfying the following estimate
	\begin{equation}
	\label{10}
	\|\vu\|_{\bm{W}^{2,q}(\Omega_F(0))} + \|\pi\|_{W^{1,q}(\Omega_F(0))}+ \|\bm{l}\|_{\mathbb{C}^3} + \|\bm{\omega}\|_{\mathbb{C}^3} \le C \|(\bm{f}, \bm{g}_1,\bm{g}_2)\|_{X}.
	\end{equation}
\end{theorem}

\begin{proof}
	From Theorem \ref{T2}, we know that there exists $\tilde{\lambda}>0$ such that $(\tilde{\lambda} I - \mathcal{A}_{FS})$ is invertible. Thus (\ref{R}) can be written as (since it is equivalent to (\ref{9}) by Proposition \ref{P0}),
	\begin{equation}
	\label{11}
	\begin{aligned}
	& \begin{pmatrix}
	\mathbb{P}\vu\\
	\bm{l}\\
	\bm{\omega}
	\end{pmatrix}
	= \left[ I + (\lambda - \tilde{\lambda})(\tilde{\lambda} I-\mathcal{A}_{FS})^{-1}\right]^{-1} (\tilde{\lambda} I-\mathcal{A}_{FS})^{-1}
	\begin{pmatrix}
	\mathbb{P}\bm{f}\\
	\tilde{\bm{g}_1}\\
	\tilde{\bm{g}_2}
	\end{pmatrix}
	\\
	& (I-\mathbb{P})\vu = (I-\mathbb{P})S(\bm{l},\bm{\omega})\\
	& \pi = N(\Delta \mathbb{P}\vu\cdot \vn) - \lambda N_S((\bm{l}+ \bm{\omega}\times \bm{y})\cdot \vn).
	\end{aligned}
	\end{equation}
	In fact $(\tilde{\lambda} I-\mathcal{A}_{FS})$ has a continuous inverse, therefore $(\tilde{\lambda} I-\mathcal{A}_{FS})^{-1}$ is a compact operator. Hence by Fredholm alternative theorem, the existence and uniqueness of solution of the above system is equivalent. It is enough to show the uniqueness of (\ref{11}) only. Then the estimate (\ref{10}) follows easily.
	
	Let $(\vu,\pi,\bm{l},\bm{\omega})\in \bm{W}^{2,q}(\Omega_F(0))\times W^{1,q}(\Omega_F(0))\times \mathbb{C}^3\times \mathbb{C}^3$ satisfies the homogeneous system
	\begin{equation}
	\label{homR}
	\begin{cases}
	\begin{aligned}
	\lambda\vu - \Delta \vu+\nabla \pi = \bm{0}, \quad
	\div \ \vu &= 0 \ &&\text{ in } \ \Omega_F(0),\\
	\vu&=\bm{0} \ &&\text{ on } \ \partial \Omega ,\\
	\vu\cdot \vn&=(\bm{l} + \bm{\omega} \times \bm{y})\cdot \vn \ &&\text{ on } \ \partial \Omega_S(0),\\ 2\left[(\DT\vu)\vn\right]_{\vt}+\alpha\vu_{\vt}&=\alpha(\bm{l} + \bm{\omega} \times \bm{y})_{\vt} \ &&\text{ on } \ \partial \Omega_S(0),\\
	\lambda m \bm{l} &= - \int\displaylimits_{\partial \Omega_S(0)}{\sigma (\vu,\pi)\vn},\\
	\lambda J(0) \bm{\omega} &= - \int\displaylimits_{\partial \Omega_S(0)}{\bm{y}\times \sigma (\vu,\pi)\vn}.
	\end{aligned}
	\end{cases}
	\end{equation}
	We first show that $(\vu,\pi)\in \bm{W}^{2,2}(\Omega_F(0))\times W^{1,2}(\Omega_F(0))$. If $q\ge 2$, it is obvious. If $ q\in (1,2)$, we rewrite (\ref{homR}) as
	\begin{equation}
	\label{}
	\begin{aligned}
	&(\tilde{\lambda} I-\mathcal{A}_{FS})
	\begin{pmatrix}
	\mathbb{P}\vu\\
	\bm{l}\\
	\bm{\omega}
	\end{pmatrix}
	= (\lambda - \tilde{\lambda})
	\begin{pmatrix}
	\mathbb{P}\bm{u}\\
	\bm{l}\\
	\bm{\omega}
	\end{pmatrix}
	\\
	& (I-\mathbb{P})\vu = (I-\mathbb{P})S(\bm{l},\bm{\omega})\\
	& \pi = N(\Delta \mathbb{P}\vu\cdot \vn) - \lambda N_S((\bm{l}+ \bm{\omega}\times \bm{y})\cdot \vn).
	\end{aligned}
	\end{equation}
	But as $\bm{W}^{2,q}(\Omega_F(0)) \subset \bm{L}^2(\Omega_F(0))$ and $(\tilde{\lambda}I - \mathcal{A}_{FS})$ is invertible, we obtain $(\vu,\pi)\in \bm{W}^{2,2}(\Omega_F(0))\times W^{1,2}(\Omega_F(0))$.
	
	Now multiplying the first equation of (\ref{homR}) by $\overline{\vu}$ and integrating by parts, we obtain
	\begin{equation*}
	\begin{aligned}
	0 & = \lambda\int\displaylimits_{\Omega_F(0)}{|\vu|^2} + 2\int\displaylimits_{\Omega_F(0)}{|\DT\vu|^2} - 2\int\displaylimits_{\partial \Omega_F(0)}{(\DT \vu)\vn\cdot \overline{\vu}} + \int\displaylimits_{\partial \Omega_F(0)}{\pi \ \overline{\vu}\cdot \vn}\\
	& = \lambda\int\displaylimits_{ \Omega_F(0)}{|\vu|^2} + 2\int\displaylimits_{\Omega_F(0)}{|\DT\vu|^2} -\int\displaylimits_{\partial \Omega_S(0)}{\sigma(\vu,\pi)\vn \cdot \overline{\vu}} .
	\end{aligned}
	\end{equation*}
	Note that
	\begin{equation*}
	\begin{aligned}
	& \int\displaylimits_{\partial \Omega_S(0)}{\sigma(\vu,\pi)\vn \cdot \overline{\vu}}\\
	= & \int\displaylimits_{\partial \Omega_S(0)}{\sigma(\vu,\pi)\vn \cdot \left( \overline{\vu} - (\overline{\bm{l}} + \overline{\bm{\omega}} \times \bm{y})\right) } + \int\displaylimits_{\partial \Omega_S(0)}{\sigma(\vu,\pi)\vn \cdot (\overline{\bm{l}} + \overline{\bm{\omega}} \times \bm{y}) }\\
	= & \int\displaylimits_{\partial \Omega_S(0)}{[\sigma(\vu,\pi)\vn]_{\vt} \cdot \left( \overline{\vu} - (\overline{\bm{l}} + \overline{\bm{\omega}} \times \bm{y})\right)_{\vt}} + \overline{\bm{l}} \cdot \int\displaylimits_{\partial \Omega_S(0)}{\sigma(\vu,\pi)\vn} + \overline{\bm{\omega}}\cdot \int\displaylimits_{\partial \Omega_S(0)}{\sigma(\vu,\pi)\vn \times \bm{y}}\\
	= & -\int\displaylimits_{\partial \Omega_S(0)}{\alpha |\vu_{\vt} - (\overline{\bm{l}}+ \overline{\bm{\omega}}\times\bm{y})_{\vt}|^2} - \lambda m |\bm{l}|^2 - \lambda J(0)\bm{\omega}\cdot\overline{\bm{\omega}}
	\end{aligned}
	\end{equation*}
	where we have used the 4th and 5th equations of (\ref{homR}), multiplied by $\overline{\bm{l}}, \overline{\bm{\omega}}$ respectively. Thus we obtain the energy equality
	\begin{equation*}
	\lambda\int\displaylimits_{\Omega_F(0)}{|\vu|^2} + 2\int\displaylimits_{\Omega_F(0)}{|\DT\vu|^2}+ \int\displaylimits_{\partial \Omega_S(0)}{\alpha |\vu_{\vt} - (\overline{\bm{l}}+ \overline{\bm{\omega}}\times\bm{y})_{\vt}|^2} + \lambda m |\bm{l}|^2 + \lambda J(0)\bm{\omega}\cdot\overline{\bm{\omega}}=0.
	\end{equation*}
	Taking the real part of the above equation yields,
	\begin{equation*}
	\begin{aligned}
	& \Re \lambda\int\displaylimits_{ \Omega_F(0)}{|\vu|^2} + 2\int\displaylimits_{\Omega_F(0)}{|\DT\vu|^2}+ \int\displaylimits_{\partial \Omega_S(0)}{\alpha |\vu_{\vt} - (\overline{\bm{l}}+ \overline{\bm{\omega}}\times\bm{y})_{\vt}|^2} + m\Re \lambda |\bm{l}|^2 + \Re\lambda \ J(0)\bm{\omega}\cdot\overline{\bm{\omega}} =0.
	\end{aligned}
	\end{equation*}
	But as $\Re \lambda \ge 0$ and $\alpha \ge 0$ (and also, $J(0)a \cdot \overline{a}>0 \ \forall a\in \R^3$), we obtain $\bm{l} = \bm{0}$ and
	\begin{equation*}
	\int\displaylimits_{\Omega_F(0)}{|\DT\vu|^2}=0.
	\end{equation*}
	This implies, along with the fact that $\vu\cdot \vn = 0$ on $\partial\Omega$ and $\alpha \not\equiv 0$, that $\vu =\bm{0}$ in $\Omega_F(0)$ (cf. \cite[Proposition 3.7]{AACG}). Finally, from the boundary condition $2\left[(\DT\vu)\vn\right]_{\vt}+\alpha\vu_{\vt}=\alpha(\bm{l} + \bm{\omega} \times \bm{y})_{\vt}$ on $\partial \Omega_S(0)$, we deduce $\bm{\omega} = \bm{0}$.
	Remember, $\left( \bm{\omega}\times\bm{y} \right) \cdot \vn =0$ holds always.
	\hfill
\end{proof}

The above result together with Proposition \ref{P_cauchy} yields the following maximal regularity for the linearised system (\ref{2}).

\begin{theorem}
	\label{T11}
	Let $\Omega_F(0)$ be a bounded domain of class $\HC{2}{1}$, $p,q\in(1,\infty)$ and $\alpha \ge 0$ be as in (\ref{alpha}). Also assume that $(\vu_0,\bm{l}_0,\bm{\omega}_0)\in B^{2-2/p}_{q,p}(\Omega_F(0))\times \R^3\times \R^3$ satisfies the compatibility condition (\ref{CN}). Then for any $\bm{f}\in L^p(0,\infty;\bm{L}^q(\Omega_F(0)))$, $\bm{g}_1\in L^p(0,\infty;\R^3)$ and $\bm{g}_2\in L^p(0,\infty;\R^3)$, problem (\ref{2}) with $\bm{h} = \bm{0}$ admits a unique solution
	\begin{equation*}
		\vu\in W^{2,1}_{q,p}(Q_F^\infty), \pi\in L^p(0,\infty;W^{1,q}(\Omega_F(0))), (\bm{l}, \bm{\omega})\in W^{1,p}(0,\infty;\R^6)	
	\end{equation*}
	which satisfies the estimate
	\begin{equation*}
		\begin{aligned}
			& \quad \ \|\vu\|_{W^{2,1}_{q,p}(Q_F^\infty)} + \|\pi\|_{L^p(0,\infty;W^{1,q}(\Omega_F(0)))} + \| \bm{l}\|_{W^{1,p}(0,\infty;\R^3)} + \|\bm{\omega}\|_{W^{1,p}(0,\infty;\R^3)}\\
			& \le C \left(\|\bm{f}\|_{L^p(0,\infty;\bm{L}^q(\Omega_F(0)))} + \|\left( \bm{g}_1,\bm{g}_2 \right) \|_{L^p(0,\infty;\R^6)}+ \|\vu_0\|_{B^{2-2/p}_{q,p}(\Omega_F(0))} + \|\left( \bm{l}_0,\bm{\omega}_0\right) \|_{\R^6} \right)
		\end{aligned}
	\end{equation*}
	where the constant $C>0$ depends only on $\alpha, p,q$ and $\Omega_S(0)$.
\end{theorem}


We are now in the position to prove the maximal $L^p-L^q$ regularity for the system (\ref{2}) with non-vanishing divergence condition which is essential to treat the full non-linear problem as described in the beginning of Section \ref{S5}. 
We prove the following theorem.
\begin{theorem}
	\label{T4}
	Let $p,q\in (1,\infty)$ and $\alpha\ge 0$ be as in (\ref{alpha}). Let $\eta\in (0,\eta_0)$ where $\eta_0$ is the constant introduced in Theorem \ref{T1} and $(\bm{l}_0, \bm{\omega}_0, \vu_0) \in \R^3 \times \R^3 \times B_{q,p}^{2(1-1/p)}(\Omega_F(0))$ satisfying the compatibility conditions (\ref{CN}).
	Then for any $e^{\eta t} \bm{f}\in L^p(0,\infty;\bm{L}^q(\Omega_F(0))), e^{\eta t}\bm{h}\in W^{2,1}_{q,p}(Q^\infty_F), e^{\eta t}\bm{g}_1 \in L^p(0,\infty;\R^3)$ and $e^{\eta t}\bm{g}_2 \in L^p(0,\infty,\R^3)$ satisfying
	$$
\bm{h}\arrowvert_{t=0} = 0 \quad \text{ on } \ \Omega_F(0) \quad \text{ and } \quad \bm{h}\cdot \vn\rvert_{\partial\Omega_F(0)} = \bm{0},
	$$
	the system (\ref{2}) admits a unique strong solution
	\begin{equation*}
	\begin{aligned}
	&e^{\eta t}\vu \in W^{2,1}_{q,p}(Q^\infty_F),\quad e^{\eta t}\pi\in L^p(0,\infty;W^{1,q}(\Omega_F(0)))\\
	&e^{\eta t}\bm{l}\in W^{1,p}(0,\infty;\R^3), \quad e^{\eta t}\bm{\omega}\in W^{1,p}(0,\infty;\R^3).
	\end{aligned}
	\end{equation*}
	Moreover, there exists a constant $C_L>0$, depending only on $\alpha, p, q$ and $\Omega_S(0)$ such that
	\begin{equation}
	\label{28}
	\begin{aligned}
	&\|e^{\eta t}\vu\|_{W^{2,1}_{q,p}(Q^\infty_F)} + \|e^{\eta t}\pi\|_{L^p(0,\infty;W^{1,q}(\Omega_F(0)))} + \|e^{\eta t}\bm{l}\|_{L^p(0,\infty;\R^3)} + \|e^{\eta t}\bm{\omega}\|_{L^p(0,\infty;\R^3)} \\
	\le & \ C_L \left( \|\vu_0\|_{B_{q,p}^{2(1-1/p)}(\Omega_F(0))} + \|\bm{l}_0\|_{\R^3} + \|\bm{\omega}_0\|_{\R^3} + \|e^{\eta t}\bm{f}\|_{L^p(0,\infty;\bm{L}^q(\Omega_F(0)))}\right.\\
	&\qquad \qquad \quad \left. + \|e^{\eta t}\bm{h}\|_{W^{2,1}_{q,p}(Q^\infty_F)} + \|e^{\eta t}\bm{g}_1\|_{L^p(0,\infty;\R^3)} +\|e^{\eta t}\bm{g}_2\|_{L^p(0,\infty;\R^3)}\right) .
	\end{aligned}
	\end{equation}
\end{theorem}

\begin{proof}
	We first consider the case $\eta =0$. Let us set $\bm{v} := \vu - \bm{h}$. Then $(\bm{v}, \pi, \bm{l}, \bm{\omega})$ satisfies the following system
	\begin{equation}
	\label{25}
	\begin{cases}
	\begin{aligned}
	&\bm{v}_t - \div \ \sigma (\bm{v},\pi) = \bm{F} &&\ \text{ in } \ \Omega_F(0)\times (0,T),\\
	&\div \ \bm{v} = 0 &&\ \text{ in } \ \Omega_F(0)\times (0,T),\\
	&\bm{v}=\bm{0} &&\ \text{ on } \ \partial \Omega \times (0,T),\\
	&\bm{v}\cdot \vn=(\bm{l} + \bm{\omega} \times \bm{y})\cdot \vn &&\ \text{ on } \ \partial \Omega_S(0) \times (0,T),\\
	&2\left[(\DT\bm{v})\vn\right]_{\vt}+\alpha\bm{v}_{\vt}=\alpha(\bm{l} + \bm{\omega} \times \bm{y})_{\vt} &&\ \text{ on } \ \partial \Omega_S(0) \times (0,T),\\
	&m \bm{l}' = - \int\displaylimits_{\partial \Omega_S(0)}{\sigma (\bm{v},\pi)\vn} + \bm{G}_1, &&\ \ t\in (0,T),\\
	&J(0) \bm{\omega}' = - \int\displaylimits_{\partial \Omega_S(0)}{\bm{y}\times \sigma (\bm{v},\pi)\vn}+\bm{G}_2, &&\ \ t\in (0,T),\\
	&\bm{v}(0) = \vu_0  &&\ \text{ in } \ \Omega_F(0),\\
	&\bm{l}( 0) = \bm{l}_0, \quad \bm{\omega}(0) = \bm{\omega}_0
	\end{aligned}
	\end{cases}
	\end{equation}
	where
	$$
	\bm{F} = \bm{f} -\partial_t \bm{h} + \Delta \bm{h}, \quad \bm{G}_1 = \bm{g}_1 - \int\displaylimits_{\partial \Omega_S(0)}{(\DT\bm{h})\vn}, \quad \bm{G}_2 = \bm{g}_2 - \int\displaylimits_{\partial \Omega_S(0)}{\bm{y}\times (\DT\bm{h})\vn}.
	$$
	Under the hypothesis of the theorem, we have that $(\bm{F},\bm{G}_1, \bm{G}_2)\in L^p(0,\infty; X)$ and
	\begin{equation}
	\label{26}
	\begin{aligned}
	\|(\bm{F},\bm{G}_1, \bm{G}_2)\|_{L^p(0,\infty; X)} \le C &\left( \|\bm{f}\|_{L^p(0,\infty;\bm{L}^q(\Omega_F(0)))} + \|\bm{h}\|_{W^{2,1}_{q,p}(Q^\infty_F)}\right.\\
	& \qquad \qquad \qquad \qquad \left.+ \|\bm{g}_1\|_{L^p(0,\infty;\R^3)} + \|\bm{g}_2\|_{L^p(0,\infty;\R^3)} \right).
	\end{aligned}
	\end{equation}
	Also, $(\vu_0, \bm{l}_0, \bm{\omega}_0)\in \left( X, \mathcal{D}(\mathcal{A}_{FS}) \right)_{1-1/p,p}$ from the assumptions of the theorem. Hence, the system (\ref{26}) admits a unique solution $(\bm{v}, \bm{l}, \bm{\omega})\in L^p(0,\infty;\mathcal{D}(\mathcal{A}_{FS}))\cap W^{1,p}(0,\infty;X)$ by Theorem \ref{T11} which also satisfies
	\begin{equation}
	\label{27}
	\begin{aligned}
	& \|(\bm{v}, \bm{l}, \bm{\omega})\|_{L^p(0,\infty;\mathcal{D}(\mathcal{A}_{FS}))} + \|(\bm{v}, \bm{l}, \bm{\omega})\|_{W^{1,p}(0,\infty;X)}\\ \le & \ C \left(\|(\vu_0, \bm{l}_0, \bm{\omega}_0)\|_{\left( X, \mathcal{D}(\mathcal{A}_{FS}) \right)_{1-1/p,p}} + \|(\bm{F},\bm{G}_1, \bm{G}_2)\|_{L^p(0,\infty; X)} \right).
	\end{aligned}
	\end{equation}
	Thus, $\vu = \bm{v}+ \bm{h}, \pi, \bm{l}, \bm{\omega}$ is the unique solution of (\ref{2}). The estimate (\ref{28}) follows combining (\ref{27}) and (\ref{26}).
	
	Finally the result for $\eta >0$ can be deduced from the previous case, simply multiplying all the functions by $e^{\eta t}$ and noting that $\mathcal{A}_{FS} + \eta I$ also generates an $C^0$-semigroup of negative type for all $\eta\in (0,\eta_0)$.
	\hfill
\end{proof}

\section{Non-linear problem}
\setcounter{equation}{0}

In order to handle the full non-linear coupled system, it is usual to make the time dependent fluid domain (and the moving interface) fixed. Therefore we use the change of variable, as used in \cite[section 2]{cumsille}, which coincide with $Q(t)\bm{y} + \bm{h}(t)$ in a neighbourhood of the rigid body and is equal to the identity far from the rigid body, to rewrite the coupled system in a fixed spatial domain. For the convenience of the reader, we summarized the construction and basic properties of the change of variable in the Appendix.

\subsection{Estimates on the non-linear terms}
\label{S6.2}

In the first part of this subsection, we show estimates on the transforms $X$ and $Y$ in terms of $\tilde{\bm{l}}, \tilde{\bm{\omega}}$. Then we can study the Lipschitz properties of $\bm{F}_0, \bm{H}, \bm{F}_1, \bm{F}_2$ (see \eqref{36}-\eqref{24} for definitions of these functions). For $p\in(1,\infty)$, let $p'$ denote the conjugate of $p$, that is, $\frac{1}{p}+ \frac{1}{p'}=1$.

We also introduce the set
\begin{equation}
\label{S_gamma1}
S_\gamma := \{(\tilde{\vu},\tilde{\pi},\tilde{\bm{l}},\tilde{\bm{\omega}}): \|(\tilde{\vu},\tilde{\pi},\tilde{\bm{l}},\tilde{\bm{\omega}})\|_S\le \gamma \}
\end{equation}
where
\begin{equation*}
\begin{aligned}
\|(\tilde{\vu},\tilde{\pi},\tilde{\bm{l}},\tilde{\bm{\omega}})\|_S:= &\|e^{\eta(\cdot)}\tilde{\vu}\|_{L^p(0,\infty;\bm{W}^{2,q}(\Omega_F(0)))}+ \|e^{\eta(\cdot)}\tilde{\vu}\|_{W^{1,p}(0,\infty;\bm{L}^q(\Omega_F(0)))}\\
&+ \|e^{\eta(\cdot)}\tilde{\pi}\|_{L^p(0,\infty;W^{1,	q}(\Omega_F(0)))}+ \|e^{\eta(\cdot)}\tilde{\bm{l}}\|_{W^{1,p}(0,\infty;\R^3)}+ \|e^{\eta(\cdot)}\tilde{\bm{\omega}}\|_{W^{1,p}(0,\infty;\R^3)} .
\end{aligned}
\end{equation*}

\begin{proposition}
	Let $p,q\in (1,\infty)$. There exists constants $\gamma_0\in (0,1)$ and $C>0$, depending only on $p,q, \eta$ and $\Omega_F(0)$ such that for every $\gamma\in (0,\gamma_0)$ and every $(\tilde{\bm{l}}, \tilde{\bm{\omega}}) $ with $\|e^{\eta(\cdot)}\tilde{\bm{l}}\|_{W^{1,p}(0,\infty;\R^3)}+ \|e^{\eta(\cdot)}\tilde{\bm{\omega}}\|_{W^{1,p}(0,\infty;\R^3)}\le \gamma$,
	\begin{align}
	\|Q - I_3\|_{L^\infty(0,\infty;\R^{3\times 3})} &\le C \ \gamma; \label{18}\\
	\|J_X - I_3\|_{L^\infty(0,\infty;C^2(\overline{\Omega})} &\le C \ \gamma; \label{17}\\
	\|J_Y - I_3\|_{L^\infty(0,\infty;C^2(\overline{\Omega})} &\le C \ \gamma; \label{21}\\
	\|\frac{\partial^2 Y}{\partial x_j \partial x_k} \|_{L^\infty(0,\infty;C^2(\overline{\Omega})} &\le C \ \gamma; \label{22}\\
	\|\partial_t X \|_{L^\infty(0,\infty;\Omega_F(0))} &\le C \ \gamma \label{20}\\
	\|J_Y Q - I_3\|_{L^\infty(0,\infty; C^2(\overline{\Omega}))} &\le C \ \gamma, \label{47}
	\end{align}
where $X, Y, J_X, J_Y$ are defined in Section \ref{appendix}.
\end{proposition}

\begin{proof}
	First we show the existence of a constant $\gamma_0\in (0,1)$ such that for every $\gamma\in (0,\gamma_0)$ and for every $(\tilde{\bm{l}}, \tilde{\bm{\omega}}) $ with $\|e^{\eta(\cdot)}\tilde{\bm{l}}\|_{W^{1,p}(0,\infty;\R^3)}+ \|e^{\eta(\cdot)}\tilde{\bm{\omega}}\|_{W^{1,p}(0,\infty;\R^3)}\le \gamma$, the condition (\ref{14}) is verified.
	
	The matrix $Q$ being an orthogonal linear transformation satisfies $Q\in SO(3)$ and thus $|Q(t)| = 1$ for all $t\ge 0$. Here $|A|$ denotes the Frobenius norm (or, Euclidean norm) for any matrix $A$. Since $Q$ satisfies the problem (\ref{15}), we can write
	$$
	Q(t)\bm{a} = \bm{a} +\int\displaylimits^t_0{e^{-\eta s} e^{\eta s} Q(s) \left( \tilde{\omega}(s)\times \bm{a}\right) \; \mathrm{d}s}
	$$
	which gives
	$$
	|Q(t)\bm{a} - \bm{a}| \le \int\displaylimits^t_0{e^{-\eta s} e^{\eta s} |\tilde{\omega}(s)\times \bm{a}| \; \mathrm{d}s} .
	$$
	Therefore, we can estimate the operator norm
	$$
	\|Q(t)-I_3\| = \underset{0\neq \bm{a}\in\R^3}{\sup} \frac{|Q(t)\bm{a}- \bm{a}|}{|\bm{a}|} \le \int\displaylimits^t_0{e^{-\eta s} e^{\eta s} |\tilde{\omega}(s)| \; \mathrm{d}s}.
	$$
	But as the operator norm and the Frobenius norm is equivalent on the matrix space, we can have
	$$
	\|Q(t) - I_3\|_{\R^{3\times 3}} \le \sqrt{3} \|Q(t)-I_3\|\le \sqrt{3} \int\displaylimits^t_0{e^{-\eta s} e^{\eta s} |\tilde{\omega}(s)| \; \mathrm{d}s}.
	$$
	Consequently,
	\begin{equation*}
	\begin{aligned}
	\|Q-I_3\|_{L^\infty(0,\infty;\R^{3\times 3})} & \le \sqrt{3} \int\displaylimits^\infty_0{e^{-\eta s} e^{\eta s} |\tilde{\omega}(s)| \; \mathrm{d}s}\\
	& \le \sqrt{3} \left( \int\displaylimits^\infty_0{e^{-p'\eta s} \; \mathrm{d}s}\right) ^{1/p'} \|e^{\eta(\cdot)}\tilde{\bm{\omega}}\|_{L^p(0,\infty;\R^3)} \le \left( \frac{1}{p'\eta}\right) ^{1/p'}\gamma .
	\end{aligned}
	\end{equation*}
	Similarly, we can write from $(\ref{13})_3$,
	\begin{equation*}
	\begin{aligned}
	\|\bm{h}\|_{L^\infty(0,\infty;\R^3)} & \le \int\displaylimits_0^\infty {e^{-\eta s}e^{\eta s}|Q(s)| |\tilde{\bm{l}}(s)|\, \mathrm{d}s} \\
	& \le  \left( \int\displaylimits^\infty_0{e^{-p'\eta s} \; \mathrm{d}s}\right) ^{1/p'} \|e^{\eta(\cdot)}\tilde{\bm{l}}\|_{L^p(0,\infty;\R^3)} \le \left( \frac{1}{p'\eta}\right) ^{1/p'}\gamma .
	\end{aligned}
	\end{equation*}
	Combining the above two inequalities give
	$$
	\|Q-I_3\|_{L^\infty(0,\infty;\R^{3\times 3})} \mathrm{diam} (\Omega_S(0)) + \|\bm{h}\|_{L^\infty(0,\infty;\R^3)} \le \left( \frac{1}{p'\eta}\right) ^{1/p'}\gamma \left(1 + \mathrm{diam} (\Omega_S(0)) \right) .
	$$
	Let us define
	\begin{equation}
	\label{gamma0}
	\gamma_0 = \min{1,\frac{\beta}{2C_{p,\eta}\left(1 + \mathrm{diam} (\Omega_S(0)) \right)}} \quad \text{ where } C_{p,\eta} = \left( \frac{1}{p'\eta}\right) ^{1/p'}.
	\end{equation}
	With this choice of $\gamma_0$, we satisfy the condition (\ref{14}).
	
	Next we prove some regularity of $X, Y, J_X, J_Y$. The mapping $X$, solution of the differential equation (\ref{12}) can be written as
	$$
	X(\bm{y},t) = \bm{y} + \int\displaylimits_0^t{\Lambda(X(\bm{y},s),s)\, \mathrm{d}s}.
	$$
	Differentiating it with respect to $\bm{y}$, we obtain,
	$$
	J_X(\bm{y},t) = I_3 + \int\displaylimits_0^t{\nabla_{\bm{y}} \Lambda(X(\bm{y},s),s)J_ X(\bm{y},s)\,\mathrm{d}s}.
	$$
	Note that from the definition of $\Lambda$, we can write, for all $\bm{x}\in\Omega_S(t)$ (cf. Lemma \ref{lem1}),
	$$
	\nabla_{\bm{x}} \Lambda (\bm{x},t) =
	\begin{pmatrix}
	0 & -\omega_3(t) & \omega_2(t)\\
	\omega_3(t) & 0 & -\omega_1(t)\\
	-\omega_2(t) & \omega_1(t) & 0
	\end{pmatrix},
	$$
where $\bm{\omega}(t) = \left( \omega_1, \omega_2, \omega_3\right) $. Also, $\nabla_{\bm{x}} \Lambda =0$ for all $\bm{x}$ with $\mathrm{dist} (\bm{x},\partial\Omega)<\beta/8$. Otherwise,
	$$
	\|\nabla_{\bm{y}} \Lambda (\bm{x},t)\|_{C^2(\overline{\Omega})} \le C \left( |\tilde{\bm{\omega}}(t)| + |\tilde{\bm{l}}(t)|\right)
	$$
	where the constant $C$ depends on $\Omega_F(0)$.
	So we have,
	\begin{equation*}
	\begin{aligned}
	& \|J_X(\cdot, t)\|_{C^2(\overline{\Omega})}\\
	\le & \ 1 + C\int\displaylimits_0^t{e^{-\eta s} e^{\eta s}\left( |\tilde{\bm{\omega}}(s)| + |\tilde{\bm{l}}(s)|\right) \|J_X(\cdot, s)\|_{C^2(\overline{\Omega})} \, \mathrm{d}s}\\
	\le & \ 1 + C\left( \|e^{\eta(\cdot)}\tilde{\bm{\omega}}\|_{L^\infty(0,\infty;\R^3)} + \|e^{\eta(\cdot)}\tilde{\bm{l}}\|_{L^\infty(0,\infty;\R^3)} \right) \int\displaylimits_0^t{e^{-\eta s} \|J_X(\cdot, s)\|_{C^2(\overline{\Omega})} \, \mathrm{d}s}\\
	\le & \ 1+ C \int\displaylimits_0^t {e^{-\eta s}\|J_X(\cdot, s)\|_{C^2(\overline{\Omega})}\,\mathrm{d}s}.
	\end{aligned}
	\end{equation*}
	Now the Gronwall's inequality yields,
	$$
	\|J_X(\cdot, t)\|_{C^2(\overline{\Omega})}\le \mathrm{exp}\left(C\int\displaylimits_0^t {e^{-\eta s}\,\mathrm{d}s}\right)  \le e^{C/\eta} \quad \forall \ t\in (0,\infty).
	$$
	With this estimate at hand, we obtain
	\begin{equation*}
	\|J_X- I_3\|_{L^\infty(0,\infty;C^2(\overline{\Omega}))}\le C \int\displaylimits^\infty_0{e^{-\eta s} e^{\eta s}\left( |\tilde{\bm{\omega}}(s)| + |\tilde{\bm{l}}(s)|\right)\, \mathrm{d}s} \le C \ \gamma.
	\end{equation*}
	Also from the following relation, since $\mathrm{det} J_X = 1$,
	$$
	(\mathrm{cof} \ J_X)^T = (\mathrm{det} J_X) J_X^{-1}= J_X^{-1},
	$$
	we can deduce,
	\begin{equation*}
	\begin{aligned}
	\|\mathrm{cof} \ J_X\|_{L^\infty(0,\infty; C^2 (\overline{\Omega}))} \le C
	\end{aligned}
	\end{equation*}
	which follows from Lemma \ref{A1}. This implies also,
	\begin{equation}
	\label{19}
	\|J_Y\|_{L^\infty(0,\infty; C^2 (\overline{\Omega}))} \le C
	\end{equation}
	since $J_Y = J_X^{-1}$. Using the above estimate and (\ref{17}), we further get
	\begin{equation*}
	\|J_Y - I_3\|_{L^\infty(0,\infty; C^2 (\overline{\Omega}))}\le \|J_Y\|_{L^\infty(0,\infty; C^2 (\overline{\Omega}))} \|J_X-I_3\|_{L^\infty(0,\infty; C^2 (\overline{\Omega}))} \le C \ \gamma.
	\end{equation*}
	To deduce the regularity of $\frac{\partial^2 Y}{\partial x_j \partial x_k}$, we write
	\begin{equation*}
	\begin{aligned}
	\frac{\partial}{\partial x_j}J_Y = \frac{\partial}{\partial y_i}\left( \mathrm{cof} \ J_X \right) \frac{\partial Y_i}{\partial x_j}
	\end{aligned}
	\end{equation*}
	from which it follows, along with the definition (\ref{23}) of $\Lambda$, since $\mathrm{cof} \ J_X$ involves second derivative of $\Lambda$,
	\begin{equation*}
	\|\frac{\partial}{\partial x_j}J_Y\|_{L^\infty(0,\infty;C^2(\overline{\Omega}))} \le C  \sup_{t\in(0,\infty)} \left(|\tilde{\bm{\omega}}(t)| + |\tilde{\bm{l}}(t)|\right) \le C \ \gamma.
	\end{equation*}
	Similarly, we obtain the following estimates
	\begin{equation*}
	\begin{aligned}
	\|\partial_t X\|_{L^\infty(0,\infty;\Omega_F(0))}& = \|\Lambda\|_{L^\infty(0,\infty;\Omega_F(0))}\\
	&\le C  \sup_{t\in(0,\infty)} \left(e^{-\eta t} e^{\eta t}( |\tilde{\bm{\omega}}(t)| + |\tilde{\bm{l}}(t)|)\right)\\
	&\le C\left(\|e^{\eta(\cdot)}\tilde{\bm{\omega}}\|_{L^\infty(0,\infty,\R^3)} + \|e^{\eta(\cdot)}\tilde{\bm{l}}\|_{L^\infty(0,\infty;\R^3)} \right)\le C \ \gamma;
	\end{aligned}
	\end{equation*}
	\begin{equation*}
	\|\partial_t J_Y\|_{L^\infty(0,\infty;\Omega_F(0))} = \|J_X^{-1} \ \partial_t J_X \ J_X^{-1}\|_{L^\infty(0,\infty;\Omega_F(0))} = \|J_Y \nabla \Lambda \ J_Y\|_{L^\infty(0,\infty;\Omega_F(0))} \le C \ \gamma;
	\end{equation*}
	\begin{equation*}
	\begin{aligned}
	\|J_Y Q - I_3\|_{L^\infty(0,\infty;C^2(\overline{\Omega}))}&\le \|J_Y - Q^T\|_{L^\infty(0,\infty;C^2(\overline{\Omega}))} \\
	&\le \|J_Y - I_3\|_{L^\infty(0,\infty;C^2(\overline{\Omega}))} + \|Q-I_3\|_{L^\infty(0,\infty;\R^{3\times 3})} \le C \ \gamma.
	\end{aligned}
	\end{equation*}
	\hfill
\end{proof}

\begin{lemma}
	\label{A1}
	If $\bm{f}(\bm{x},t)$ belongs to $L^\infty(0,\infty;C^2(\Omega))$ with $\bm{f}(\bm{x},t) \ge m >0$ in $\Omega\times (0,\infty)$
	then $1/\bm{f}$ belongs to $L^\infty(0,\infty;C^2(\Omega))$,
	\begin{equation}
	\label{16}
	\|1/\bm{f}\|_{L^\infty(0,\infty;C^2(\Omega))} \le C \|\bm{f}\|_{L^\infty(0,\infty;C^2(\Omega))}.
	\end{equation}
\end{lemma}

\begin{proof}
	Let $G\in C^\infty(\R)$, non-negative such that $G(0) =0$ and $G(r)=1/r$ for $|r|\ge m$. Since the derivative of $G$ is bounded and $G(0)=0$, by the Mean value theorem, we have
	$$
	|G(s)|\le M s \quad \forall \ s\in\R.
	$$
	Thus $|G(\bm{f}(\bm{x},t))|\le M |\bm{f}(\bm{x},t)|$ for every $\bm{x}\in\Omega$ which implies (\ref{16}).
	\hfill
\end{proof}

The following general embedding of $W^{2,1}_{q,p}(Q_F^\infty)$ is needed to cope with the gradient terms. It mainly relies on the mixed derivative theorem, followed by Sobolev embedding.

\begin{lemma}$\cite[Lemma \ 4.2]{DGH}$
\label{L2}
Let $\Omega_F(0)$ be a $\HC{1}{1}$ domain with compact boundary, $p,q\in(1,\infty)$, $\theta\in (0,1)$ and $T>0$. Also assume that $s=0$ or $s=1$ and $k,m\in(1,\infty)$ obeys $\frac{2-s}{2}+\frac{3}{2m} - \frac{3}{2q}\ge \frac{1}{p}- \frac{1}{q}$. Then
\begin{equation*}
W^{2,1}_{q,p}(Q_F^\infty)\hookrightarrow W^{\theta,p}(0,T;W^{2-2\theta,q}(\Omega_F(0)))\hookrightarrow L^k(0,T;W^{s,m}(\Omega_F(0))).
\end{equation*}
\end{lemma}

Now we are in the position to estimate the non-linear terms.

\begin{proposition}
	\label{P1}
	Let assume $p,q\in (1,\infty)$ satisfying the condition $\frac{1}{p} + \frac{3}{2q}\le \frac{3}{2}$. There exist constants $\gamma_0\in (0,1)$ and $C_N>0$, depending only on $p,q, \eta$ and $\Omega_F(0)$ such that for every $\gamma\in (0,\gamma_0)$ and for every $(\tilde{\vu},\tilde{\pi},\tilde{\bm{l}},\tilde{\bm{\omega}})\in S_\gamma$, we have
	\begin{equation}
	\label{32}
	\begin{aligned}
	\|e^{\eta(\cdot)}\bm{F}_0\|_{L^p(0,\infty; \bm{L}^q(\Omega_F(0)))} &+ \|e^{\eta(\cdot)}\bm{H}\|_{W^{2,1}_{q,p}(Q_F^\infty)}\\
	& + \|e^{\eta(\cdot)}\bm{F}_1\|_{L^p(0,\infty;\R^3)}+ \|e^{\eta(\cdot)}\bm{F}_2\|_{L^p(0,\infty;\R^3)}\le C_N \gamma^2,
	\end{aligned}
	\end{equation}
where $S_\gamma$ is defined in (\ref{S_gamma1}) and $\bm{F}_0, \bm{F}_1, \bm{F}_2, \bm{H}$ are defined in (\ref{36})-(\ref{24}).

	Moreover, there exists a constant $C_{\mathrm{lip}}>0$, depending only on $p, q, \eta$ and $\Omega_F(0)$ such that for every $(\tilde{\vu}^i, \tilde{\pi}^i, \tilde{\bm{l}}^i, \tilde{\bm{\omega}}^i)\in S_\gamma, i =1,2$,
	\begin{equation}
	\label{31}
	\begin{aligned}
	& \quad \|e^{\eta(\cdot)}\bm{F}_0(\tilde{\vu}^1, \tilde{\pi}^1, \tilde{\bm{l}}^1, \tilde{\bm{\omega}}^1) - \bm{F}_0(\tilde{\vu}^2, \tilde{\pi}^2, \tilde{\bm{l}}^2, \tilde{\bm{\omega}}^2)\|_{L^p(0,\infty; \bm{L}^q(\Omega_F(0)))}\\
	&+ \|e^{\eta(\cdot)}\bm{H}(\tilde{\vu}^1,\tilde{\bm{l}}^1,\tilde{\bm{\omega}}^1) - e^{\eta(\cdot)}\bm{H}(\tilde{\vu}^2,\tilde{\bm{l}}^2, \tilde{\bm{\omega}}^2)\|_{W^{2,1}_{q,p}(Q_F^\infty)}\\
	&+ \|e^{\eta(\cdot)}\bm{F}_1(\tilde{\bm{l}}^1, \tilde{\bm{\omega}}^1) - e^{\eta(\cdot)}\bm{F}_1(\tilde{\bm{l}}^2, \tilde{\bm{\omega}}^2)\|_{L^p(0,\infty;\R^3)}+ \|e^{\eta(\cdot)}\bm{F}_2( \tilde{\bm{\omega}}^1) - e^{\eta(\cdot)}\bm{F}_2( \tilde{\bm{\omega}}^2)\|_{L^p(0,\infty;\R^3)}\\
	&\le C_{\mathrm{lip}} \gamma \ \|(\tilde{\vu}^1, \tilde{\pi}^1, \tilde{\bm{l}}^1, \tilde{\bm{\omega}}^1) - (\tilde{\vu}^2, \tilde{\pi}^2, \tilde{\bm{l}}^2, \tilde{\bm{\omega}}^2)\|_{S}.
	\end{aligned}
	\end{equation}
\end{proposition}

\begin{proof}
	\textbf{Estimate of $\bm{F}_0$}:
	\begin{equation*}
	\|e^{\eta(\cdot)}\bm{F}_0\|_{L^p(0,\infty;\bm{L}^q(\Omega_F(0)))}\le C \ \gamma^2 .
	\end{equation*}
For the first three terms of $\bm{F}_0$, with the help of the estimates (\ref{18}), (\ref{19}) and (\ref{20}), we get,
	\begin{equation*}
	\begin{aligned}
	& \|e^{\eta(\cdot)}\left( \left( (I_3-Q)\partial_t \tilde{\vu}\right) _i - \left( Q(\tilde{\bm{\omega}}\times\tilde{\vu}) \right) _i -(\partial_t X\cdot J_Y^T \nabla)(Q\tilde{\vu})_i\right) \|_{L^p(0,\infty;\bm{L}^q(\Omega_F(0)))}\\
	\le & \ C \left( \|I_3-Q\|_{L^\infty(0,\infty;\R^{3\times 3})} + \|\tilde{\bm{\omega}}\|_{L^\infty(0,\infty;\R^3)}\right.\\
	& \qquad \qquad \qquad \qquad \quad \quad \left. + \|\partial_t X\|_{L^\infty(0,\infty;\Omega_F(0))} \|J_Y\|_{L^\infty(0,\infty;C^2(\overline{\Omega}))}\right) \|e^{\eta(\cdot)}\tilde{\vu}\|_{W^{2,1}_{q,p}(Q^\infty_F)}\\
	\le & \ C \ \gamma^2.
	\end{aligned}
	\end{equation*}
For the fourth term of $\bm{F}_0$, the estimate (\ref{19}) and the fact that $|Q(t)| \le 1$ for all $t\ge 0$ yield,
	\begin{equation*}
	\begin{aligned}
	& \| e^{\eta(\cdot)} \left( (Q\tilde{\vu})\cdot(J_Y^T\nabla)\right) (Q\tilde{\vu})_i\|_{L^p(0,\infty;\bm{L}^q(\Omega_F(0)))}\\
	\le & \ C \|e^{\eta(\cdot)}\tilde{\bm{u}}\|_{L^{3p}(0,\infty;\bm{L}^{3q}(\Omega_F(0)))}\|\nabla \tilde{\vu}\|_{L^{3p/2}(0,\infty;\bm{L}^{3q/2}(\Omega_F(0)))}.
	\end{aligned}
	\end{equation*}
	But due to the condition $\frac{1}{p}+ \frac{3}{2q}\le \frac{3}{2}$, the following embeddings
	$$
	W^{2,1}_{q,p}(Q^\infty_F) \hookrightarrow L^{3p}(0,\infty;\bm{L}^{3q}(\Omega_F(0))) \quad \text{ and } \quad W^{2,1}_{q,p}(Q^\infty_F) \hookrightarrow L^{3p/2}(0,\infty;\bm{L}^{3q/2}(\Omega_F(0)))
	$$
hold from lemma \ref{L2}. Thus we obtain,
	$$
	\| e^{\eta(\cdot)} \left( (Q\tilde{\vu})\cdot(J_Y^T\nabla)\right) (Q\tilde{\vu})_i\|_{L^p(0,\infty;\bm{L}^q(\Omega_F(0)))} \le C \ \gamma^2.
	$$
Concerning the fifth term of $\bm{F}_0$, it follows from (\ref{22}) that
	\begin{equation*}
	\begin{aligned}
	&\|e^{\eta(\cdot)} \sum_{m,l,j} \frac{\partial (Q\tilde{\vu})_i}{\partial y_l} \frac{\partial Y_m}{\partial x_j} \frac{\partial}{\partial y_m}\left( \frac{\partial Y_l}{\partial x_j}\right) \|_{L^p(0,\infty;\bm{L}^q(\Omega_F(0)))}\\
	\le & \ C \|\frac{\partial^2 Y}{\partial x_j \partial x_k}\|_{L^\infty(0,\infty;C^2(\overline{\Omega}))} \|e^{\eta(\cdot)}\tilde{\vu}\|_{W^{2,1}_{q,p}(Q^\infty_F)} \le C \ \gamma^2.
	\end{aligned}
	\end{equation*}
For the sixth and seventh term of $\bm{F}_0$, the two terms can be re-written as,
	\begin{equation*}
	\begin{aligned}
	&\sum_{m,l,j}\frac{\partial^2 (Q\tilde{\vu})_i}{\partial y_m \partial y_l} \frac{\partial Y_l}{\partial x_j} \frac{\partial Y_m}{\partial x_j} - \Delta \tilde{\vu}_i\\
	= & \sum_{m,l,j}\frac{\partial^2 (Q\tilde{\vu})_i}{\partial y_m \partial y_l} \left( \frac{\partial Y_l}{\partial x_j} - \delta_{lj}\right) \frac{\partial Y_m}{\partial x_j} + \sum_{m,l}\frac{\partial^2 (Q\tilde{\vu})_i}{\partial y_m \partial y_l} \left( \frac{\partial Y_m}{\partial x_l} - \delta_{ml}\right) + \left( (Q- I_3)\Delta \tilde{\vu}\right)_i
	\end{aligned}
	\end{equation*}
	Therefore, it follows from (\ref{18}), (\ref{19}) and (\ref{21}) that
	\begin{equation*}
	\begin{aligned}
	&\|e^{\eta(\cdot)} \left( \sum_{m,l,j}\frac{\partial^2 (Q\tilde{\vu})_i}{\partial y_m \partial y_l} \frac{\partial Y_l}{\partial x_j} \frac{\partial Y_m}{\partial x_j} - \Delta \tilde{\vu}_i\right)  \|_{L^p(0,\infty;\bm{L}^q(\Omega_F(0)))}\\
	\le & \ C \left( \|J_Y - I_3\|_{L^\infty(0,\infty;C^2(\overline{\Omega}))} + \|Q-I_3\|_{L^\infty(0,\infty;\R^{3\times 3})}\right) \|e^{\eta(\cdot)}\tilde{\vu}\|_{W^{2,1}_{q,p}(Q^\infty_F)}\le C \ \gamma^2 .
	\end{aligned}
	\end{equation*}
Finally, the following estimate for the last term of $\bm{F}_0$ follows from (\ref{21}),
	\begin{equation*}
	\begin{aligned}
	&\|e^{\eta(\cdot)}\left( (I_3 -J^T_Y)\nabla\tilde{\pi}\right) _i\|_{L^p(0,\infty;\bm{L}^q(\Omega_F(0)))}\\
	\le & \|I_3-J_Y\|_{L^\infty(0,\infty;C^2(\overline{\Omega}))} \|e^{\eta(\cdot)}\nabla\tilde{\pi}\|_{L^p(0,\infty;L^q(\Omega_F(0)))} \le C \ \gamma^2.
	\end{aligned}
	\end{equation*}
	\textbf{Estimate of $\bm{H}$}:
	\begin{equation*}
	\|e^{\eta(\cdot)}\bm{H}\|_{W^{2,1}_{q,p}(Q^\infty_F)} \le C \ \gamma^2.
	\end{equation*}
	We obtain from (\ref{47}),
	$$
	\|e^{\eta(\cdot)}(I_3-J_YQ) \tilde{\bm{u}}\|_{W^{2,1}_{q,p}(Q^\infty_F)} \le \|I_3-J_YQ\|_{L^\infty(0,\infty;C^2(\overline{\Omega}))} \|e^{\eta(\cdot)}\tilde{\vu}\|_{W^{2,1}_{q,p}(Q^\infty_F)} \le C \ \gamma^2.
	$$
	\textbf{Estimate of $\bm{F}_1$ and $\bm{F}_2$}: From the expressions of $\bm{F}_1$ and $\bm{F}_2$, it is obvious to see
	\begin{equation*}
	\|e^{\eta(\cdot)}\bm{F}_1\|_{L^p(0,\infty;\R^3)} + \|e^{\eta(\cdot)}\bm{F}_2\|_{L^p(0,\infty;\R^3)} \le C \ \gamma^2.
	\end{equation*}
	This completes the proof of the estimate (\ref{32}).
	
	The Lipschitz property can be proved in the same way.
	\hfill
\end{proof}

\subsection{Fixed point argument}

\begin{theorem}
	\label{T9}
	Let $\Omega_F(0)$ be a bounded domain of class $\HC{2}{1}$, $p,q\in(1,\infty)$ satisfy the condition $\frac{1}{p}+ \frac{3}{2q}\le \frac{3}{2}$ and $\alpha\ge 0$ be as in (\ref{alpha}). Let $\eta\in (0,\eta_0)$ where $\eta_0$ is the constant introduced in Theorem \ref{T1}. Then there exist a constant $\tilde{\gamma}>0$ depending only on $p,q,\eta$ and $\Omega_F(0)$ such that for all $\gamma\in (0,\tilde{\gamma})$ and for all $(\vu_0,\bm{l}_0,\bm{\omega}_0)\in B^{2(1-1/p)}_{q,p}(\Omega_F(0))\times \R^3 \times \R^3$ satisfying the compatibility condition (\ref{CN})
	and
	\begin{equation}
	\label{29}
	\|\vu_0\|_{B^{2(1-1/p)}_{q,p}(\Omega_F(0))}+\|\bm{l}_0\|_{\R^3}+\|\bm{\omega}_0\|_{\R^3}\le \frac{\gamma}{2C_L},
	\end{equation}
	where $C_L$ is the continuity constant appeared in Theorem \ref{T4}, the system (\ref{fixed_FSI})-(\ref{24}) admits a unique global in time strong solution $(\tilde{\vu}, \tilde{\pi}, \tilde{\bm{l}}, \tilde{\bm{\omega}})$ such that
	\begin{equation*}
	\|(\tilde{\vu}, \tilde{\pi}, \tilde{\bm{l}}, \tilde{\bm{\omega}})\|_{S} \le \gamma .
	\end{equation*}
\end{theorem}

\begin{proof}
	Let us define
	$$
	\tilde{\gamma} = \min\left\lbrace\gamma_0, \frac{1}{2C_L C_N}, \frac{1}{2C_L C_{\mathrm{lip}}}\right\rbrace
	$$
	where $\gamma_0$ is defined as in (\ref{gamma0}) and $C_L, C_N, C_{lip}$ are the constants appearing in Theorem \ref{T4} and Proposition \ref{P1}. It can be shown with the help of the estimates, obtained in the previous subsection, that the mapping
	$$
	\mathcal{N}: (\bm{v}, \varphi, \bm{\kappa}, \bm{\tau}) \mapsto (\tilde{\vu}, \tilde{\pi}, \tilde{\bm{l}}, \tilde{\bm{\omega}})
	$$
	which maps $(\bm{v}, \varphi, \bm{\kappa}, \bm{\tau}) \in S_\gamma$ to the solution $(\tilde{\vu}, \tilde{\pi}, \tilde{\bm{l}}, \tilde{\bm{\omega}})$ of the linear problem (\ref{fixed_FSI}) with right hand sides $\bm{F}_0(\bm{v}, \varphi,\bm{\kappa}, \bm{\tau}), \bm{H}(\bm{v},\bm{\kappa},\bm{\tau}), \bm{F}_1(\bm{\kappa},\bm{\tau}), \bm{F}_2(\bm{\tau})$, is a contraction in $S_\gamma$ for $\gamma\in (0,\tilde{\gamma})$. The fixed point of $\mathcal{N}$ then satisfies (\ref{fixed_FSI})-(\ref{24}).
\hfill
\end{proof}

\begin{proof}[Proof of Theorem \ref{T_N}]
	The solution to the original problem (\ref{1}) can be obtained from the corresponding backward change of coordinates and variables, given in (\ref{13}) which preserves regularity. Moreover, the solution $(\vu,\pi,\bm{l},\bm{\omega})$ to the original problem must be unique as a consequence of the uniqueness of the fixed point.
	
	Since $\gamma<\tilde{\gamma}$, condition (\ref{14}) is verified and $X(\cdot, t)$ is a well-defined $C^1$-diffeomorphism from $\Omega_F(0)$ to $\Omega_F(t)$ for every $t\in [0,\infty)$. Therefore there exists a unique $Y(\cdot,t)$ as defined in Lemma \ref{L0}. For all $t\in[0,\infty)$ and $\bm{x}\in\Omega_F(t)$, setting
	\begin{equation*}
	\begin{aligned}
	\vu(\bm{x},t) = Q(t)\tilde{\vu}(Y(\bm{x},t),t), &\qquad \pi(\bm{x},t) = \tilde{\pi}(Y(\bm{x},t),t),\\
	\bm{l}(t)=Q(t)\tilde{\bm{l}}(t), &\qquad \bm{\omega}(t) = Q(t)\tilde{\bm{\omega}}(t),
	\end{aligned}	
	\end{equation*}
	the new variables $(\vu, \pi, \bm{l}, \bm{\omega})$ satisfy the original system (\ref{1}) with the estimate (\ref{30}). Note that all the derivatives of the solution $(\vu, \pi, \bm{l}, \bm{\omega})$ are combinations of $(\tilde{\vu}, \tilde{\pi}, \tilde{\bm{l}}, \tilde{\bm{\omega}})$ multiplied at most by $X$ and its derivatives which are smooth enough to obtain the prescribed regularity.
	\hfill
\end{proof}

\section{Non-Newtonian case}
\setcounter{equation}{0}

In this section, we discuss the problem (\ref{GFSI}) with the non-Newtonian stress tensor $T$.
The main difference and the difficulty here is that the Laplacian in the fluid equation is now replaced by a quasi-linear operator arising from the Generalized stress tensor $T$. Observe that we can write this new term as,
\begin{align}
\left[ \div \left( \mu(|\DT\vu|^2) \DT \vu\right)\right] _i &= \sum_{j=1}^3\partial_j\left( \mu(|\DT\vu|^2) \DT_{ij} \vu\right) \nonumber \\
&= \frac{1}{2}\mu(|\DT\vu|^2) \sum_{j=1}^3 \left( \partial^2_{ij}u_j + \partial^2_j u_i\right) + \mu'(|\DT\vu|^2) \sum_{j,k,l=1}^3 2\DT_{ij}\vu \ \DT_{kl}\vu \ \partial_j \DT_{kl}\vu \nonumber \nonumber \\
&= \frac{1}{2}\mu(|\DT\vu|^2) \sum_{j=1}^3 \left( \partial^2_{ij}u_j + \partial^2_j u_i\right) + 2\mu'(|\DT\vu|^2) \sum_{j,k,l=1}^3 \DT_{ij}\vu \ \DT_{kl}\vu \ \partial_j \partial_k u_l \nonumber \\
&= \sum_{j,k,l=1}^3a^{kl}_{ij}(\vu)\partial_j \partial_k u_l \label{operatorA}
\end{align}
where
\begin{equation*}
a^{kl}_{ij}(\vu):= \frac{1}{2}\mu(|\DT\vu|^2) \left( \delta_{ik}\delta_{jl} + \delta_{il}\delta_{jk} \right) + 2\mu'(|\DT\vu|^2) \DT_{ij}\vu \ \DT_{kl}\vu.
\end{equation*}
Note that the coefficients $a^{kl}_{ij}(\vu)$ are real. Consider the differential operator
\begin{equation*}
A(\vu) := \sum_{j,k=1}^3 a^{kl}_{ij}(\vu)\partial_j \partial_k
\end{equation*}
which is defined precisely as
\begin{equation*}
\left( A(\vu)\bm{w}\right)_i := \sum_{j,k,l=1}^3a^{kl}_{ij}(\vu)\partial_j \partial_k w_l.
\end{equation*}

As done in the Newtonian case, we transfer the system of equations (\ref{GFSI}) defined on an unknown moving domain to a fixed domain by coordinate transformation. We use the same change of variables as in (\ref{13}) where the transformed generalized stress tensor is defined by
\begin{equation*}
\tilde{T}(\tilde{\vu},\tilde{\pi}) = Q^{-1}(t) \ T(Q(t)\tilde{\vu}(\bm{y},t),\tilde{\pi}(\bm{y}, t) \ Q(t).
\end{equation*}
To transform the term $A(\vu)\vu$, we calculate
\begin{equation*}
2\DT_{ij}\vu = \partial_i u_j +\partial_j u_i = \sum_{k=1}^3 \frac{ (Q\tilde{\vu})_j}{\partial y_k} \frac{\partial Y_k}{\partial x_i}+ \sum_{l=1}^3 \frac{ (Q\tilde{\vu})_i}{\partial y_l} \frac{\partial Y_l}{\partial x_j} =: 2\tilde{\DT}_{ij}\tilde{\vu}
\end{equation*}
and use the obvious notation $\tilde{\DT} \bm{w} = \left( \tilde{\DT}_{ij}\bm{w}\right) _{ij}$ to denote the transformed symmetric part of the gradient. Therefore, the transformed quasi-linear fluid operator can be written as, from (\ref{operatorA}),
\begin{equation*}
\left( \mathcal{A}(\tilde{\vu})\bm{w}\right) _i = \tilde{a}^{klm}_{ij}(\tilde{\vu}) \partial_j \partial_k w_l
\end{equation*}
where
\begin{equation*}
\tilde{a}^{klm}_{ij}(\tilde{\vu}) := \frac{1}{2}\mu(|\tilde{\DT}\tilde{\vu}|^2) \left( \delta_{ik}\delta_{jl} + \delta_{il}\delta_{jk} \right) + 2\mu'(|\tilde{\DT}\tilde{\vu}|^2) \tilde{\DT}_{ij}\tilde{\vu} \ \tilde{\DT}_{kl}\tilde{\vu}.
\end{equation*}
On the fixed domain, the transformed system then becomes,
\begin{equation}
\label{fixed_GFSI}
\begin{cases}
\begin{aligned}
\tilde{\vu}_t - \mathcal{A}(\tilde{\vu}) \tilde{\vu} + \nabla\tilde{\pi} &= \bm{F}_0(\tilde{\vu}, \tilde{\pi}, \tilde{\bm{l}}, \tilde{\bm{\omega}})\ &&\text{ in } \Omega_F(0)\times (0,T),\\
\div \ \tilde{\vu} &= \div \ \bm{H}(\tilde{\vu}, \tilde{\bm{l}}, \tilde{\bm{\omega}}) \ &&\text{ in } \Omega_F(0)\times (0,T),\\
\tilde{\vu} &=\bm{0} \ &&\text{ on } \ \partial \Omega \times (0,T),\\
\tilde{\vu}\cdot\tilde{\vn}=\tilde{\vu}_S\cdot \tilde{\vn}, \quad &\left[\tilde{\mathbb{T}}(\tilde{\vu},\tilde{\pi})\tilde{\vn}\right]_{\vt}+\alpha\tilde{\vu}_{\vt}=\alpha\tilde{\vu}_{S\vt} \ &&\text{ on } \ \partial \Omega_S(0) \times (0,T),\\
m \tilde{\bm{l}}' &= - \int\displaylimits_{\partial \Omega_S(0)}{\tilde{\mathbb{T}}(\tilde{\vu},\tilde{\pi})\tilde{\vn}} + \bm{F}_1(\tilde{\bm{l}},\tilde{\bm{\omega}}), \ &&\ t\in (0,T),\\
\tilde{J} \tilde{\bm{\omega}}' &= - \int\displaylimits_{\partial \Omega_S(0)}{y\times \tilde{\mathbb{T}}(\tilde{\vu},\tilde{\pi})\tilde{\vn}} + \bm{F}_2(\tilde{\bm{\omega}}), \ &&\ t\in (0,T),\\
\tilde{\vu}(0) &= \vu_0 \ &&\text{ in } \ \Omega_F(0),\\
\tilde{\bm{h}}^{'}( 0) = \bm{l}_0, \quad &\tilde{\bm{\omega}}(0) = \bm{\omega}_0
\end{aligned}
\end{cases}
\end{equation}
where $\bm{F}_0, \bm{H}, \bm{F}_1, \bm{F}_2$ are defined in (\ref{36}) - (\ref{24}).

Next we linearize the above system and prove the maximal regularity for the linear problem. We linearize $\mathcal{A}$ by the operator $A_*$, defined as
$$
A_* \tilde{\vu} := \mathcal{A}(\vu^*)\tilde{\vu}
$$
which fixes the coefficients in the original operator $A$ to a reference solution $\vu^*$ of the following problem with Newtonian fluid,
\begin{equation}
\label{39}
\begin{cases}
\begin{aligned}
\partial_t \vu^* - \Delta \vu^* + \nabla \pi^* &= \bm{0} \ &&\text{ in } \ \Omega_F(0)\times (0,T),\\
\div \ \vu^* &= 0 \ &&\text{ in } \ \Omega_F(0)\times (0,T),\\
\vu^*&=\bm{0} \ &&\text{ on } \ \partial \Omega \times (0,T),\\
\vu^*\cdot \tilde{\vn}&=(\bm{l}^* + \bm{\omega}^* \times \bm{y})\cdot \tilde{\vn} \ &&\text{ on } \ \partial \Omega_S(0) \times (0,T),\\ \left[\sigma(\vu^*,\pi^*)\tilde{\vn}\right]_{\vt}+\alpha\vu^*_{\vt}&=\alpha(\bm{l}^* + \bm{\omega}^* \times \bm{y})_{\vt} \ &&\text{ on } \ \partial \Omega_S(0) \times (0,T),\\
m (\bm{l}^*)' &= - \int\displaylimits_{\partial \Omega_S(0)}{\sigma (\vu^*,\pi^*)\tilde{\vn}}, \ && \ t\in(0,T),\\
J(0) (\bm{\omega}^*)' &= - \int\displaylimits_{\partial \Omega_S(0)}{\bm{y}\times \sigma (\vu^*,\pi^*)\tilde{\vn}}, \ && \ t\in (0,T),\\
\vu^*(0) &= \vu_0 \ &&\text{ in } \ \Omega_F(0),\\
\bm{l}^*( 0) = \bm{l}_0, &\quad \bm{\omega}^*(0) = \bm{\omega}_0 .
\end{aligned}
\end{cases}
\end{equation}
The existence of $(\vu^*, \pi^*, \bm{l}^*, \bm{\omega}^*)$ satisfying (\ref{39}) follows from Theorem \ref{T11}. We then define $(\hat{\vu}, \hat{\pi}, \hat{\bm{l}}, \hat{\bm{\omega}}) := (\tilde{\vu}-\vu^*, \tilde{\pi}-\pi^*, \tilde{\bm{l}}-\bm{l}^*, \tilde{\bm{\omega}}-\bm{\omega}^*)$ to rewrite (\ref{fixed_GFSI}) into the equivalent system
\begin{equation}
\label{40}
\begin{cases}
\begin{aligned}
\partial_t \hat{\vu} - A_* \hat{\vu} + \nabla \hat{\pi} &= \bm{G}_0(\hat{\vu}, \hat{\pi}, \hat{\bm{l}}, \hat{\bm{\omega}}) \ &&\text{ in } \ \Omega_F(0)\times (0,T),\\
\div \ \hat{\vu} &= \div \ \bm{H}(\hat{\vu}, \hat{\bm{l}}, \hat{\bm{\omega}}) \ &&\text{ in } \ \Omega_F(0)\times (0,T),\\
\hat{\vu}&=\bm{0} \ &&\text{ on } \ \partial \Omega \times (0,T),\\
\hat{\vu}\cdot \tilde{\vn}&=(\hat{\bm{l}} + \hat{\bm{\omega}} \times \bm{y})\cdot \tilde{\vn} \ &&\text{ on } \ \partial \Omega_S(0) \times (0,T),\\ \left[\sigma(\hat{\vu},\hat{\pi})\tilde{\vn}\right]_{\vt}+\alpha\hat{\vu}_{\vt}&=\alpha(\hat{\bm{l}} + \hat{\bm{\omega}} \times \bm{y})_{\vt}+\bm{H}_1(\hat{\vu}, \hat{\pi}, \hat{\bm{l}}, \hat{\bm{\omega}}) \ &&\text{ on } \ \partial \Omega_S(0) \times (0,T),\\
m \,\hat{\bm{l}}' &= - \int\displaylimits_{\partial \Omega_S(0)}{\sigma (\hat{\vu},\hat{\pi})\tilde{\vn}} +\bm{G}_1(\hat{\vu}, \hat{\pi}, \hat{\bm{l}}, \hat{\bm{\omega}}), \ && \ t\in(0,T),\\
J(0) \,\hat{\bm{\omega}}' &= - \int\displaylimits_{\partial \Omega_S(0)}{\bm{y}\times \sigma (\hat{\vu},\hat{\pi})\tilde{\vn}} +\bm{G}_2(\hat{\vu}, \hat{\pi}, \hat{\bm{l}}, \hat{\bm{\omega}}),  && \ t\in (0,T),\\
\hat{\vu}(0) &= \bm{0} \ &&\text{ in } \ \Omega_F(0),\\
\hat{\bm{l}}( 0) = \bm{0}, &\quad \hat{\bm{\omega}}(0) = \bm{0}
\end{aligned}
\end{cases}
\end{equation}
where
\begin{equation}
\label{45}
\begin{aligned}
\bm{G}_0(\hat{\vu}, \hat{\pi}, \hat{\bm{l}}, \hat{\bm{\omega}}) &:= \bm{F}_0(\hat{\vu}, \hat{\pi}, \hat{\bm{l}}, \hat{\bm{\omega}}) - \Delta \vu^* - Q(\vu^*,\hat{\vu}),\\
Q(\vu^*,\hat{\vu}) &:= A_*\hat{\vu} - \mathcal{A}(\vu^* + \hat{\vu}) (\vu^* + \hat{\vu}),\\
\bm{G}_1(\hat{\vu}, \hat{\pi}, \hat{\bm{l}}, \hat{\bm{\omega}}) &:= \int\displaylimits_{\partial \Omega_S(0)}{(\sigma-\tilde{\mathbb{T}})(\hat{\vu},\hat{\pi})\tilde{\vn}} + \int\displaylimits_{\partial \Omega_S(0)}{(\sigma-\tilde{\mathbb{T}})(\vu^*,\pi^*)\tilde{\vn}} - m(\bm{\omega}^* + \hat{\bm{\omega}})\times (\bm{l}^* + \hat{\bm{l}}),\\
\bm{G}_2(\hat{\vu}, \hat{\pi}, \hat{\bm{l}}, \hat{\bm{\omega}}) &:= \int\displaylimits_{\partial \Omega_S(0)}{\bm{y}\times(\sigma-\tilde{\mathbb{T}})(\hat{\vu},\hat{\pi})\tilde{\vn}} + \int\displaylimits_{\partial \Omega_S(0)}{\bm{y}\times (\sigma-\tilde{\mathbb{T}})(\vu^*,\pi^*)\tilde{\vn}} - \tilde{J}(\bm{\omega}^* + \hat{\bm{\omega}})\times (\bm{\omega}^* + \hat{\bm{\omega}}),\\
\bm{H}_1(\hat{\vu}, \hat{\pi}, \hat{\bm{l}}, \hat{\bm{\omega}}) &:= \left[ (\sigma-\tilde{\mathbb{T}})(\hat{\vu},\hat{\pi})\tilde{\vn}+(\sigma-\tilde{\mathbb{T}})(\vu^*,\pi^*)\tilde{\vn} \right] _{\vt}
\end{aligned}
\end{equation}
and $\bm{F}_0, \bm{H}$ are the same as in (\ref{36}). 
The following theorem states the main result concerning the linearized problem .

\begin{theorem}
	\label{T8}
	Let $\Omega_F(0)$ be a bounded domain of class $\HC{2}{1}$, $p>5$ and $\alpha\ge 0$ be as in (\ref{alpha}). Also assume that
	$$
	g_0\in L^p(0,T;\bm{L}^p(\Omega_F(0))),\chi \in W^{2,1}_{p,p}(Q^T_F),
	\bm{l}\in L^p(0,T;\mathbb{R}^3),\bm{\omega}\in L^p(0,T;\mathbb{R}^3)
	$$
	and
	$$
	h_1\in W^{\frac{1}{2}-\frac{1}{2p},p}(0,T;\bm{L}^p(\partial\Omega_F(0)))\cap L^p(0,T;\bm{W}^{1-\frac{1}{p}}(\partial\Omega_F(0)))
	$$
	where
	$$
	\div\ \chi\arrowvert_{t=0} =0,
	h_1\cdot\tilde{\vn} = 0 \text{ on } \partial\Omega_F(0).
	$$
 Then the problem
	\begin{equation}
	\label{41}
	\begin{cases}
	\begin{aligned}
	\partial_t \vu - A_* \vu + \nabla \pi &= g_0 \ &&\text{ in } \ \Omega_F(0)\times (0,T),\\
	\div \ \vu &= \div \ \chi \ &&\text{ in } \ \Omega_F(0)\times (0,T),\\
	\vu&=\bm{0} \ &&\text{ on } \ \partial \Omega \times (0,T),\\
	\vu\cdot \tilde{\vn}&=(\bm{l} +\bm{\omega} \times \bm{y})\cdot \tilde{\vn} \ &&\text{ on } \ \partial \Omega_S(0) \times (0,T),\\ \left[\sigma(\vu,\pi)\tilde{\vn}\right]_{\vt}+\alpha\vu_{\vt}&=\alpha(\bm{l} + \bm{\omega} \times \bm{y})_{\vt}+h_1 \ &&\text{ on } \ \partial \Omega_S(0) \times (0,T),\\
	m \bm{l}' &= - \int\displaylimits_{\partial \Omega_S(0)}{\sigma (\vu,\pi)\tilde{\vn}} +g_1, \ && \ t\in(0,T),\\
	J(0) \bm{\omega}' &= - \int\displaylimits_{\partial \Omega_S(0)}{\bm{y}\times \sigma (\vu,\pi)\tilde{\vn}} +g_2, \ && \ t\in (0,T),\\
	\vu(0) &= \bm{0} \ &&\text{ in } \ \Omega_F(0),\\
	\bm{l}( 0) = \bm{0}, &\quad \bm{\omega}(0) = \bm{0}
	\end{aligned}
	\end{cases}
	\end{equation}
	has a local in time unique solution $\vu\in 
	W^{2,1}_{p,p}(Q^T_F), \pi \in L^p(0,T;W^{1,p}(\Omega_F(0)))$, $(\bm{l}, \bm{\omega})\in W^{1,p}(0,T; \R^6)$ satisfying the estimate
	\begin{equation}
	\label{44}
	\begin{aligned}
	&\quad \|\vu\|_{W^{2,1}_{p,p}(Q^T_F)} + \|\pi\|_{L^p(0,T;W^{1,p}(\Omega_F(0)))} + \|\bm{l}\|_{L^p(0,T;\R^3)} + \|\bm{\omega}\|_{L^p(0,T;\R^3)} \\
	\le & \ C \left( \|g_0\|_{L^p(0,T;\bm{L}^p(\Omega_F(0)))} + \|\chi\|_{W^{2,1}_{p,p}(Q^T_F)} +\|(g_1,g_2)\|_{L^p(0,T;\R^6)} \right.\\
	& \hspace{3cm} \left.+ \|h_1\|_{W^{\frac{1}{2}-\frac{1}{2p},p}(0,T;\bm{L}^p(\partial\Omega_F(0)))\cap L^p(0,T;\bm{W}^{1-\frac{1}{p}}(\partial\Omega_F(0)))} \right) .
	\end{aligned}
	\end{equation}
The above constant $C>0$ depends on $T, \alpha, p$ with the property that $C(T)\to 0$ as $T\to 0$.
\end{theorem}

Unlike in the Newtonian case, here we first reduce the inhomogeneous divergence condition to the divergence-free problem and then study the system. The reason being, showing exponential stability for non-Newtonian fluid seems more intricate and might require different approach which is currently out of reach.

To prove the maximal regularity property, splitting the fluid and the solid equations and with the help of the maximal regularity of the \textit{generalized Stokes operator} $A_*$, we rewrite suitably the forces acting on the rigid body in terms of the fluid velocity and pressure, as in \cite{GGH}. Note that, in this case we obtain the regularity $u\in L^p(0, T; L^p), T< \infty$ (in the above Theorem) directly, unlike the Newtonian case (cf. Theorem \ref{T2}) which follows as in \cite[Theorem 10.1]{GGH}. The reason being, we do not use here the $\mathcal{R}$-sectoriality property neither the property that $0$ is in the resolvent set.

In the following subsection, we combine the relevant results concerning the linear fluid-structure problem corresponding to the non-Newtonian fluid.

\subsection{Maximal regularity of the linearised system}

We first recall the following existence result of maximal regular solution of the linearised fluid part of the above problem.

\begin{proposition}$\cite[ Theorem \ 4.1]{BP}.$
\label{P_MR}
Let $p>5$ and $\alpha\ge 0$ be as in (\ref{alpha}). Further assume that
$$
g_0\in L^p(0,T;\bm{L}^p(\Omega_F(0))),\chi \in W^{2,1}_{p,p}(Q^T_F),
\bm{l}\in L^p(0,T;\mathbb{R}^3),\bm{\omega}\in L^p(0,T;\mathbb{R}^3)
$$
and
$$
h_1\in W^{\frac{1}{2}-\frac{1}{2p},p}(0,T;\bm{L}^p(\partial\Omega_F(0)))\cap L^p(0,T;\bm{W}^{1-\frac{1}{p}}(\partial\Omega_F(0))), \vu_0\in W^{2-2/p,p}(\Omega_F(0))
$$
where
$$
\div\ \chi\arrowvert_{t=0}=0, \ (\bm{l} + \bm{\omega} \times \bm{y})\cdot \tilde{\vn}\arrowvert_{t=0} =\bm{0}\text{ on } \partial\Omega_S(0), \ h_1\cdot\tilde{\vn} = 0 \text{ on } \partial\Omega_F(0)
$$
and
$$
\mathbb{P}h_1 + \alpha\mathbb{P}(\bm{l} + \bm{\omega} \times \bm{y})_{\vt}  \mathbbm{1}_{\partial\Omega_S(0)} = \bm{0}.
$$
Then there exists a unique strong solution
$$
\mathring{\vu}\in L^p(0,T;\bm{W}^{2,p}(\Omega_F(0)))\cap W^{1,p}(0,T;\bm{L}^p(\Omega_F(0))), \mathring{\pi}\in L^p(0,T;W^{1,p}(\Omega_F(0)))
$$
of the following Stokes problem
	\begin{equation}
	\label{42}
	\begin{cases}
	\begin{aligned}
\mathring{\vu}_t - A_* \mathring{\vu}+\nabla \mathring{\pi} &= g_0 \ &&\text{ in } \ \Omega_F(0),\\
	\div \ \mathring{\vu} &= \div \ \chi \ &&\text{ in } \ \Omega_F(0),\\
	\mathring{\vu}&=\bm{0} \ &&\text{ on } \ \partial \Omega,\\
	\mathring{\vu}\cdot \tilde{\vn}&=(\bm{l} + \bm{\omega} \times \bm{y})\cdot \tilde{\vn} \ &&\text{ on } \ \partial \Omega_S(0),\\
	2\left[(\DT\mathring{\vu})\tilde{\vn}\right]_{\vt}+\alpha \mathring{\vu}_{\vt}&=\alpha (\bm{l} + \bm{\omega} \times \bm{y})_{\vt}+ h_1 \ &&\text{ on } \ \partial \Omega_S(0),\\
	\mathring{\vu}(0) &=\bm{0} .
	\end{aligned}
	\end{cases}
	\end{equation}
	The solution depends continuously on the data in the corresponding spaces.
\end{proposition}

Note that the proof of the above result is done in \cite[Theorem 4.1]{BP} for the full slip condition $\alpha =0$. The case when $\alpha >0$ is a function, the additional term $\alpha \vu_{\vt}$ being a lower order perturbation does not affect the analysis of well-posedness and regularity and can be derived by the same analysis. Indeed, with the help of \cite[Proposition 3.3.9]{pruss} , we obtain the maximal regularity (or equivalently, bounded imaginary powers) of the perturbed Stokes operator with full Navier boundary condition from the Stokes operator with $\alpha=0$.

Next we deduce the divergence free and homogeneous boundary conditions for the problem (\ref{41}) by subtracting the solution $(\mathring{\vu}, \mathring{\pi})$ of (\ref{42}).

Writing $(\hat{\vu}, \hat{\pi}, \hat{\bm{l}}, \hat{\bm{\omega}}) = (\vu - \mathring{\vu}, \pi - \mathring{\pi}, \bm{l}, \bm{\omega})$, the system (\ref{41}) becomes equivalent to,
\begin{equation}
\label{43}
\begin{cases}
\begin{aligned}
\partial_t \hat{\vu} - A_* \hat{\vu} + \nabla \hat{\pi} &= \bm{0} \ &&\text{ in } \ \Omega_F(0)\times (0,T),\\
\div \ \hat{\vu} &= 0 \ &&\text{ in } \ \Omega_F(0)\times (0,T),\\
\hat{\vu}&=\bm{0} \ &&\text{ on } \ \partial \Omega \times (0,T),\\
\hat{\vu}\cdot \tilde{\vn}=0, &\quad \left[\sigma(\hat{\vu},\hat{\pi})\tilde{\vn}\right]_{\vt}+\alpha\hat{\vu}_{\vt}=\bm{0} \ &&\text{ on } \ \partial \Omega_S(0) \times (0,T),\\
m \hat{\bm{l}}' = - \int\displaylimits_{\partial \Omega_S(0)}{\sigma (\hat{\vu},\hat{\pi})\tilde{\vn}} &- \int\displaylimits_{\partial \Omega_S(0)}{\sigma (\mathring{\vu},\mathring{\pi})\tilde{\vn}} +g_1, \ && \ t\in(0,T),\\
J(0) \hat{\bm{\omega}}' = - \int\displaylimits_{\partial \Omega_S(0)}{\bm{y}\times \sigma (\hat{\vu},\hat{\pi})\tilde{\vn}} &- \int\displaylimits_{\partial \Omega_S(0)}{\bm{y}\times \sigma (\mathring{\vu},\mathring{\pi})\tilde{\vn}} +g_2, \ && \ t\in (0,T),\\
\hat{\vu}(0) &= \bm{0} \ &&\text{ in } \ \Omega_F(0),\\
\hat{\bm{l}}( 0) = \bm{0}, &\quad \hat{\bm{\omega}}(0) = \bm{0}.
\end{aligned}
\end{cases}
\end{equation}
Let us define the operator, for all $0<\varepsilon<1-\frac{1}{p}$,
\begin{equation*}
\begin{aligned}
\mathcal{J}: W^{\varepsilon+1/p,p}(\Omega_F(0);\R^{3\times 3})&\to \R^6\\
h&\mapsto \begin{pmatrix}
\int\displaylimits_{\partial\Omega_S(0)}{h\tilde{\vn}}\\
\int\displaylimits_{\partial\Omega_S(0)}{\bm{y}\times h\tilde{\vn}}
\end{pmatrix}.
\end{aligned}
\end{equation*}
It follows from the boundedness of the trace operator that
\begin{equation*}
\|\mathcal{J}h\| \le C \|h\|_{W^{\varepsilon+1/p,p}(\Omega_F(0))} .
\end{equation*}
Now the fifth, sixth and eighth equations for the rigid motion in the system (\ref{43}) can be written as
\begin{equation*}
\begin{cases}
\begin{aligned}
&\mathbb{I}\begin{pmatrix}
\hat{\bm{l}}'\\
\hat{\bm{\omega}}'
\end{pmatrix}
= - \mathcal{J}(\sigma(\hat{\vu}, \hat{\pi})) - \mathcal{J}(\sigma(\mathring{\vu}, \mathring{\pi})) +
 \begin{pmatrix}
 g_1\\
 g_2
 \end{pmatrix},\\
& (\hat{\bm{l}}(0), \hat{\bm{\omega}}(0)) = (\bm{0}, \bm{0}).
 \end{aligned}
 \end{cases}
\end{equation*}
where $\mathbb{I}$ is the constant momentum matrix as before. This allows us to rewrite the above set of equations in the form
\begin{equation*}
\begin{pmatrix}
\hat{\bm{l}}\\
\hat{\bm{\omega}}
\end{pmatrix}
= \mathcal{R}
\begin{pmatrix}
\hat{\bm{l}}\\
\hat{\bm{\omega}}
\end{pmatrix}
+ \hat{g}
\end{equation*}
where $\mathcal{R}:W^{1,p}_0(0,T;\R^6) \to W^{1,p}_0(0,T;\R^6) $ is given by
\begin{equation*}
\mathcal{R}
\begin{pmatrix}
\hat{\bm{l}}\\
\hat{\bm{\omega}}
\end{pmatrix}
(t) := -\int\displaylimits_0^t{\mathbb{I}^{-1}\mathcal{J}(\sigma(\hat{\vu},\hat{\pi}))}
\end{equation*}
and
\begin{equation*}
\hat{g}(t) := \int\displaylimits_0^t {\mathbb{I}^{-1}\left[ - \mathcal{J}(\sigma(\mathring{\vu}, \mathring{\pi})) +
	\begin{pmatrix}
	g_1\\
	g_2
	\end{pmatrix}\right] }.
\end{equation*}
The following lemma says that for sufficiently small $T>0$, there exists a unique $(\hat{\bm{l}}, \hat{\bm{\omega}})\in W^{1,p}_0(0,T;\R^6)$ satisfying
\begin{equation*}
\begin{pmatrix}
\hat{\bm{l}}\\
\hat{\bm{\omega}}
\end{pmatrix}
= (I_{6} - \mathcal{R})^{-1} \hat{g}.
\end{equation*}

\begin{lemma}
The map $\mathcal{R}$ is bounded and $\|\mathcal{R}\|_{W^{1,p}_0(0,T;\R^6)} \le 1$ for sufficiently small $T>0$. Moreover, $\hat{g}\in W^{1,p}_0(0,T;\R^6)$.
\end{lemma}

Proof of this lemma is similar to the one, done in \cite{GGH} which is based on the estimate of $\mathcal{J}$ and the maximal regularity of $(\vu,\pi)$ solving the fluid part of the system (\ref{43}). Note that we have the presence of slip in the boundary condition which is different from \cite{GGH}; Thus the maximal regularity property of the generalized Stokes operator $A_*$ with slip condition is required here (as mentioned just after Proposition \ref{P_MR}).

From the above lemma, apart from the existence of $(\hat{\bm{l}}, \hat{\bm{\omega}})$, we obtain the following estimate:
\begin{equation*}
\begin{aligned}
& \quad \|\hat{\bm{l}}\|_{W^{1,p}(0,T;\R^3)} + \|\hat{\bm{\omega}}\|_{W^{1,p}(0,T;\R^3)}\\
&\le C \left( 1-\|\mathcal{R}\|_{\mathcal{L}(W_0^{1,p}(0,T),W_0^{1,p}(0,T))}\right)\left( \|g_0\|_{L^p(0,\infty;\bm{L}^q(\Omega_F(0)))} + \|\chi\|_{W^{2,1}_{q,p}(Q^T_F)} +\|(g_1,g_2)\|_{L^p(0,\infty;\R^6)} \right.\\
& \qquad \qquad \qquad \qquad \qquad \qquad \qquad \qquad \left.+ \|h_1\|_{W^{\frac{1}{2}-\frac{1}{2p},p}(0,T;\bm{L}^p(\partial\Omega_F(0)))\cap L^p(0,T;\bm{W}^{1-\frac{1}{p}}(\partial\Omega_F(0)))} \right) .
\end{aligned}
\end{equation*}
Plugging $\hat{\bm{l}}, \hat{\bm{\omega}}$ into (\ref{43}) yields a solution
\begin{equation*}
\begin{aligned}
& \vu = \hat{\vu} + \mathring{\vu} \in W^{2,1}_{p,p}(Q^T_F), \ &&\pi = \hat{\pi} + \mathring{\pi} \in L^p(0,T;W^{1,p}(\Omega_F(0))),\\
& \bm{l} = \hat{\bm{l}} \in W^{1,p}(0,T;\R^3), \ &&\bm{\omega} = \hat{\bm{\omega}} \in W^{1,p}(0,T;\R^3)
\end{aligned}
\end{equation*}
of (\ref{41}) satisfying the estimate (\ref{44}). Also the dependence of the constant $C$ on $T$ can be deduced as in \cite{GGH}, from the corresponding pressure estimate.

\subsection{Fixed point argument}

Theorem \ref{T8} now allows us to solve (\ref{40}) via a contraction mapping argument. We introduce as in the Section \ref{S6.2},
\begin{equation}
\label{S_gamma}
S_\gamma := \{(\tilde{\vu},\tilde{\pi},\tilde{\bm{l}},\tilde{\bm{\omega}}): \|(\tilde{\vu},\tilde{\pi},\tilde{\bm{l}},\tilde{\bm{\omega}})\|_S\le \gamma \}
\end{equation}
with
\begin{equation*}
\begin{aligned}
\|(\tilde{\vu},\tilde{\pi},\tilde{\bm{l}},\tilde{\bm{\omega}})\|_S:= \|\tilde{\vu}\|_{ W^{2,1}_{p,p}(Q^T_F)}+ \|\tilde{\pi}\|_{L^p(0,T;W^{1,q}(\Omega_F(0)))}+ \|\tilde{\bm{l}}\|_{W^{1,p}(0,T;\R^3)}+ \|\tilde{\bm{\omega}}\|_{W^{1,p}(0,T;\R^3)}
\end{aligned}
\end{equation*}
as the underlying set in the natural function spaces. Let
\begin{equation*}
\mathcal{N}:
\begin{pmatrix}
\tilde{\vu}\\
\tilde{\pi}\\
\tilde{\bm{l}}\\
\tilde{\bm{\omega}}
\end{pmatrix}
\mapsto
\begin{pmatrix}
\bm{G}_0(\tilde{\vu},\tilde{\pi},\tilde{\bm{l}},\tilde{\bm{\omega}})\\
\bm{H}(\tilde{\vu},\tilde{\pi},\tilde{\bm{l}},\tilde{\bm{\omega}})\\
\bm{H}_1(\tilde{\vu},\tilde{\pi},\tilde{\bm{l}},\tilde{\bm{\omega}})\\
\bm{G}_1(\tilde{\vu},\tilde{\pi},\tilde{\bm{l}},\tilde{\bm{\omega}})\\
\bm{G}_2(\tilde{\vu},\tilde{\pi},\tilde{\bm{l}},\tilde{\bm{\omega}})
\end{pmatrix}
\mapsto
\begin{pmatrix}
\vu\\
\pi\\
\bm{l}\\
\bm{\omega}
\end{pmatrix}
\end{equation*}
be the function which maps $(\tilde{\vu},\tilde{\pi},\tilde{\bm{l}},\tilde{\bm{\omega}})\in S_\gamma$ to $(\bm{G}_0, \bm{H}, \bm{H}_1, \bm{G}_1, \bm{G}_2)$ which are defined in (\ref{45}), and then to the solution of the linear problem with fixed right hand sides, using Theorem \ref{T8}. For sufficiently small $\gamma>0$, we show that the Banach fixed point theorem can be applied to the map $\mathcal{N}$.

\begin{theorem}
	\label{T10}
	For $T$ and $\gamma$ sufficiently small, the function $\mathcal{N}$ maps $S_\gamma$ into itself and it is contractive.
\end{theorem}

\begin{proof}
First we show that the image of $\mathcal{N}$ is contained in $S_\gamma$. Let us assume that $(\tilde{\vu},\tilde{\pi},\tilde{\bm{l}},\tilde{\bm{\omega}})\in S_\gamma$
and that $(\vu^*, \pi^*, \bm{l}^*, \bm{\omega}^*)$ are given by (\ref{39}).
We want to show the following estimate
\begin{equation}
\label{46}
\|\mathcal{N}(\tilde{\vu},\tilde{\pi},\tilde{\bm{l}},\tilde{\bm{\omega}})\|_{S} \le C(T,\gamma) \|(\tilde{\vu},\tilde{\pi},\tilde{\bm{l}},\tilde{\bm{\omega}})\|_{S}
\end{equation}
where $C(T,\gamma)\to 0$ as $T,\gamma\to 0$. Due to the maximal regularity of the linear problem, it follows directly from the following estimate
\begin{equation*}
\begin{aligned}
&\quad \|\bm{G}_0\|_{L^p(0,T;\bm{L}^p(\Omega_F(0)))} + \|\bm{H}\|_{W^{2,1}_{p,p}(Q^\infty_F)}+ \|\bm{H}_1\|_{}+\|\bm{G}_1\|_{L^p(0,T;\R^3)}+ \|\bm{G}_2\|_{L^p(0,T;\R^3)}\\
& \le C(T,\gamma) \|(\tilde{\vu},\tilde{\pi},\tilde{\bm{l}},\tilde{\bm{\omega}})\|_{S}.
\end{aligned}
\end{equation*}
Estimate for $\bm{H}$ is deduced in (\ref{32}).
In order to show the estimate for $\bm{G}_0$, it suffices to consider the term
$$
Q(\vu^*,\tilde{\vu}) := A_*\tilde{\vu} - \mathcal{A}(\vu^* + \tilde{\vu}) (\vu^* + \tilde{\vu})
$$
which is new compared to the Newtonian case. Estimating $Q$ follows the same argument as in \cite[pp 1431-34]{GGH} which involves essentially writing it as a difference of suitable forms, hence we skip it.
We obtain
\begin{equation*}
\|Q(\vu^*, \tilde{\vu})\|_{L^p(0,T;\bm{L}^p(\Omega_F(0))}\le C \left(\gamma^2 + \|\vu^*\|_{W^{2,1}_{p,p}(Q^\infty_F)} + T \right)
\end{equation*}
which yields
\begin{equation*}
\begin{aligned}
& \quad \ \|\bm{G}_0(\tilde{\vu},\tilde{\pi},\tilde{\bm{l}},\tilde{\bm{\omega}})\|_{L^p(0,T;\bm{L}^p(\Omega_F(0))}\\
& \le \|\bm{F}_0(\tilde{\vu},\tilde{\pi},\tilde{\bm{l}},\tilde{\bm{\omega}})\|_{L^p(0,T;\bm{L}^p(\Omega_F(0))} + \|\vu^*\|_{W^{2,1}_{p,p}(Q^\infty_F)} +\|Q(\vu^*, \tilde{\vu})\|_{L^p(0,T;\bm{L}^p(\Omega_F(0))}\\
& \le C (\gamma^2 +\|\vu^*\|_{W^{2,1}_{p,p}(Q^\infty_F)} + T ).
\end{aligned}
\end{equation*}
Also, by definition,
\begin{equation*}
\begin{aligned}
& \quad \ \|\bm{G}_1(\tilde{\vu},\tilde{\pi},\tilde{\bm{l}},\tilde{\bm{\omega}})\|_{L^p(0,T)}\\
& \le m\|(\bm{\omega}^* + \hat{\bm{\omega}})\times (\bm{l}^* + \hat{\bm{l}})\|_{L^p(0,T)} + \Big\|\int\displaylimits_{\partial \Omega_S(0)}(\sigma - \tilde{\sigma})(\tilde{\vu}, \tilde{\pi})\tilde{\vn}\Big\|_{} + \Big\|\int\displaylimits_{\partial \Omega_S(0)}(\tilde{\sigma} - \tilde{\mathbb{T}})(\tilde{\vu}, \tilde{\pi})\tilde{\vn}\Big\|_{L^p(0,T)}\\
& \qquad + \Big\|\int\displaylimits_{\partial \Omega_S(0)}(\sigma - \tilde{\sigma})(\vu^*, \pi^*)\tilde{\vn}\Big\|_{L^p(0,T)}+ \Big\|\int\displaylimits_{\partial \Omega_S(0)}(\tilde{\sigma} - \tilde{\mathbb{T}})(\vu^*, \pi^*)\tilde{\vn}\Big\|_{L^p(0,T)}\\
& \le C \gamma^2 + C \|\mathcal{J}(\sigma - \tilde{\sigma})(\tilde{\vu}, \tilde{\pi})\|_{L^p(0,T)} + C \|\mathcal{J}(Q^T(2-\mu(|\DT\tilde{\vu}|^2))\DT(Q\tilde{\vu})Q)\|_{L^p(0,T)} \\
& \qquad \quad + C \|\mathcal{J}(\sigma - \tilde{\sigma})(\vu^*, \pi^*)\|_{L^p(0,T)} + C \|\mathcal{J}((2\DT\vu^*-Q^T\mu(|\DT\vu^*|^2))\DT(Q\vu^*)Q)\|_{L^p(0,T)}\\
& \le C \gamma^2 + C \|\DT(Q(\vu^* + \tilde{\vu}))\|_{L^p(0,T;C(\overline{\Omega_F(0)}))} \\
& \le C \gamma^2.
\end{aligned}
\end{equation*}
Similarly,
\begin{equation*}
\begin{aligned}
& \quad \ \|\bm{G}_2(\tilde{\vu},\tilde{\pi},\tilde{\bm{l}},\tilde{\bm{\omega}}) \|_{L^p(0,T)}\\
& \le C\Big\|\mathcal{J}(\sigma - \tilde{\sigma})(\tilde{\vu}, \tilde{\pi})\Big\|_{L^p(0,T)} +C \Big\|\mathcal{J}(\tilde{\sigma} - \tilde{\mathbb{T}})(\tilde{\vu}, \tilde{\pi})\Big\|_{L^p(0,T)}
 + C\Big\|\mathcal{J}(\sigma - \tilde{\sigma})(\vu^*, \pi^*)\Big\|_{L^p(0,T)}\\
& \qquad + C\Big\|\mathcal{J}(\tilde{\sigma} - \tilde{\mathbb{T}})(\vu^*, \pi^*)\Big\|_{L^p(0,T)}+ C\|(\bm{\omega}^* + \hat{\bm{\omega}})\times (\bm{\omega}^* + \hat{\bm{\omega}})\|_{L^p(0,T)}\\
& \le C\gamma^2
\end{aligned}
\end{equation*}
and for $\bm{H}_1(\tilde{\vu},\tilde{\pi},\tilde{\bm{l}},\tilde{\bm{\omega}})$.
Thus we obtain (\ref{46}).

Next we prove that $\mathcal{N}$ is contractive. For that, let $(\tilde{\vu}^i,\tilde{\pi}^i,\tilde{\bm{l}}^i,\tilde{\bm{\omega}}^i) \in S_\gamma, i\in\{1,2\} $. As done in Theorem \ref{T9}, we estimate the differences of the functions $\bm{G}_0, \bm{H}, \bm{H}_1, \bm{G}_1, \bm{G}_2$ corresponding to $(\tilde{\vu}^i,\tilde{\pi}^i,\tilde{\bm{l}}^i,\tilde{\bm{\omega}}^i)$. This Lipschitz estimates can be shown again exactly using the same argument as in \cite[pp 1435-37]{GGH}. Concerning the estimates on two extra terms $\bm{H}$ and $\bm{H}_1$, $\bm{H}$ is already treated in (\ref{31}) and $\bm{H}_1$ can also be estimated in the similar way as for $\bm{G}_1$. This completes the proof. The fixed point of $\mathcal{N}$ is the solution of (\ref{40})-(\ref{45}).
\hfill
\end{proof}

\begin{proposition}
	Let $p>5$, $\Omega$ be a bounded domain of class $\HC{2}{1}$ and $\alpha\ge 0$ satisfies (\ref{alpha}). Also assume that $(\vu_0, \bm{l}_0, \bm{\omega}_0)\in \bm{W}^{2-2/p,p}(\Omega_F(0))\times \R^3\times \R^3$ satisfying the compatibility condition (\ref{CN}). Then there exists $T_0>0$ such that
the problem (\ref{GFSI}) admits a unique strong solution on $[0,T_0)$,
\begin{equation*}
\begin{aligned}
& \vu \in L^p(0,T_0;\bm{W}^{2,p}(\Omega_F(\cdot))) \cap W^{1,p}(0,T_0;\bm{L}^p(\Omega_F(\cdot))),\\
& \pi \in L^p(0,T_0;W^{1,p}(\Omega_F(\cdot))), \bm{l} \in W^{1,p}(0,T_0;\R^3), \bm{\omega} \in W^{1,p}(0,T_0;\R^3).
\end{aligned}
\end{equation*}

Moreover, we can choose $T_0$ such that one of the following alternatives holds true:\\
(a) $T_0 = \infty$;\\
(b) either the function $t\mapsto \|\vu(t)\|_{\bm{W}^{2-2/p,p}(\Omega_F(t))}$ is not bounded in $[0,T_0)$ or,
\\
\indent \ $\lim_{t\to T_0} \mathrm{dist} (\partial\Omega, \Omega_S(t))=0$.
\end{proposition}

\begin{proof}
Theorem \ref{T10} gives a unique strong solution $(\hat{\vu}, \hat{\pi}, \hat{\bm{l}}, \hat{\bm{\omega}}) $ to problem (\ref{40})-(\ref{45}). The solution to the original problem (\ref{GFSI}) can be obtained by adding the reference solution $(\vu^*, \pi^*, \bm{l}^*, \bm{\omega}^*) $ and doing the backward coordinate transform, as in the proof of Theorem \ref{T_N}.

That one of the alternatives (a) or (b) holds true, can be proved in the standard way, see for example \cite[Section 3.3]{cumsille}.
\hfill
\end{proof}

\subsection{Non-linear slip condition}

In this final subsection, we discuss a more generalized boundary condition where the velocity of the fluid flow satisfies a wall-law:
\begin{equation*}
[\mathbb{T}(\vu,\pi)\vn]_{\vt} + \alpha |\vu| \vu_{\vt} =\bm{0}.
\end{equation*}
Consider the system (\ref{GFSI}) with the above boundary condition, namely
\begin{equation}
\label{GGFSI}
\left\{
\begin{aligned}
\partial_t \vu + (\vu \cdot \nabla_x) \vu &= \div \ \mathbb{T} (\vu,\pi) \ &&\text{ in } \ \Omega_F(t)\times (0,T),\\
\div \ \vu &= 0 \ &&\text{ in } \ \Omega_F(t)\times (0,T),\\
\vu&=\bm{0} \ &&\text{ on } \ \partial \Omega \times (0,T),\\
\vu\cdot \vn &=\vu_{S}\cdot \vn \ &&\text{ on } \ \partial \Omega_S(t) \times (0,T),\\
\left[\mathbb{T} (\vu,\pi)\vn\right]_{\vt}+\alpha|\vu| \vu_{\vt}&=\alpha |\vu_S| \vu_{S\vt} \ &&\text{ on } \ \partial \Omega_S(t) \times (0,T),\\
m \bm{l}^{'}(t) &= - \int\displaylimits_{\partial \Omega_S(t)}{\mathbb{T} (\vu,\pi)\vn}, \ && \ t\in(0,T),\\
(J \bm{\omega})'(t) &= - \int\displaylimits_{\partial \Omega_S(t)}{(\bm{x}-\bm{h}(t))\times \mathbb{T} (\vu,\pi)\vn},  && \ t\in(0,T),\\
\vu(0) &= \vu_0 \ &&\text{ in } \ \Omega_F(0),\\
\bm{l}( 0) = \bm{l}_0, &\quad \bm{\omega}(0) = \bm{\omega}_0 .
\end{aligned}
\right.
\end{equation}
To prove the well-posedness of (\ref{GGFSI}), the idea is first to linearize the boundary condition as
\begin{equation*}
[\mathbb{T}(\vu,\pi)\vn]_{\vt} + \alpha\vu_{\vt} = \alpha (1-|\tilde{\vu}|) \tilde{\vu}_{\vt}
\end{equation*}
which falls under the non-Newtonian case, for given $\tilde{\vu}$.
Then one may show this is a contraction map which finally establishes the existence result Theorem \ref{T_nonlin_slip} with the help of the Banach fixed point theorem. To prove the contraction, it exactly follows from Theorem \ref{T10} and the fact that the boundary condition can be written as,
\begin{equation*}
\alpha (1-|\tilde{\vu}^1|) \tilde{\vu}^1_{\vt} - \alpha (1-|\tilde{\vu}^2|) \tilde{\vu}^2_{\vt} = \alpha (\tilde{\vu}^1 - \tilde{\vu}^2)_{\vt} - \alpha |\tilde{\vu}^1| (\tilde{\vu}^1 - \tilde{\vu}^2)_{\vt} -\alpha (|\tilde{\vu}^1| - |\tilde{\vu}^2|)\tilde{\vu}^2_{\vt}.
\end{equation*}

\section{Appendix: change of variables}
\label{appendix}
\setcounter{equation}{0}
In this Section we summarize main facts about the change of variables used to transform the problem to the fixed reference domain. Let us first assume that
\begin{equation}
\label{14}
\|\bm{h}\|_{L^\infty(0,\infty;\R^3)} + \|Q-I_3\|_{L^\infty(0,\infty;\R^3)} \text{diam}(\Omega_S(0))\le \frac{\beta}{2}.
\end{equation}
This implies $\text{dist}(\Omega_S(t),\partial\Omega)\ge \beta/2$ for all $t\in [0,\infty)$. For all $\mu>0$, we denote,
$$
\Omega_\mu = \{\bm{x}\in \Omega: \text{dist}(\bm{x},\partial\Omega)>\mu\}.
$$
Now we consider a a cut-off function $\psi\in C^\infty(\R^3,\R)$ with compact support contained in $\Omega_{\beta/8}$ and equal to $1$ in $\overline{\Omega}_{\beta/4}$.
Let us also introduce the functions $\bm{w}:\R^3\times [0,T] \to \R^3$ as
\begin{equation*}
\bm{w}(\bm{x},t) = \bm{l}(t)\times (\bm{x}-\bm{h}(t))+\frac{|\bm{x}-\bm{h}(t)|^2}{2} \bm{\omega}(t)
\end{equation*}
and $\Lambda: \R^3\times [0,T] \to \R^3$ defined as
\begin{equation}
\label{23}
\begin{aligned}
\Lambda(\bm{x},t) = \psi(\bm{x}) \left( \bm{l}(t)+\bm{\omega}(t)\times (\bm{x} - \bm{h}(t))\right) +
\begin{pmatrix}
\frac{\partial \psi}{\partial x_2}(\bm{x})w_3(\bm{x},t) - \frac{\partial \psi}{\partial x_3}(\bm{x})w_2(\bm{x},t)\\
\frac{\partial \psi}{\partial x_3}(\bm{x})w_1(\bm{x},t) - \frac{\partial \psi}{\partial x_1}(\bm{x})w_3(\bm{x},t)\\
\frac{\partial \psi}{\partial x_1}(\bm{x})w_2(\bm{x},t) - \frac{\partial \psi}{\partial x_2}(\bm{x})w_1(\bm{x},t)
\end{pmatrix}.
\end{aligned}
\end{equation}
With these definitions, $\Lambda$ satisfies the following lemma (cf. \cite[Lemma 2.1]{cumsille}):
\begin{lemma}
\label{lem1}
	Let $\bm{w}$ and $\Lambda$ be defined as above. Then, we have\\
	$(1) \ \Lambda = 0$ outside $\Omega_{\beta/8}$.\\
	$(2) \ \div \ \Lambda = 0$ in $\R^3\times [0,T]$.\\
	$(3) \ \Lambda (\bm{x},t) =\bm{l}(t)+\bm{\omega}(t)\times (\bm{x} - \bm{h}(t))$ for all $\bm{x}\in \Omega_S(t)$ and $t\in [0,T]$.\\
	$(4) \ \Lambda \in C(\R^3\times [0,T],\R^3)$. $\text{Moreover}$, for all $t\in [0,T], \Lambda(\cdot, t)$ is a $C^\infty$ function and for all $\bm{x}\in \R^3$, $\Lambda(\bm{x},\cdot)\in H^1([0,T],\R^3)$.
\end{lemma}

Next consider $X$ be the flow associated to $\Lambda$, satisfying the differential equation
\begin{equation}
\label{12}
\begin{aligned}
\frac{\partial X}{\partial t}(\bm{y}, t) & = \Lambda(X(\bm{y},t), t), \quad t>0\\
X(\bm{y},0) & = \bm{y} \in \R^3.
\end{aligned}
\end{equation}
We have the following result, proved in \cite[Lemma 2.2]{cumsille}.

\begin{lemma}
	\label{L0}
	For all $\bm{y}\in \R^3$, the initial value problem (\ref{12}) admits a unique solution $X(\bm{y}, \cdot):[0,T]\to \R^3$ which is a $C^1$ function in $[0,T]$. Moreover, we have the following properties,\\
	$(1)$ For all $t\in [0,T]$, the mapping $\bm{y}\mapsto X(\bm{y},t)$ is a $C^\infty$-diffeomorphism from $\R^3$ onto itself and from $\Omega_F(0)$ onto $\Omega_F(t)$.\\
	$(2)$ Denote by $Y(\cdot, t)$ the inverse of $X(\cdot,t)$. Then, for all $\bm{x}\in \R^3$, the mapping $t\mapsto Y(\bm{x},t)$ is a $C^1$ function in $[0,T]$.\\
	$(3)$ For all $\bm{y}\in \R^3$ and for all $t\in [0,T]$, the determinant of the jacobian matrix $J_X$ of $X(\cdot,t)$ is equal to $1$, that is,
	\begin{equation*}
	\mathrm{det} \ J_X(\bm{y},t) = 1.
	\end{equation*}
\end{lemma}
From here onwards, $J_X$ and $J_Y$ denote the jacobian matrix of $X$ and $Y$ respectively, that is,
\begin{equation*}
J_X = \left(\frac{\partial X_i}{\partial y_j} \right)_{ij} \quad \text{ and } J_Y = \left( \frac{\partial Y_i}{\partial x_j}\right)_{ij} .
\end{equation*}
Note that, for each $\bm{y}\in \Omega_S(0)$, the function $X(\bm{y},t) = \bm{h}(t) + Q(t)\bm{y}, t\ge 0$ is the solution of (\ref{12}), which is easy to verify. This implies, on $\overline{\Omega_S(0)}$,
\begin{equation}
\label{43.}
J_X = Q \quad \text{ and consequently, } \quad J_Y = Q^T .
\end{equation}
Similarly, on $\partial\Omega$, $X(\bm{y},t) = \bm{y}, t\ge 0$ which yields $J_X = I_3 = J_Y$.

Let us now define the functions: for $(\bm{y},t)\in \Omega_F(0)\times (0,\infty)$,
\begin{equation}
\label{13}
\begin{cases}
\begin{aligned}
\tilde{\vu}(\bm{y},t) &= Q^{-1}(t) \ \vu(X(\bm{y},t),t),\\
\tilde{\pi}(\bm{y},t) &= \pi(X(\bm{y},t),t),\\
\tilde{\bm{l}}(t) &= Q^{-1}(t) \ \bm{l}(t),\\
\tilde{\bm{\omega}}(t) &= Q^{-1}(t) \ \bm{\omega}(t)\\
\tilde{J} &= Q^{-1}(t) J(t) Q(t)\\
\tilde{\bm{n}}(\bm{y},t) &= Q^{-1}(t)\bm{n}(X(\bm{y},t),t).
\end{aligned}
\end{cases}
\end{equation}
Notice that $\tilde{\bm{n}}$ becomes the outward normal at $\Omega_F(0)$. Also, from (\ref{Q}) and $(\ref{13})_4$, it easily follows that
\begin{equation}
\label{15}
\dot{Q}(t) \bm{a} = Q(t) (\tilde{\bm{\omega}}\times \bm{a}) \quad \forall \bm{a}\in\R^3.
\end{equation}
In these new variables, the time derivative is transformed into
$$
\partial_t u_i = (\dot{Q}\tilde{\vu})_i + (Q\partial_t \tilde{\vu})_i + (\partial_t X\cdot J_Y^T \nabla)(Q\tilde{\vu})_i = \left( Q(\tilde{\omega}\times \tilde{\vu})\right) _i + (Q\partial_t \tilde{\vu})_i + (\partial_t X\cdot J_Y^T \nabla)(Q\tilde{\vu})_i ,
$$
the convection term is transformed into
$$
(\vu\cdot \nabla_x)u_i = \left( (Q\tilde{\vu})\cdot(J_Y^T\nabla_y)\right) (Q\tilde{\vu})_i ,
$$
the diffusion term is transformed into
$$
\Delta_x u_i =\sum_{m,l,j} \frac{\partial (Q\tilde{\vu})_i}{\partial y_l} \frac{\partial Y_m}{\partial x_j} \frac{\partial}{\partial y_m}\left( \frac{\partial Y_l}{\partial x_j}\right) + \sum_{m,l,j}\frac{\partial^2 (Q\tilde{\vu})_i}{\partial y_m \partial y_l} \frac{\partial Y_l}{\partial x_j} \frac{\partial Y_m}{\partial x_j} ,
$$
and the pressure is transformed to,
$$
(\nabla \pi)_i = (J_Y^T \nabla_y\tilde{\pi})_i.
$$
Furthermore, we obtain
\begin{equation*}
\div \ \vu = \nabla_y \tilde{\vu}: (J_Y Q)^T
\end{equation*}
which can also be written as, by Piola's identity (cf. \cite[pp 39]{ciarlet}, \cite[Ch. 8.1.4.b]{evans}),
\begin{equation*}
\nabla_y \tilde{\vu}: (J_Y Q)^T = \div_y\left( (J_Y Q)\tilde{\vu}\right)
\end{equation*}
since $J_Y Q = \mathrm{cof} (Q J_X) = \mathrm{cof}\nabla_y (QX)$  because of $\mathrm{det}J_X =1 $. Concerning the boundary condition, we calculate the symmetric gradient,
\begin{equation*}
(\nabla_x \vu)_{ij} = \partial_j u_i = \sum_{l=1}^3 \frac{ \partial (Q\tilde{\vu})_i}{\partial y_l} \frac{\partial Y_l}{\partial x_j} = \sum_{l,k=1}^3  Q_{ik} \frac{\partial\tilde{\vu}_k}{\partial y_l} \frac{\partial Y_l}{\partial x_j} = (Q \nabla_y \tilde{\vu} J_Y)_{ij}.
\end{equation*}
This shows that at the interface $\partial\Omega_S(0)$, because of (\ref{43.}), $\nabla_x \vu = Q \, \nabla_y \tilde{\vu} \, Q^{T}$ and hence, $(\nabla_x \vu)^T = Q (\nabla_y \tilde{\vu})^T Q^{T}$ which gives,
\begin{equation*}
\DT_x\vu = Q \,\DT_y\tilde{\vu}\, Q^T
\end{equation*}
and consequently,
\begin{equation*}
\sigma(\vu,\pi) = Q\sigma(\tilde{\vu},\tilde{\pi})Q^T .
\end{equation*}
Therefore, the slip boundary condition becomes,
\begin{equation*}
[\sigma(\tilde{\vu},\tilde{\pi})\tilde{\vn}]_{\vt} + \alpha \tilde{\vu}_{\vt} = \alpha \left( \tilde{\bm{l}}+\tilde{\bm{\omega}}\times \bm{y}\right) _{\vt} \quad \text{ on } \ \partial\Omega_S(0)
\end{equation*}
and similarly at $\partial\Omega$. It can be shown as in \cite[Theorem 2.5]{IW} that the fluid part of the original problem (\ref{1}) admits a strong solution $(\vu,\pi)$ if and only if there exists a corresponding solution $(\tilde{\vu},\tilde{\pi}) \in W^{2,1}_{q,p}(Q^\infty_F) \times L^p(0,\infty;W^{1,q}(\Omega_F(0)))$ to the fluid part of the transformed problem (\ref{fixed_FSI}).

Next, we write the equations for rigid body.
From $(\ref{13})_3$, we find that
$$
m\bm{l}^{'}(t) = m (\dot{Q}\tilde{\bm{l}}+ Q\tilde{\bm{l}}^{'})= mQ(\tilde{\bm{\omega}}\times \tilde{\bm{l}}) + mQ \tilde{\bm{l}}^{'}.
$$
Moreover, we have
$$
\int\displaylimits_{\partial \Omega_S(t)}{\sigma(\vu,\pi)\vn} = Q \int\displaylimits_{\partial \Omega_S(0)}{\sigma(\tilde{\vu},\tilde{\pi})\tilde{n}}
$$
and
$$
\int\displaylimits_{\partial \Omega_S(t)}{(\bm{x} - \bm{h}(t))\times \sigma(\vu,\pi)\vn} = Q \int\displaylimits_{\partial \Omega_S(0)}{\bm{y}\times\sigma(\tilde{\vu},\tilde{\pi})\tilde{n}}.
$$
Therefore, the equation of linear momentum becomes,
$$
m \tilde{\bm{l}}' +m \tilde{\bm{\omega}} \times \tilde{\bm{l}}= - \int\displaylimits_{\partial \Omega_S(0)}{\sigma (\tilde{\vu},\tilde{\pi})\tilde{\vn}}.
$$
Similarly, using the following identity, for any special orthogonal matrix $M\in SO(3)$,
$$
Ma \times Mb = M(a\times b) \quad \forall a,b \in \R^3,
$$
the equation of angular momentum becomes,
$$
\tilde{J} \tilde{\bm{\omega}}^{'}(t) - \tilde{J} \tilde{\bm{\omega}}\times \tilde{\bm{\omega}} = - \int\displaylimits_{\partial \Omega_S(0)}{\bm{y}\times \sigma (\tilde{\vu},\tilde{\pi})\tilde{\vn}}.
$$
Note that, $\tilde{J}$ is independent of time, since
$$
\tilde{J}a\cdot b = \int\displaylimits_{\partial\Omega_S(0)}{\rho_S(\bm{y})(a\times \bm{y})\cdot (b\times \bm{y})\mathrm{d}y} \quad \forall a,b\in \R^3 .
$$
Therefore, on the cylindrical domain $\Omega_F(0)\times (0,T)$, the coupled system for the Newtonian fluid (\ref{1}) transforms into,
\begin{equation}
\label{fixed_FSI}
\begin{cases}
\begin{aligned}
\tilde{\vu}_t - \Delta \tilde{\vu} + \nabla\tilde{\pi} &= \bm{F}_0(\tilde{\vu}, \tilde{\pi}, \tilde{\bm{l}}, \tilde{\bm{\omega}})\ &&\text{ in } \Omega_F(0)\times (0,T),\\
\div \ \tilde{\vu} &= \mathcal{G}(\tilde{\vu}, \tilde{\bm{l}}, \tilde{\bm{\omega}}) = \div \ \bm{H}(\tilde{\vu}, \tilde{\bm{l}}, \tilde{\bm{\omega}}) \ &&\text{ in } \Omega_F(0)\times (0,T),\\
\tilde{\vu}&=\bm{0} \ &&\text{ on } \ \partial \Omega \times (0,T),\\
\tilde{\vu}\cdot\tilde{\vn}=\tilde{\vu}_S\cdot \tilde{\vn}, \quad &\left[\sigma (\tilde{\vu},\tilde{\pi})\tilde{\vn}\right]_{\vt}+\alpha\tilde{\vu}_{\vt}=\alpha\tilde{\vu}_{S\vt} \ &&\text{ on } \ \partial \Omega_S(0) \times (0,T),\\
m \tilde{\bm{l}}' &= - \int\displaylimits_{\partial \Omega_S(0)}{\sigma (\tilde{\vu},\tilde{\pi})\tilde{\vn}} + \bm{F}_1(\tilde{\bm{l}},\tilde{\bm{\omega}}), \ &&\ t\in (0,T),\\
\tilde{J} \tilde{\bm{\omega}}' &= - \int\displaylimits_{\partial \Omega_S(0)}{\bm{y}\times \sigma (\tilde{\vu},\tilde{\pi})\tilde{\vn}} + \bm{F}_2(\tilde{\bm{\omega}}), \ &&\ t\in (0,T),\\
\tilde{\vu}(0) &= \vu_0 \ &&\text{ in } \ \Omega_F(0),\\
\tilde{\bm{l}}( 0) = \bm{l}_0, \quad &\tilde{\bm{\omega}}(0) = \bm{\omega}_0
\end{aligned}
\end{cases}
\end{equation}
where
$$
\tilde{\vu}_S := \tilde{\bm{l}} + \tilde{\bm{\omega}}\times\bm{y} ;
$$
\begin{equation}
\label{36}
\begin{aligned}
(\bm{F}_0)_{i}(\tilde{\vu}, \tilde{\pi}, \tilde{\bm{l}}, \tilde{\bm{\omega}})& := \left( (I_3-Q)\partial_t \tilde{\vu}\right) _i - \left( Q(\tilde{\bm{\omega}}\times\tilde{\vu}) \right) _i -(\partial_t X\cdot J_Y^T \nabla)(Q\tilde{\vu})_i - \left( (Q\tilde{\vu})\cdot(J_Y^T\nabla)\right) (Q\tilde{\vu})_i\\
& + \sum_{m,l,j} \frac{\partial (Q\tilde{\vu})_i}{\partial y_l} \frac{\partial Y_m}{\partial x_j} \frac{\partial}{\partial y_m}\left( \frac{\partial Y_l}{\partial x_j}\right) + \sum_{m,l,j}\frac{\partial^2 (Q\tilde{\vu})_i}{\partial y_m \partial y_l} \frac{\partial Y_l}{\partial x_j} \frac{\partial Y_m}{\partial x_j} -\Delta\tilde{\vu} _i\\
& +\left( (I_3 -J^T_Y)\nabla\tilde{\pi}\right) _i ;
\end{aligned}
\end{equation}
\begin{equation}
\label{37}
\mathcal{G}(\tilde{\vu}, \tilde{\bm{l}}, \tilde{\bm{\omega}}) := \nabla \tilde{\vu}: (I_3 - (J_Y Q)^T) = \div \ \bm{H} \quad \text{ with } \quad \bm{H}(\tilde{\vu}, \tilde{\bm{h}}, \tilde{\bm{\omega}}) := (I_3-J_YQ) \tilde{\bm{u}} ;
\end{equation}
\begin{equation}
\label{38}
\bm{F}_1(\tilde{\bm{l}},\tilde{\bm{\omega}}) := -m \tilde{\bm{\omega}} \times \tilde{\bm{l}} ;
\end{equation}
\begin{equation}
\label{24}
\bm{F}_2(\tilde{\bm{\omega}}) := \tilde{J} \tilde{\bm{\omega}}\times \tilde{\bm{\omega}}.
\end{equation}
Note that a solution $(\tilde{\vu},\tilde{\pi},\tilde{\bm{l}},\tilde{\bm{\omega}})$ to (\ref{fixed_FSI}) yields a solution $(\vu,\pi,\bm{l},\bm{\omega})$ to (\ref{1}) by (\ref{13}).

\section*{Acknowledgements}
The works of Amrita Ghosh and \v S\' arka  Ne\v casov\'a were supported by the Czech Science Foundation grant GA19-04243S in the framework of RVO 67985840 and the work of Boris Muha was supported by the Croatian Science Foundation (Hrvatska zaklada za znanost) grant IP-2018-01-3706.

\bibliography{L_p__Strong_solution}

\begin{thebibliography}{10}

\bibitem{AACG}
P.~Acevedo, C.~Amrouche, C.~Conca, and A.~Ghosh.
\newblock Stokes and {N}avier-{S}tokes equations with {N}avier boundary
  condition.

\bibitem{HindSarkaBoris}
H.~Al~Baba, N.~V. Chemetov, \v{S}. Ne\v{c}asov\'{a}, and B.~Muha.
\newblock Strong solutions in {$L^2$} framework for fluid-rigid body
  interaction problem. {M}ixed case.
\newblock {\em Topol. Methods Nonlinear Anal.}, 52(1):337--350, 2018.

\bibitem{Amann1}
H.~Amann.
\newblock {\em Linear and quasilinear parabolic problems. {V}ol. {I}},
  volume~89 of {\em Monographs in Mathematics}.
\newblock Birkh\"{a}user Boston, Inc., Boston, MA, 1995.
\newblock Abstract linear theory.

\bibitem{AEG}
C.~Amrouche, M.~Escobedo, and A.~Ghosh.
\newblock Semigroup theory for the {S}tokes operator with {N}avier boundary
  condition on $l^p$-spaces.

\bibitem{BBN}
H.~Bellout, Bloom, F., and J.~Ne\v{c}as.
\newblock Young measure-valued solutions for non-{N}ewtonian incompressible
  fluids.
\newblock {\em Comm. Partial Differential Equations}, 19(11-12):1763--1803,
  1994.

\bibitem{BB}
H.~Bellout and F.~Bloom.
\newblock {\em Incompressible bipolar and non-Newtonian viscous fluid flow}.
\newblock Advances in Mathematical Fluid Mechanics. Birkh\"{a}user, Springer,
  Cham, 2014.

\bibitem{ben}
A.~Bensoussan, G.~Da~Prato, M.~C. Delfour, and S.K. Mitter.
\newblock {\em Representation and control of infinite-dimensional systems.
  {V}ol. 1}.
\newblock Systems \& Control: Foundations \& Applications. Birkh\"{a}user
  Boston, Inc., Boston, MA, 1992.

\bibitem{BP}
D.~Bothe and J.~Pr\"{u}ss.
\newblock {$L_P$}-theory for a class of non-{N}ewtonian fluids.
\newblock {\em SIAM J. Math. Anal.}, 39(2):379--421, 2007.

\bibitem{bourgain}
J.~Bourgain.
\newblock Some remarks on {B}anach spaces in which martingale difference
  sequences are unconditional.
\newblock {\em Ark. Mat.}, 21(2):163--168, 1983.

\bibitem{Bravin}
M.~Bravin.
\newblock On the weak uniqueness of "viscous incompressible fluid + rigid body"
  system with {N}avier-slip-with-friction conditions in a 2d bounded domain.
\newblock {\em J. Math. Fluid Mech}, 21(2), 2019.

\bibitem{burkholder}
D.~L. Burkholder.
\newblock A geometric condition that implies the existence of certain singular
  integrals of {B}anach-space-valued functions.
\newblock In {\em Conference on harmonic analysis in honor of {A}ntoni
  {Z}ygmund, {V}ol. {I}, {II} ({C}hicago, {I}ll., 1981)}, Wadsworth Math. Ser.,
  pages 270--286. Wadsworth, Belmont, CA, 1983.

\bibitem{ChemetovSarka}
N.~V. Chemetov and \v{S}. Ne\v{c}asov\'{a}.
\newblock The motion of the rigid body in the viscous fluid including
  collisions. {G}lobal solvability result.
\newblock {\em Nonlinear Anal. Real World Appl.}, 34:416--445, 2017.

\bibitem{ChemetpvSarkaBoris}
N.~V. Chemetov, \v{S}. Ne\v{c}asov\'{a}, and B.~Muha.
\newblock Weak-strong uniqueness for fluid-rigid body interaction problem with
  slip boundary condition.
\newblock {\em J. Math. Phys.}, 60(1):011505, 13, 2019.

\bibitem{ciarlet}
P.~G. Ciarlet.
\newblock {\em Mathematical elasticity. {V}ol. {I}}, volume~20 of {\em Studies
  in Mathematics and its Applications}.
\newblock North-Holland Publishing Co., Amsterdam, 1988.
\newblock Three-dimensional elasticity.

\bibitem{CP}
P.~Cl\'{e}ment and J.~Pr\"{u}ss.
\newblock {\em An operator-valued transference principle and maximal regularity
  on vector-valued {$L_p$}-spaces}, volume 215 of {\em Lecture Notes in Pure
  and Appl. Math.}
\newblock Dekker, New York, 2001.

\bibitem{cumsille}
P.~Cumsille and T.~Takahashi.
\newblock Wellposedness for the system modelling the motion of a rigid body of
  arbitrary form in an incompressible viscous fluid.
\newblock {\em Czechoslovak Math. J.}, 58(133)(4):961--992, 2008.

\bibitem{DG}
G.~Da~Prato and P.~Grisvard.
\newblock Sommes d'op\'{e}rateurs lin\'{e}aires et \'{e}quations
  diff\'{e}rentielles op\'{e}rationnelles.
\newblock {\em J. Math. Pures Appl. (9)}, 54(3):305--387, 1975.

\bibitem{DHP}
R.~Denk, M.~Hieber, and J.~Pr\"uss.
\newblock {R}-boundedness, {F}ourier multipliers and problems of elliptic and
  parabolic type.
\newblock {\em Mem. Amer. Math. Soc.}, 166(788):viii+114, 2003.

\bibitem{valdi}
E.~Di~Nezza, G.~Palatucci, and E.~Valdinoci.
\newblock Hitchhiker's guide to the fractional {S}obolev spaces.
\newblock {\em Bull. Sci. Math.}, 136(5):521--573, 2012.

\bibitem{DGH}
E.~Dintelmann, M.~Geissert, and M.~Hieber.
\newblock Strong {$L^p$}-solutions to the {N}avier-{S}tokes flow past moving
  obstacles: the case of several obstacles and time dependent velocity.
\newblock {\em Trans. Amer. Math. Soc.}, 361(2):653--669, 2009.

\bibitem{dore}
G.~Dore.
\newblock {$L^p$} regularity for abstract differential equations.
\newblock In {\em Functional analysis and related topics, 1991 ({K}yoto)},
  volume 1540 of {\em Lecture Notes in Math.}, pages 25--38. Springer, Berlin,
  1993.

\bibitem{dore_venni}
G.~Dore and A.~Venni.
\newblock On the closedness of the sum of two closed operators.
\newblock {\em Math. Z.}, 196(2):189--201, 1987.

\bibitem{dussan}
E.~B. Dussan~V.
\newblock The moving contact line: the slip boundary condition.
\newblock {\em Journal of Fluid Mechanics}, 77(4):665--684, 1976.

\bibitem{engel}
K.~Engel and R.~Nagel.
\newblock {\em One-parameter semigroups for linear evolution equations}, volume
  194 of {\em Graduate Texts in Mathematics}.
\newblock Springer-Verlag, New York, 2000.
\newblock With contributions by S. Brendle, M. Campiti, T. Hahn, G. Metafune,
  G. Nickel, D. Pallara, C. Perazzoli, A. Rhandi, S. Romanelli and R.
  Schnaubelt.

\bibitem{evans}
L.~C. Evans.
\newblock {\em Partial differential equations}, volume~19 of {\em Graduate
  Studies in Mathematics}.
\newblock American Mathematical Society, Providence, RI, 1998.

\bibitem{FHN}
E.~Feireisl, M.~Hillairet, and \v{S}. Ne\v{c}asov\'{a}.
\newblock On the motion of several rigid bodies in an incompressible
  non-{N}ewtonian fluid.
\newblock {\em Nonlinearity}, 21(6):1349--1366, 2008.

\bibitem{galdi}
G.~P. Galdi.
\newblock On the motion of a rigid body in a viscous liquid: a mathematical
  analysis with applications.
\newblock In {\em Handbook of mathematical fluid dynamics, {V}ol. {I}}, pages
  653--791. North-Holland, Amsterdam, 2002.

\bibitem{GGH}
M.~Geissert, K.~G\"{o}tze, and M.~Hieber.
\newblock {$L^p$}-theory for strong solutions to fluid-rigid body interaction
  in {N}ewtonian and generalized {N}ewtonian fluids.
\newblock {\em Trans. Amer. Math. Soc.}, 365(3):1393--1439, 2013.

\bibitem{GVHill}
D.~G\'{e}rard-Varet and M.~Hillairet.
\newblock Existence of weak solutions up to collision for viscous fluid-solid
  systems with slip.
\newblock {\em Comm. Pure Appl. Math.}, 67(12):2022--2075, 2014.

\bibitem{VHW}
D.~G\'{e}rard-Varet, M.~Hillairet, and C.~Wang.
\newblock The influence of boundary conditions on the contact problem in a 3d
  navier-stokes flow.
\newblock {\em J. Math. Pures Appl.}, 9(1):1--38, 2015.

\bibitem{Hillairet07}
M.~Hillairet.
\newblock Lack of collision between solid bodies in a 2{D} incompressible
  viscous flow.
\newblock {\em Comm. Partial Differential Equations}, 32(7-9):1345--1371, 2007.

\bibitem{IW}
A.~Inoue and M.~Wakimoto.
\newblock On existence of solutions of the {N}avier-{S}tokes equation in a time
  dependent domain.
\newblock {\em J. Fac. Sci. Univ. Tokyo Sect. IA Math.}, 24(2):303--319, 1977.

\bibitem{KMN}
M.~Kalousek, V.~M\v{a}cha, and \v{S}. Ne\v{c}asov\'{a}.
\newblock Local-in-time existence of strong solutions to a class of
  compressible non-{N}ewtonian {N}avier-{S}tokes equations.
\newblock {\em arXiv:2012.01795}.

\bibitem{kato}
T.~Kato and H.~Fujita.
\newblock On the nonstationary {N}avier-{S}tokes system.
\newblock {\em Rend. Sem. Mat. Univ. Padova}, 32:243--260, 1962.

\bibitem{KB}
J.~Koplik and J.~R. Banavar.
\newblock Corner flow in the sliding plate problem.
\newblock {\em Physics of Fluids}, 7(12):3118--3125, 1995.

\bibitem{kunstmann_weis}
P.~C. Kunstmann and L.~Weis.
\newblock Perturbation theorems for maximal {$L_p$}-regularity.
\newblock {\em Ann. Scuola Norm. Sup. Pisa Cl. Sci. (4)}, 30(2):415--435, 2001.

\bibitem{KW}
P.~C. Kunstmann and L.~Weis.
\newblock Maximal {$L_p$}-regularity for parabolic equations, {F}ourier
  multiplier theorems and {$H^\infty$}-functional calculus.
\newblock In {\em Functional analytic methods for evolution equations}, volume
  1855 of {\em Lecture Notes in Math.}, pages 65--311. Springer, Berlin, 2004.

\bibitem{weis1}
P.C. Kunstmann and L.~Weis.
\newblock Perturbation theorems for maximal {$L_p$}-regularity.
\newblock {\em Ann. Scuola Norm. Sup. Pisa Cl. Sci. (4)}, 30(2):415--435, 2001.

\bibitem{LUS}
O.~A. Lady\v{z}enskaja, V.~A. Solonnikov, and N.~N. Uralceva.
\newblock {\em Linear and quasilinear equations of parabolic type}.
\newblock Translated from the Russian by S. Smith. Translations of Mathematical
  Monographs, Vol. 23. American Mathematical Society, Providence, R.I., 1968.

\bibitem{Lewandowski}
R.~Lewandowski, B.~Pinier, E.~M{\'e}min, and P.~Chandramouli.
\newblock Testing a one-closure equation turbulence model in neutral boundary
  layers.
\newblock 2018.

\bibitem{MT17}
D.~Maity and M.~Tucsnak.
\newblock A maximal regularity approach to the analysis of some particulate
  flows.
\newblock In {\em Particles in flows}, Adv. Math. Fluid Mech., pages 1--75.
  Birkh\"{a}user/Springer, Cham, 2017.

\bibitem{MT}
D.~Maity and M.~Tucsnak.
\newblock {$L^p$}-{$L^q$} maximal regularity for some operators associated with
  linearized incompressible fluid-rigid body problems.
\newblock In {\em Mathematical analysis in fluid mechanics---selected recent
  results}, volume 710 of {\em Contemp. Math.}, pages 175--201. Amer. Math.
  Soc., Providence, RI, 2018.

\bibitem{Necas}
J.~M\'{a}lek, J.~Ne\v{c}as, M.~Rokyta, and M.~R{u}\v{z}i\v{c}ka.
\newblock {\em Weak and measure-valued solutions to evolutionary PDEs.}
\newblock Applied Mathematics and Mathematical Computation. Chapman and Hall,
  London, 1996.

\bibitem{MH}
M.~T. Matthews and J.~M. Hill.
\newblock Newtonian flow with nonlinear navier boundary condition.
\newblock {\em Acta Mechanica}, 191(3):195--217, Jul 2007.

\bibitem{max}
J.~C. Maxwell.
\newblock On stresses in rarefied gases arising from inequalities of
  temperature. [abstract].
\newblock {\em Proceedings of the Royal Society of London}, 27:304--308, 1878.

\bibitem{muha2019uniqueness}
B.~Muha, \v{S}. Ne{\v{c}}asov{\'a}, and A.~Rado{\v{s}}evi{\'c}.
\newblock A uniqueness result for 3d incompressible fluid-rigid body
  interaction problem.
\newblock {\em arXiv preprint arXiv:1904.05102}, 2019.

\bibitem{SunnyBoris}
B.~Muha and S.~\v{C}ani\'{c}.
\newblock Existence of a weak solution to a fluid-elastic structure interaction
  problem with the {N}avier slip boundary condition.
\newblock {\em J. Differential Equations}, 260(12):8550--8589, 2016.

\bibitem{Navier}
C.~L. M.~H. Navier.
\newblock M\'{e}moire sur les lois du mouvement des fluides.
\newblock {\em M\'{e}m. Acad. Sci. Inst. de France (2)}, pages 389--440, 1823.

\bibitem{OT}
B.~Obando and T.~Takahashi.
\newblock Existence of weak solutions for a {B}ingham fluid-rigid body system.
\newblock {\em Ann. Inst. H. Poincar{\' e} Anal. Non Lin{\' e}aire},
  36(5):1281--1309, 2019.

\bibitem{OG}
T.~Ogawa and S.~Shimizu.
\newblock Maximal {$L^1$}-regularity for parabolic boundary value problems with
  inhomogeneous data in the half-space.
\newblock {\em Proc. Japan Acad. Ser. A Math. Sci.}, 96(7):57--62, 2020.

\bibitem{P}
J.~Pr\"{u}ss.
\newblock Maximal regularity for evolution equations in {$L_p$}-spaces.
\newblock {\em Conf. Semin. Mat. Univ. Bari}, (285):1--39 (2003), 2002.

\bibitem{pruss}
J.~Pr\"{u}ss and G.~Simonett.
\newblock {\em Moving interfaces and quasilinear parabolic evolution
  equations}, volume 105 of {\em Monographs in Mathematics}.
\newblock Birkh\"{a}user/Springer, [Cham], 2016.

\bibitem{raymond1}
J.~Raymond.
\newblock Feedback stabilization of a fluid-structure model.
\newblock {\em SIAM J. Control Optim.}, 48(8):5398--5443, 2010.

\bibitem{SHS}
Y.~Shibata and S.~Shimizu.
\newblock On the maximal {$L_p$}-{$L_q$} regularity of the {S}tokes problem
  with first order boundary condition; model problems.
\newblock {\em J. Math. Soc. Japan}, 64(2):561--626, 2012.

\bibitem{Starovoitov04}
V.~N. Starovoitov.
\newblock Behavior of a rigid body in an incompressible viscous fluid near a
  boundary.
\newblock In {\em Free boundary problems ({T}rento, 2002)}, volume 147 of {\em
  Internat. Ser. Numer. Math.}, pages 313--327. Birkh\"{a}user, Basel, 2004.

\bibitem{takeo}
T.~Takahashi.
\newblock Analysis of strong solutions for the equations modeling the motion of
  a rigid-fluid system in a bounded domain.
\newblock {\em Adv. Differential Equations}, 8(12):1499--1532, 2003.

\bibitem{TT}
T.~Takahashi and M.~Tucsnak.
\newblock Global strong solutions for the two-dimensional motion of an infinite
  cylinder in a viscous fluid.
\newblock {\em J. Math. Fluid Mech.}, 6(1):53--77, 2004.

\bibitem{TT1}
P.~A. Thompson and S.~M. Troian.
\newblock A general boundary condition for liquid flow at solid surfaces.
\newblock {\em Nature}, 389:360--362, 1997.

\bibitem{triebel}
H.~Triebel.
\newblock {\em Theory of function spaces. {II}}, volume~84 of {\em Monographs
  in Mathematics}.
\newblock Birkh\"{a}user Verlag, Basel, 1992.

\bibitem{Wang}
C.~Wang.
\newblock Strong solutions for the fluid-solid systems in a 2-{D} domain.
\newblock {\em Asymptot. Anal.}, 89(3-4):263--306, 2014.

\bibitem{weis}
L.~Weis.
\newblock Operator-valued {F}ourier multiplier theorems and maximal
  {$L_p$}-regularity.
\newblock {\em Math. Ann.}, 319(4):735--758, 2001.

\bibitem{Aneta}
A.~Wr\'{o}blewska-Kami\'{n}ska.
\newblock Existence result for the motion of several rigid bodies in an
  incompressible non-{N}ewtonian fluid with growth conditions in orlicz spaces.
\newblock {\em Nonlinearity}, 27(4), 2014.

\end{thebibliography}
\bibliographystyle{plain}
\end{document}